\documentclass[12pt,leqno,a4paper]{amsart}
\usepackage{amssymb,enumerate,hyperref}
\usepackage{xcolor}

\newcommand{\FF}{{\mathbb{F}}}
\newcommand{\ZZ}{{\mathbb{Z}}}

\newcommand{\fA}{{\mathfrak{A}}}
\newcommand{\fE}{\mathfrak{E}}
\newcommand{\fS}{{\mathfrak{S}}}

\newcommand{\bC}{{\mathbf{C}}}
\newcommand{\bE}{{\mathbf{E}}}
\newcommand{\bF}{{\mathbf{F}}}
\newcommand{\bG}{{\mathbf{G}}}
\newcommand{\tbG}{{\widetilde{\bG}}}
\newcommand{\bH}{{\mathbf{H}}}
\newcommand{\bK}{{\mathbf{K}}}
\newcommand{\bL}{{\mathbf{L}}}
\newcommand{\bM}{{\mathbf{M}}}
\newcommand{\bN}{{\mathbf{N}}}
\newcommand{\bO}{{\mathbf{O}}}
\newcommand{\bP}{{\mathbf{P}}}
\newcommand{\bS}{{\mathbf{S}}}
\newcommand{\bT}{{\mathbf{T}}}
\newcommand{\bZ}{{\mathbf{Z}}}

\newcommand{\cD}{{\mathcal{D}}}
\newcommand{\cE}{{\mathcal{E}}}

\newcommand\type[1]{\operatorname{#1}}

\newcommand\tA{\operatorname{A}}
\newcommand\tB{\operatorname{B}}

\newcommand\tD{\operatorname{D}}
\newcommand\tE{\operatorname{E}}
\newcommand\tF{\operatorname{F}}

\newcommand{\Aut}{{\operatorname{Aut}}}
\newcommand{\cd}{{\operatorname{cd}}}
\newcommand{\diag}{{\operatorname{diag}}}
\newcommand{\Hom}{{\operatorname{Hom}}}

\newcommand{\Inn}{{\operatorname{Inn}}}
\newcommand{\Inndiag}{{\operatorname{Inndiag}}}
\newcommand{\Irr}{{\operatorname{Irr}}}
\newcommand{\Out}{{\operatorname{Out}}}
\newcommand{\rnk}{\operatorname{rnk}}
\newcommand{\SC}{{\operatorname{sc}}}
\newcommand{\Stab}{\operatorname{Stab}}
\newcommand{\Syl}{{\operatorname{Syl}}}
\newcommand{\Sym}{{\operatorname{Sym}}}

\newcommand{\GL}{{\operatorname{GL}}}

\newcommand{\PGL}{{\operatorname{PGL}}}
\newcommand{\PSL}{{\operatorname{L}}}
\newcommand{\SL}{{\operatorname{SL}}}
\newcommand{\SU}{{\operatorname{SU}}}
\newcommand{\PSU}{{\operatorname{U}}}
\newcommand{\Sp}{{\operatorname{Sp}}}
\newcommand{\SO}{{\operatorname{SO}}}
\newcommand{\GO}{{\operatorname{GO}}}
\newcommand{\PSO}{{\operatorname{O}}}

\newcommand{\hc}{R}
\newcommand\RTG{{R_\bT^\bG}}
\newcommand\RLG{{R_\bL^\bG}}
\newcommand\RHG{{R_\bH^\bG}}
\newcommand\RLH{{R_\bL^\bH}}

\newcommand\wt[1]{\widetilde{#1}}

\newcommand{\tw}[1]{{}^{#1}\!}

\let\al=\alpha
\let\eps=\epsilon
\let\vhi=\varphi
\let\la=\lambda

\def\cent#1#2{{\bf C}_{#1}(#2)}
\def\iitem#1{\goodbreak\par\noindent{\bf #1}}
\def\irr#1{{\rm Irr}(#1)}
\let\nor=\unlhd
\def\norm#1#2{{\bf N}_{#1}(#2)}
\def\oh#1#2{{\bf O}_{#1}(#2)}
\def\Oh#1#2{{\bf O}^{#1}(#2)}
\let\sbs=\subseteq
\def\syl#1#2{{\rm Syl}_#1(#2)}
\def\zent#1{{{\bf Z}(#1)}}
\renewcommand{\emptyset}{\varnothing}

\newtheorem{thm}{Theorem}[section]
\newtheorem{lem}[thm]{Lemma}

\newtheorem{cor}[thm]{Corollary}
\newtheorem{prop}[thm]{Proposition}

\newtheorem*{thmA}{THEOREM A}
\newtheorem*{thmB}{THEOREM B}
\newtheorem*{thmC}{THEOREM C}
\newtheorem*{corB}{COROLLARY}

\theoremstyle{definition}

\theoremstyle{remark}
\newtheorem{rem}[thm]{Remark}
\newtheorem{hyp}[thm]{Hypothesis}

\numberwithin{equation}{thm}

\makeatletter
\def\namedlabel#1#2{\begingroup
    #2%
    \def\@currentlabel{#2}%
    \phantomsection\label{#1}\endgroup
}
\begin{document}

%%%%%%%%%%%%%%%%%%%%%%%%%%%%%%%%%%%%%%%%%%%%%%%%%%%%%%%%%%%%%%%%%%%%%%%%%
\title{Brauer's Height Zero Conjecture}
%%%%%%%%%%%%%%%%%%%%%%%%%%%%%%%%%%%%%%%%%%%%%%%%%%%%%%%%%%%%%%%%%%%%%%%%%

\date{\today}

\author[G. Malle]{Gunter Malle}
\address[G. Malle]{FB Mathematik, TU Kaiserslautern, Postfach 3049,
  67653 Kaisers\-lautern, Germany.}
\email{malle@mathematik.uni-kl.de}

\author[G. Navarro]{Gabriel Navarro}
\address[G. Navarro]{Departament de Matem\`atiques, Universitat de Val\`encia,
  46100 Burjassot, Val\`encia, Spain}
\email{gabriel@uv.es}

\author[A. A. Schaeffer Fry]{A. A. Schaeffer Fry}
\address[A. A. Schaeffer Fry]{Since Sept 2023: Department of Mathematics,
 University of Denver, Denver, CO 80210, USA (Formerly: Dept. Mathematics and
 Statistics, MSU Denver, Denver, CO 80217)}
\email{mandi.schaefferfry@du.edu}

\author[P. H. Tiep]{Pham Huu Tiep}
\address[P. H. Tiep]{Department of Mathematics, Rutgers University, Piscataway,
  NJ 08854, USA}
\email{tiep@math.rutgers.edu}

\thanks{The first author gratefully acknowledges financial support by SFB TRR
195 -- Project-ID 286237555. The research of the second author is
supported by Ministerio de Ciencia e Innovaci\'on PID2019-103854GB-I00. The
third author is grateful for support from the NSF, under grant DMS-2100912.
The fourth author gratefully acknowledges the support of the NSF (grants
DMS-1840702 and DMS-2200850), the Simons Foundation, and the Joshua Barlaz
Chair in Mathematics. Part of this work was
done while the authors participated in the programme ``\emph{Groups,
Representations, and Applications: New Perspectives}'' at the Isaac Newton
Institute for Mathematical Sciences in 2022. This work was supported
by EPSRC grant number EP/R014604/1}

\thanks{We thank Zhicheng Feng, Ruwen Hollenbach, Radha Kessar, Markus
Linckelmann, Alexander Moret\'o, Noelia Rizo, Britta Sp\"ath and Lizhong Wang
for helpful
discussions on aspects of this project. We are grateful to the referee for
careful reading and insightful comments that helped improve the paper.}



\begin{abstract}
We complete the proof of Brauer's Height Zero Conjecture from 1955 by
establishing the open implication for all odd primes.
\end{abstract}

\maketitle

%\pagestyle{myheadings}
%%\markboth{}{}
%\markboth{for personal use only}{preliminary}

\tableofcontents

%%%%%%%%%%%%%%%%%%%%%%%%%%%%%%%%%%%%%%%%%%%%%%%%%%%%%%%%%%%%%%%%%%%%%%%%%
\section{Introduction}   \label{sec:intro}

Brauer's Height Zero Conjecture (BHZ), formulated in 1955 \cite{bhz}, has been
one of the most fundamental and challenging problems in the representation
theory of finite groups. Deeply influencing the research in the field, it is
also a source of many developments in the theory. If $p$ is a prime
and $B$ is a Brauer $p$-block with defect group $D$ of a finite group~$G$,
R.~Brauer proved that $|G:D|_p$ is the largest power of $p$ dividing the degrees
of all the irreducible complex characters in $B$. (In this paper, $n_p$ denotes
the largest power of $p$ dividing the integer $n$.) Hence, if $\chi\in\irr B$,
the set of irreducible complex characters in $B$, then
$\chi(1)_p=|G:D|_p\, p^{h_\chi}$ for some non-negative integer $h_\chi$ called
the \emph{height of $\chi$}. The conjecture asserts that
$\chi(1)_p=|G:D|_p$ for all $\chi \in \irr B$ if and only if $D$ is abelian.
That is, $h_\chi=0$ for all $\chi \in \irr B$ if and only if $D$ is abelian. 
   \par
The ``if" implication of the Height Zero Conjecture was proven in \cite{KM13},
using the classification of finite simple groups, after decades of
contributions by many authors. The ``only if" implication  was proven for
$p$-solvable groups in \cite{GW}; for $p=2$ and blocks of maximal defect
(that is, when $D$ is a Sylow $2$-subgroup of $G$) in \cite{NT1}; and recently
for principal blocks, for every prime, in \cite{MN21}. Furthermore, building
upon work in~\cite{NT2}, it was shown in \cite{NS} that Brauer's Height Zero
Conjecture is
implied by the inductive Alperin--McKay condition on simple groups (a strong
form of another main conjecture in our field). This has enabled L.~Ruhstorfer to
recently prove the  Height Zero Conjecture for $p=2$ in \cite{Ru}.  However, the
verification of the inductive Alperin--McKay condition on simple groups for odd
primes remains an enormous challenge.
\medskip

In this paper we take a different approach and prove the open direction of
Brauer's Height Zero Conjecture in the case that $p$ is odd.

\begin{thmA}
 Let $G$ be a finite group, let $p$ be an odd prime, and let $B$ be a $p$-block
 of $G$ with defect group $D$. If $\chi(1)_p=|G:D|_p$ for all $\chi\in\irr B$
 then $D$ is abelian.
\end{thmA}

As discussed above, this implies:

\begin{corB}
 Brauer's Height Zero Conjecture holds.
\end{corB}

A key novelty of our approach is a combined use of new results on blocks of
quasi-simple groups as well as on permutation groups, which allows us to
tightly control the structure of a minimal counterexample to BHZ and overcome
certain difficulties in proving extendibility of characters from normal
subgroups that were encountered in previous approaches.

There are, at least, two major obstacles for our approach. The first is to
prove that irreducible characters in $p$-blocks of quasi-simple groups lie in
sufficiently many distinct orbits under the action by their automorphism groups.
We think that the following result has independent interest and that it will be
useful in the resolution of other problems.

\begin{thmB}
 Suppose that $p$ is an odd prime, $S$ is a quasi-simple group, and $b$ is a
 $p$-block of $S$ with non-cyclic defect groups. Then at least one of the
 following statements holds.
 \begin{enumerate}[\rm(1)]
  \item $\Irr(b)$ contains characters from at least three different
   $\Aut(S)$-orbits; or
  \item all characters in $\Irr(b)$ have the same degree.
 \end{enumerate} 
\end{thmB}

In the presence of blocks with cyclic defect groups, or when all the irreducible
characters in $b$ have the same degree (and therefore $b$ is nilpotent in the
sense of Brou\'e--Puig), we will instead use deep results by
Koshitani--Sp\"ath, Brou\'e--Puig and K\"ulshammer--Puig (\cite{KoS,BP,KP}) in
order to prove Theorem~A.
\medskip

The second obstacle to our approach is inherent to Brauer's Height Zero
Conjecture and independent of any road that is followed to prove it. Suppose
that $G$ is a finite group and $\sigma$ is an automorphism of order a power of
$p$ of $G$ that stabilises a $p$-block $B$ of $G$, a defect group $D$ of $B$,
and a $p$-block $b_D$ of $\cent GD$ that induces $B$. If $D$ is abelian,
Brauer's Height Zero conjecture (and the inductive Alperin--McKay condition)
implies that all the irreducible characters of
$B$ are fixed by $\sigma$ if and only if $\sigma$ acts trivially on $D$. In
fact, we will need a more sophisticated version of the following result for
quasi-simple groups (see Theorem~\ref{thm:newg}).

\begin{thmC} 
 Suppose that $p$ is an odd prime and that $S$ is a quasi-simple group such
 that $\zent S$ is a cyclic $p'$-group. Let $b$ be a $p$-block of $S$ with
 abelian defect group $D$. Suppose that $\sigma$ is an automorphism of $S$ of
 $p$-power order that fixes all the irreducible characters of~$b$,
 normalises $D$, and stabilises a block $b_D$ of $\cent SD$ that induces $b$.
 Then $\sigma$ acts trivially on $D$.
\end{thmC}

In Section \ref{sec:prelim}, we prove the needed result on permutation groups
and other technical results that we will use in the proof of Theorem A. In
Section~\ref{sec:ThmB}, we prove Theorem~B. Section~\ref{sec:ThmC} is devoted
to the proof of a refined version of Theorem~C and its needed generalisations.
Finally, in Section \ref{sec:finalpf}, Theorem A is proved. As was the case in
the proof of the other direction of BHZ, as well as the
cases of maximal defect for $p=2$ and for principal blocks, our proof
(including the result on permutations groups that we have mentioned as well as
Theorems B and C) relies on the Classification of Finite Simple Groups.

%%%%%%%%%%%%%%%%%%%%%%%%%%%%%%%%%%%%%%%%%%%%%%%%%%%%%%%%%%%%%%%%%%%%%%%%%
\section{Preliminary results}\label{sec:prelim}
In this section, we first prove a consequence of \cite{GLPST} that we will use
in the proof of Theorem~A.

\begin{thm}   \label{orbits}
 Let $p > 2$ be a prime, $n > 1$, and let $G < \Sym(\Omega)=\fS_n$ be a
 subgroup such that $G = \bO^{p'}(G) \neq 1$. Then there is a partition
 $$\Omega = \Delta_1 \sqcup \Delta_2 \sqcup \Delta_3,$$
 with $|\Delta_1|, |\Delta_2| > 0$ and $|\Delta_3| \geq 0$ such that the index
 of $\cap^3_{i=1}\Stab_G(\Delta_i)$ in $G$ is divisible by $p$. 
 Moreover, one can choose this partition to have $\Delta_3=\emptyset$, unless
 $G$ has a simple quotient $S$ such that one of the following holds:
 \begin{enumerate}[\rm(1)]
  \item $S = \fA_{ap^s-1}$ with $1\leq a\leq p-1$, $s\ge1$ and
   $(a,s)\neq (1,1)$; or
  \item $p=3$ and $S = C_3$ or $\SL_3(2)$.
 \end{enumerate}
\end{thm}

\begin{proof}
(i) First we consider the case that $G$ is primitive on $\Omega$. 
In this case, applying \cite[Thm~2]{GLPST} we obtain a partition
$\Omega = \Delta_1 \sqcup \Delta_2$ with
$p\mid[G:\cap^2_{i=1}\Stab_G(\Delta_i)]$,
unless $G$ has a simple quotient $S$ and $(S,p)$ are as in (1), in fact
with $S=\fA_n$ acting on $n=ap^s-1$ points, or we are in (2), in fact with
$G = {\rm {ASL}}_3(2)$ or ${\rm A}\Gamma{\rm L}_1(8)$, acting on the $8$
vectors of $\FF_2^3 = \langle e_1,e_2,e_3 \rangle_{\FF_2}$. 

In the former case, choosing 
$$\Delta_1=\{1\},\quad\Delta_2 = \{2,3, \ldots,p\},
  \quad \Delta_3 = \{p+1, p+2, \ldots,n\},$$ 
we have 
$$[G:\bigcap^3_{i=1}\Stab_G(\Delta_i)] = p\binom{n}{p}.$$
In the latter case, choosing 
$$\Delta_1=\{0\},\quad\Delta_2 = \{e_1,e_2\},
  \quad\Delta_3 = \FF_2^3 \smallsetminus \{0,e_1,e_2\},$$ 
we see that $\Stab_G(\Delta_i)$ has order dividing $8$ and so its index in $G$
is divisible by $p=3$.

\smallskip
(ii) Now assume that $G$ is transitive but imprimitive on $\Omega$. Let
$\Omega = \cup^m_{i=1}\Omega_i$ be a $G$-invariant partition of~$\Omega$, with
$1 \leq |\Omega_i| = n/m <n$, and~$m$ chosen to be smallest possible subject to
these conditions. 
Let $B := \cap^m_{i=1}\Stab_G(\Omega_i)$ be the base subgroup. Then $G/B$
permutes the $m > 1$ blocks $\Omega_i$ transitively, so $1 \neq G/B$. Since
$G=\bO^{p'}(G)$, we again have $G/B=\bO^{p'}(G/B)$.
%If $S$ is a composition factor of $B$ 
%then $G/B$ is a $p'$-group, whence $G=B$, contradicting the transitivity of $G$ on $\Omega$.
%Hence $S$ is not a composition factor of $B$, and so $p \nmid |B|$.  
Now $G/B$ satisfies the assumptions on $G$, and $G/B$ acts transitively,
faithfully, and primitively (by minimality of $m$) on 
$\{\Omega_1,\ldots,\Omega_m\}$. A desired partition for $G/B$ on this set gives
rise to a desired partition on $\Omega$.

\smallskip
(iii) Finally, we consider the case $G$ acts intransitively on $\Omega = \{1, \ldots,n\}$.
Suppose that $\Omega_1,\ldots,\Omega_m$ are all $G$-orbits on $\Omega$, and let
$K_i$ denote the kernel of~$G$ acting on $\Omega_i$. 
%If $S$ is a composition factor of $K_i$, then $G/K_i$ is a $p'$-group and so $G=K_i$. 
Since $G \neq 1$, $G$ must act non-trivially on at least one $\Omega_i$. So we
may assume that $1 \neq G/K_1$. But $G = \bO^{p'}(G)$, so $p$ divides $|G/K_1|$,
and in fact $G/K_1=\bO^{p'}(G/K_1)$.
%$S$ is not a composition factor of $K_1$, and so $K_1$ is a $p'$-group. 
Now $G/K_1$ satisfies the assumptions on $G$, and $G/K_1$ acts transitively
(and faithfully) on $\Omega_1$. If $\Omega_1 = \Delta_1\sqcup\Delta_2$ is a
desired partition for $G/K_1$, then
$$\Omega = \Delta_1 \sqcup
   \bigl(\Delta_2 \cup (\Omega \smallsetminus \Omega_1)\bigr)$$
is a desired partition for $G$. If
$\Omega_1 = \Delta_1\sqcup\Delta_2\sqcup\Delta_3$ is a desired partition for
$G/K_1$ with $\Delta_3 \neq \emptyset$, then
$$\Omega = \Delta_1 \sqcup \Delta_2 \sqcup
   \bigl(\Delta_3 \cup (\Omega \smallsetminus \Omega_1)\bigr)$$
is a desired partition for $G$. 
\end{proof}

For the first part of Theorem \ref{orbits}, see also \cite[Lemma 3.2]{DMN}.

Next we study the structure of some almost simple groups. In the case $\bar S$
is a simple group of Lie type, we will use the notation $\Inndiag(\bar S)$ as
described in \cite[Thm~2.5.12]{GLS3}; for other simple groups $\bar S$ we use
the convention that $\Inndiag(\bar S) = \bar S$.

\begin{prop}   \label{str}
 Let $p$ be an odd prime and let $S$ be a quasi-simple group. Let
 $\bar S:= S/\bZ(S)$, $\bar S \le H \le \Aut(S)$, and assume that 
 $\bO^{p'}(H/\bar S)=H/\bar S$. 
 \begin{enumerate}[\rm(a)]
  \item Then $H/\bar S$ has a normal $p$-complement.
  \item Suppose that $p \geq 5$ and that $S$ is not of Lie type $\tA_n$,
   $\tw2 \tA_n$. Then $H/\bar S$ is a cyclic $p$-group. The same is true if
   $p=3$ but $S$ is not of Lie type $\tA_n$, $\tw2 \tA_n$, $\tD_4$ or
   $\tE_6(\eps q)$ with $3|(q-\eps)$, $\eps\in\{\pm1\}$.
 \item In general, if $H\le \Inndiag(\bar S)$ then $H/\bar S$ is a cyclic
  $p$-group.
 \end{enumerate}
\end{prop}

\begin{proof}
(i) Note that $H/\bar S$ is embedded in $\Out(\bar S)$. If $\bar S$ is an
alternating or sporadic simple group, then $\Out(\bar S)$ is  a $2$-group and
hence the statements are obvious. Suppose that $\bar S$ is a simple group of
Lie type. In this case,  the structure of $\Out(\bar S)$ is described in
\cite[Thm~2.5.12]{GLS3}, and we will now verify (a) and (b). 

Assume in addition that $S$ is not of type $\tD_4$ when $p=3$. Then the
assumption $\bO^{p'}(H/\bar S)=H/\bar S$ implies that $\bar H:= H/\bar S$ is
contained in $O \rtimes A$, where $O = \mathrm{Outdiag}(\bar S)$ is abelian
and $A$ is a cyclic $p$-group. In particular, (a) holds in this case. Given
the assumptions in (b), we have that either $O$ is a cyclic $p'$-group of order
at most $4$, or $S$ is of type $\tD_{2m}$ and $O$ is a Klein $4$-group.
In the former case, $\Aut(O \cap \bar H)$ is of order at most~$2$. Hence
$\bar H$ centralises $O \cap \bar H$, $O \cap \bar H \leq \bZ(\bar H)$, and
$\bar H/(O \cap \bar H) \hookrightarrow A$ is cyclic. Thus $\bar H$ is
abelian, with cyclic Sylow $p$-subgroup, and the statement follows. In the
latter case, $A$ centralises $O$ (see \cite[Thm~2.5.12(h)]{GLS3}), so again
$\bar H$ is abelian, and we are again done.

Now we complete the proof of (a) in the case $S$ is of type $\tD_4$ over
$\FF_{r^f}$, where $r$ is any prime, and $p=3$. In this case, 
$\Out(\bar S) = O \rtimes (C_f \times \fS_3)$, so $\bar H$ is contained in
$O \rtimes (C_e \times C_3)$, where $e$ is the $3$-part of $f$. Since $O =1$ or
$O = C_2 \times C_2$, the claim follows.

\smallskip
(ii) For (c), just note that $\Inndiag(\bar S)/\bar S$ is either cyclic
or of $p'$-order.
\end{proof}

Next we prove some results on blocks that will be useful later on.
Our notation for block theory mostly follows \cite{N}.

\begin{lem}   \label{bD}
 Suppose that $N\nor G$ and let $b$ be a $G$-invariant block of $N$ with
 defect group~$D$. Let $b_D$ be a block of $D\cent ND$ inducing $b$ with defect
 group $D$. Let $T$ be the stabiliser of $b_D$ in $\norm GD$. If $B$ is a block
 of $G$ covering $b$, then there is a defect group $D_0$ of $B$ such that
 $D_0 \cap N=D$ and $D_0 \leq T$. Also, $\norm GD=\norm ND T$. Furthermore,
 $T=(T \cap N) D_0$ if $G/N$ is a $p$-group.
\end{lem}

\begin{proof}
Since $b$ is $G$-invariant, we have $G=N\norm GD$ by the Frattini argument.
Now $b_0=b_D^{\norm ND}$ is the Brauer First Main correspondent of $b$.
Since $b$ is $\norm GD$-invariant, it follows that $b_0$ is
$\norm GD$-invariant by the uniqueness in the Brauer correspondence. By the
Harris--Kn\"orr correspondence \cite[Thm~9.28]{N}, let $B_0$ be the unique
block of $\norm GD$ that induces $B$ and covers $b_0$. Let $D_0$ be a defect
group of $B_0$, which by \cite[Thm~9.28]{N} is a defect group of $B$.
By Kn\"orr's theorem \cite[Thm~9.26]{N}, we have $D_0 \cap \norm ND$ is a
defect group of $b_0$, and therefore $D_0 \cap N=D_0 \cap \norm ND=D$. 
Since $b_0$ is $\norm GD$-invariant and $b_0$ covers an $\norm ND$-orbit of
blocks of $D\cent ND$, we have $\norm GD=\norm ND T$ by the Frattini argument,
where recall that $T$ is the stabiliser of $b_D$ in $\norm GD$. By the
Fong--Reynolds Theorem 9.14 of \cite{N},
there exists an $\norm GD$-conjugate of $D_0$ contained in $T$.

Suppose then that $D_0^x \leq T$, where $x \in \norm GD$, and $D_0^x$ is a
defect group of $b_T$, the block of $T$ which is the Fong--Reynolds
correspondent of $B_0$ over $b_D$. Then $D_0^x$ is a defect group of $B$. Also
$D_0^x \cap N=(D_0 \cap N)^x=D^x=D$, and this proves the first part. 
For the final part, notice that $T/(T\cap N)$ is a $p$-group.  Furthermore,
$(b_D)^{T\cap N}$ is the Fong--Reynolds correspondent of $b_0$ over $b_D$.
By uniqueness, it follows that $(b_D)^{T\cap N}$ is $T$-invariant (using that
$b_0$ and $b_D$ are $T$-invariant). Also, notice that $(b_D)^{T\cap N}$
is the only block of $T\cap N$ covering $b_D$, using \cite[Cor.~9.21]{N}. We
conclude that $b_T$ covers $(b_D)^{T\cap N}$, a block with defect group $D$.
Since
$(b_D)^{T\cap N}$ is $T$-invariant and $b_T$ is the only block of $T$ covering
it (\cite[Cor.~9.6]{N}), we have $T=(T\cap N)D_0^x$, by \cite[Thm~9.17]{N}, for
instance.
\end{proof}

\begin{prop}   \label{useful}
 Let $N\nor G$ and let $B$ be a $p$-block of $G$ with abelian
 defect groups. Suppose that $G/N$ is a $p$-group
 and that $B$ covers a $G$-invariant block $b$ of $N$ with defect group $D$.
 Suppose $b_{D}$ is a block of $\cent N{D}$ with defect group $D$
 that induces $b$. If $x \in \norm G{D}$ is a $p$-element that
 fixes $b_{D}$, then $[x, D]=1$.   
\end{prop}

\begin{proof}
Since $B$ has abelian defect groups, notice that $D$ is abelian by Kn\"orr's
theorem \cite[Thm~9.26]{N}. Let $T$ be the stabiliser of $b_D$ in $\norm GD$.
By Lemma~\ref{bD}, there is a defect group $D_0$ of $B$ such that
$D_0 \cap N=D$, $D_0 \le T$, and $T=(T\cap N)D_0$. Since by hypothesis $D_0$
is abelian, we have $T=(T\cap N)\cent TD$. Then $(T\cap N)/\cent ND$ is a
$p'$-group by \cite[Thm~9.22]{N}. Therefore $T/\cent TD$ is a $p'$-group,
and the result follows. 
\end{proof}

Recall that a block $B$ of a finite group $G$ is called \emph{quasi-primitive}
if whenever $N\nor G$, then $B$ covers a unique block of $N$.

\begin{prop}   \label{amaz}
 Suppose that $B$ is a quasi-primitive $p$-block of a finite group $G$ such
 that all the irreducible characters of $B$ have height zero. Let $N\nor G$
 and let $b$ be a block of $N$ covered by $B$.
 If $\theta\in\irr b$, then $|G:G_\theta|$ is not divisible by $p$.
\end{prop}

\begin{proof}
We use Dade's group $K=G[b]$. (See, for instance, \cite{Mu3} for an
introduction to this object.) Then all the irreducible characters of $b$
are $K$-invariant, see \cite[Lemma 3.2(a)]{KoS1}, for instance. In particular,
$\theta$ is $K$-invariant. Also, $K \nor G_b=G$, where $G_b$ is the stabiliser
of $b$ in $G$. Also if $B'$ is the (unique) block of $K$ covered by $B$, then
$B'^G=B$ and $B$ is the only block of $G$ covering $B'$. (See
\cite[Thm~3.5]{Mu3}.) Since $B$ covers $b$, let $\chi\in\irr B$ be
over~$\theta$. Let $\eta \in \irr K$ be under $\chi$ and over $\theta$.
Since $\chi$ is over $\eta$, it follows that the block of $\eta$ is necessarily
$B'$.
By hypothesis, $B'$ is $G$-invariant. By \cite[Cor.~9.18]{N}, we have
$|G:G_\eta|$ is prime to~$p$. Since $\eta_N=v \theta$ for some $v\ge1$,
it follows that $G_\eta \leq G_\theta$  and therefore $|G:G_\theta|$ is also
not divisible by~$p$.
\end{proof}

In order to prove Theorem A, we will need the block theory \emph{above} a
nilpotent block, and \emph{above} a block with cyclic defect group.

\begin{thm}   \label{kp}
 Let $G$ be a finite group, and $N\nor G$. Let $b$ be a $G$-invariant nilpotent
 block of $N$ with defect group $Q$. Then there exist a $p$-subgroup $P$ of $G$
 such that $Q=P\cap N$, a finite group $L$ with a Sylow $p$-subgroup $P$, a
 central $p'$-extension $L'$ of $L$ by $Z\leq \zent{L'}$ and $\mu\in\irr Z$,
 such that:
 \begin{enumerate}[\rm(a)]
  \item $PN/N \in \syl p{G/N}$.
  \item We have $Q \nor L$ and $G/N \cong L/Q$. If $|\zent G|$ is not divisible
   by $p$, $\zent G \leq N$ and $Q\zent G<N$, then $|L':\zent{L'}|<|G:\zent G|$.
  \item There is a bijection $B \mapsto B'$ between the blocks of $G$ that
   cover $b$ and the blocks of $L'$ that cover the block of $\mu$ preserving
   defect groups. Also, there is a height preserving bijection
   $\irr B \rightarrow \irr{B'}$.
 \end{enumerate}
\end{thm}

\begin{proof}
These are consequences of the theory of blocks above nilpotent blocks developed
in \cite{KP}. This is also described in Section 8.12 of \cite{L18}. (See also
Section 7.2 of \cite{Sa}).
The existence of $P$, the fact that $P\cap N=Q$ and that $PN/N$ is a Sylow
$p$-subgroup of $G/N$ follow from~8.12.5 and~8.12.6 of \cite{L18}. The
existence of $L$ and the fact that $P\in\syl pL$, that $Q\nor L$ and that
$L/Q \cong G/N$ are part of the statement in~8.12.5. Theorem~8.12.5 also
provides a bijection between the blocks of $G$ covering $b$ and the blocks of a
certain twisted group algebra ${\mathcal O}_\alpha L$ of $L$,
and corresponding blocks are Morita equivalent (see Remark~8.12.8).
In Remark~8.12.8, the existence of $L'$ and the relationship with
${\mathcal O}_\alpha L$ is given. 

Corollary~9.2.5 of \cite{L18} says that a perfect isometry preserves heights,
and any of Corollary~9.3.3 or~9.3.4 imply that a Morita equivalence induces a
perfect isometry. Then~9.7.1 of \cite{L18} implies that any Morita equivalence
given by a bimodule with source of rank prime to~$p$ induces isomorphisms
between defect groups. This is restated as part of 9.11.2 of of \cite{L18}.

This covers all but the last inequality in (b). So, assume $|\zent G|$ is not
divisible by $p$, $\zent G \leq N$ and $Q\zent G<N$. Then 
$$|L':\zent{L'}| \le |L|=|G/N||Q|
  =|G/N||Q\zent G: \zent G|<|G:\zent G|.\qedhere$$
\end{proof}

\begin{thm}   \label{cyciAM}
 Suppose $N\nor G$ are finite groups, $B$ is a $p$-block of $G$ with
 defect group~$D$ covering a $G$-invariant block $b_0$ of $N$, where $p$ is odd.
 Suppose that $N$ is the central product of the $G$-conjugates of a component
 $S$ of $G$.  Suppose that the block $b_0$ covers the block $b$ of $S$ and that
 $D_0=D\cap S$ is cyclic and non-central in $S$.
 Let $D_1=D\cap N$, and let $b_1$ be the block of $\norm N{D_1}$ that
 induces $b_0$. Let $B_1$ be the block of $\norm G{D_1}$ with defect group
 $D$ that covers $b_1$ and induces $B$. If all characters in $\irr B$ have
 height zero, then all characters in $\irr{B_1}$ have height zero.
\end{thm}

\begin{proof} 
By \cite[Thms~1.1 and 7.6]{KoS}, the block $b$ of $S$ satisfies the inductive
Alperin--McKay condition. (In fact, the ``intermediate subgroup" in
\cite[Def.~7.2]{KoS}, is the normaliser of the corresponding defect group.) 
We apply Theorem~6.1 and Proposition~6.2 of \cite{NS}, noticing that we do not
need to assume that the simple group $S/\zent S$ satisfies the inductive
Alperin--McKay  condition, because we know that blocks of $S$ involved in our
statement have cyclic defect groups (and therefore satisfy the inductive
Alperin--McKay condition). Notice that $D_1$ is a $p$-radical subgroup of $N$
since it is the defect group of $b_1$.
If $f$ is any block of $N$ with defect group $D_1$, using that $N$ is the
central product of $G$-conjugates of~$S$, then $f$ covers a unique block of $S$
with defect group $D_1 \cap S=D_0$, which is cyclic.
Hence, by \cite[Prop.~6.2]{NS}, there is an $\norm G{D_1}$-equivariant bijection
$$\Omega: \irr{N|D_1} \rightarrow \irr{\norm N{D_1}|D_1}$$ 
such that
$$(G_\theta, N, \theta) \sim_b (\norm G{D_1}_{\theta'},\norm N{D_1}, \theta')$$
for every $\theta \in \irr{b_0}$, where $\theta'=\Omega(\theta)$. 
(Here $\irr{N|D_1}$ is the set of characters of $N$ belonging to blocks with
defect group $D_1$. The definition of block isomorphism of character triples
denoted above with $\sim_b$ is given in \cite[Def.~3.6]{NS}.) Now, we apply
\cite[Cor.~3.10]{NS} to construct a height preserving bijection
$\irr{B|\theta}\rightarrow \irr{B_1|\theta'}$. Since
$\Omega(\irr{b_0})=\irr{b_1}$ by \cite[Thm~6.1(b)]{NS}, we easily conclude that
all irreducible characters in $B_1$ have height zero.
\end{proof}

The following gives a shorter proof of a generalisation of the main result
of~\cite{K}.  

\begin{thm}   \label{kul}
 Suppose that $N,M\nor G$ with $G=NM$. Let $B$ be a block of $G$ which
 covers a unique block $b_{N \cap M}$ of $N\cap M$. Then there is a defect
 group $D$ of $B$ such that $D=(D\cap N)(D \cap M)$.
\end{thm}

\begin{proof}
We argue by induction on $|G:N| +|G:M|+ |G|$. We may assume that $N, M<G$.
Let $b_N$ be a block of $N$ covered by $B$ and $b_M$ a block of $M$
covered by $B$, both covering $b_{N\cap M}$.  
Let $T=G_{b_N}$ be the stabiliser of $b_N$, and by the Fong--Reynolds theorem,
let $b_T$ be the block of $T$ that induces $B$ and covers $b_N$.
If $T<G$, then  by induction there is a defect group $D$ of $b_T$ (and
therefore of $B$) such that 
$$D=(D\cap N)(D\cap T\cap M)\leq (D\cap N)(D\cap M).$$
We thus may assume that $b_N$ and $b_M$ are $G$-invariant.

Now, let us fix $D$ a defect group of $B$.
Let $X=M\cap ND$, so that $X/(N\cap M)$ is a $p$-group. Thus $b_{N\cap M}$ is
covered by a unique block $b_X$ of $X$. Also the unique block $b_{ND}$ that
covers $b_N$ has defect group $D$ by Problem 9.4 of \cite{N}. If $ND<G$, by
induction there is $n\in ND$ such that
$$D^n=(D^n \cap N)(D^n\cap X)=(D \cap N)^n(D \cap X)^n,$$ 
using that $X \nor ND$. Hence 
$$D=(D\cap N)(D \cap X)\leq (D\cap N)(D\cap M).$$ 
So we may assume
$G=ND=MD$; in particular, $G/N$ and $G/M$ are $p$-groups.
\par
Suppose that $N\le Y\nor G$, where $Y<G$, and let $b_Y$ be the unique block
of~$Y$ covering~$b_N$.  Let $Z=M\cap Y$, and let $b_Z$ be the unique block
of $Z$ covering $b_{N\cap M}$. Thus $b_Z$ is $G$-invariant, and covered by $B$.
By induction, there is a defect group $D_1$ of $B$
such that $D_1=(D_1\cap Y)(D_1\cap M)$. Now, since $b_Y$ is $G$-invariant,
we have $D_1\cap Y$ is a defect group of $b_Y$. Again by induction, there is
$y \in Y$ such that 
$$D_1^y \cap Y=(D_1^y\cap N)(D_1^y \cap Z).$$ 
Hence, $D_1\cap Y=(D_1 \cap N)(D_1\cap Z)$.
Thus $D_1=(D_1\cap N)(D_1\cap M)$, and we are done.
\par
So we may assume $|G:N|=p=|G:M|$. Also, $G=(N\cap M)D$, using Fong's
Theorem~9.17
of~\cite{N}. Since $G=(N\cap M)D$, we have that $N=(N\cap M)(N \cap D)$
and $M=(N\cap M)(M \cap D)$. In particular $N\cap D \ne M\cap D$. Since
$D/(D\cap N\cap M)$ is $C_p\times C_p$, we necessarily have
$D=(D\cap N)(D \cap M)$.
\end{proof}
 
\begin{cor}   \label{central}
 Suppose that $G=S_1 *\cdots *S_m$ is a central product of its subgroups $S_i$,
 $1 \leq i \leq m$. If $B$ is a block of $G$ with defect group $D$, then
 $D=(D\cap S_1) \cdots (D \cap S_m)$.
\end{cor}

\begin{proof}
Use induction on $m$ and Theorem \ref{kul}. (See also Lemma 7.5 of \cite{Sa}.)
\end{proof}

The following elementary result will be used in the final step of our
proof of Theorem~A.

\begin{lem}   \label{elem}
 Suppose that $N\nor G$ are finite groups with $G/N$ a $p$-group.
 Let $Q \nor G$ be such that $Q\cap N=1$. Let $b$ be a $G$-invariant block
 of $N$ and let $B$ be the block of $G$ that covers $b$.
 Let $\bar b$ be the unique block of $NQ/Q$ which corresponds to $b$
 under the natural isomorphism, and let $\bar B$ be the unique block
 of $G/Q$ that covers $\bar b$. Then $D/Q$ is a defect group of $\bar B$.
\end{lem}

\begin{proof}
We know that the block $\bar B$ is contained in a unique block $\tilde B$ of
$G$. (See the remark before Theorem 7.6 of \cite{N}.)
Let $E/Q$ be a defect group of $\bar B$. Let $\tau \in \irr b$, and consider
$$\gamma:=\tau \times 1_Q \in \irr{N\times Q}.$$
The block of $N\times Q$ that contains $\gamma$ is only covered by $B$ by 
\cite[Cor.~9.6]{N}. Then~$\gamma$, considered as a character of~$NQ/Q$, lies
in~$\bar b$. Let $\psi \in \irr{\bar B}$ over $\gamma$. Then
$\psi$, considered as a character of $G$, lies over $\tau$, and therefore
$\psi\in\irr B$. It follows that $\irr B\cap \irr{\bar B}\ne\emptyset$, and
hence $\bar B$ is contained in $B$.
(See the remark before Theorem 9.9 of \cite{N}.) Consequently, $\tilde B=B$.
By \cite[Thm~9.9(a)]{N}, we have $E/Q\leq D/Q$. Notice that $\bar b$ is
$G$-invariant by uniqueness. Therefore $(E/Q)(NQ/Q)=G/Q$ by \cite[Thm~9.17]{N}.
By the same reason, $DN=G$. Also, $(E/Q) \cap (NQ/Q)$ is a defect group of
$\bar b$, and $D\cap N$ is a defect group of $b$
(by \cite[Thm~9.26]{N}). In particular, $|E\cap NQ|=|Q||D\cap N|$.
Then $|E/Q|=|G:NQ||D\cap N|=|D/Q|$, and the proof is complete.
\end{proof}

%%%%%%%%%%%%%%%%%%%%%%%%%%%%%%%%%%%%%%%%%%%%%%%%%%%%%%%%%%%%%%%%%%%%%%%%%
\section{Orbits of characters in a block}   \label{sec:ThmB}

In this section, we prove Theorem B, which we now restate.  
 
 \begin{thm}\label{newmandi}
 Suppose that $p$ is an odd prime, $S$ is a quasi-simple group and $b$ is a
 $p$-block of~$S$ with non-cyclic defect groups. Then at least one of the
 following statements holds:
 \begin{enumerate}[\rm(1)]
  \item There exist characters $\alpha,\beta,\gamma\in\Irr(b)$
   that belong to three different $\Aut(S)$-orbits; or
  \item all characters in $\Irr(b)$ have the same degree.
 \end{enumerate}
\end{thm}
  
We remark that in the proof of Theorem A in Section~\ref{sec:finalpf},
Theorem~\ref{newmandi} is only needed for blocks with abelian defect groups.
However, the more general statement may be of independent interest, and its
proof is not substantially different. We also remark that the assumption of
non-cyclic defect groups is needed.  For example, for certain values of $q$
with $p\mid\mid(q-1)$, there is a cyclic quasi-isolated block of $\SL_2(q)$
with characters from only two $\Aut(\SL_2(q))$-orbits.
%Examples also occur in some triple covers of sporadics when $p=3$.
  
Throughout, for $A$ a group acting on a group $G$ as automorphisms and a block
$b$ of~$G$, we write $k_A(b)$ for the number of distinct $A$-orbits of
characters of $G$ whose intersection with $\Irr(b)$ is non-empty. In the
situation of Theorem~\ref{newmandi}, our aim will be to show
$$k_{\Aut(S)}(b)\geq 3.$$ 
We also use $\cd(b)$ to denote the set of distinct
character degrees in $\Irr(b)$.

%%%%%%%%%%%%%%%%%%%%%%%%%%%%%%%%%%%%%
\subsection{Initial considerations}\label{sec:othergroups}

We begin by considering cases that can be completed more computationally.

\begin{prop}   \label{prop:sporadic}
 Let $p\geq 3$ be a prime, $S$ a quasi-simple group such that  
 $S/\zent{S}$ is one of the sporadic simple groups, the Tits group
 $\tw{2}\tF_4(2)'$, $\type{G}_2(2)'$, $\tw{2}\type{G}_2(3)'=\PSL_2(8)$,
 or a simple group of Lie type with exceptional Schur multiplier. Let $b$ be a
 $p$-block for $S$ with non-cyclic defect and $|\cd(b)|>1$. Then
 $k_{\Aut(S)}(b)\geq 3$.  In particular, Theorem~{\rm\ref{newmandi}} holds for
 these groups.
\end{prop}

\begin{proof}
This can be seen using the GAP Character Table Library \cite{GAP}. We note that
the groups with exceptional Schur multipliers are listed in
\cite[Tab.~6.1.3]{GLS3}. 
\end{proof}

\begin{prop}   \label{prop:alts1}
 Let $p\geq 3$ a prime and let $S$ be quasi-simple such that
 $S/\zent{S}=\fA_n$, with $5\leq n\leq 8$. Then $|\cd(b)|\geq 3$ for every
 non-cyclic $p$-block $b$ of $S$.  In particular, Theorem~{\rm\ref{newmandi}}
 holds for these groups.
\end{prop}

\begin{proof}
This can again be seen using GAP and the GAP character table library. 
\end{proof}

\begin{prop}   \label{prop:alts2}
 Let $S=\hat{\fA}_n$ be the double cover of the alternating group~$\fA_n$,
 with $n\geq 9$.  Suppose that $p\geq3$ is a prime and let $b$ be a $p$-block
 of $S$ with non-cyclic defect groups.  Then $k_{\Aut(S)}(b)\geq 3$.
 In particular, Theorem~{\rm\ref{newmandi}} holds for these groups.
\end{prop}

\begin{proof}
Recall that $\Aut(S)=\fS_n$.  If $b$ is a $p$-block of
$\fA_n$, then $k_{\fS_n}(b)\geq k(\wt{b})/2$, where $\wt{b}$ is a $p$-block
of $\fS_n$ above $b$. By \cite[Prop.~11.4]{Ol93}, we have $k(\wt{b})=k(p,w)$,
where $w$ is the so-called weight of $\wt{b}$ and $k(p,w)$ is as in
\cite[(3.11)]{Ol93}. But note that $k(p,w)/2>2$ for $p\geq 3$ and $w\ge 2$. 
Hence $k_{\fS_n}(b)\geq 3$.
\par
If, instead, $b$ is a block of $S$ lying over the (unique) non-trivial
character of $\zent{S}$, then
$k_{\hat{\fS}_n}(b)\geq k(\wt{b})/2$, where now $\wt{b}$ is a so-called spin
block of $\hat{\fS}_n$ above $b$ and $\hat{\fS}_n$ is a double cover of
$\fS_n$ inducing all automorphisms of $S$. In this case, we have an analogous
invariant $k(\wt{b})=\hat{k}^\pm(\bar{p},w)$ (see \cite[Prop.~(13.4)]{Ol93}).
Since $b$ is non-cyclic, we have $w> 1$, and using the definitions in
\cite[Sec.~13]{Ol93}, we see again that $\hat{k}^\pm(\bar p,w)/2>2$, and we have
$k_{\Aut(S)}(b)\geq 3$ as desired.
\end{proof}

We also need to consider the groups of Lie type that arise as the fixed points
of a simple, simply connected linear algebraic group under a Steinberg
endomorphism but are not quasi-simple. Throughout, we let
\begin{equation}\label{eq:setE} 
  \fE:=\{\SL_2(2), \SU_3(2), \Sp_4(2)\}.
\end{equation}  

\begin{prop}   \label{prop:nonquasisimple}
 Let $p\geq 3$ and $B$ a $p$-block of a group $G\in\fE$, with positive defect.
 Then $|\cd(B)|\geq 2$.  If $B$ is non-cyclic, then $|\cd(B)|\geq 3$.
\end{prop}

\begin{proof}
This can be seen using the GAP character table library.
\end{proof}

We remark that the groups $\SL_2(3)$, $\type{G}_2(2)$, $\tw{2}\tB_2(2)$,
$\tw{2}\type{G}_2(3)$, and $\tw{2}\tF_4(2)$, which also occur as fixed points
of simple, simply connected groups but are not quasi-simple, also satisfy the
above statement, with the exception of the cyclic $3$-blocks of $\SL_2(3)$.
However, we will not need this here.

Next, we deal with the Suzuki, Ree, and triality groups.  

\begin{prop}   \label{suzree}
 Let $S$ be quasi-simple such that $\bar S:=S/\zent{S}$ is a Suzuki, Ree, or
 triality group $\tw{2}\tB_2(q^2)$, $\tw{2}\type{G}_2(q^2)$,
 $\tw{2}\tF_4(q^2)$, or $\tw{3}\tD_4(q)$. Let $p\geq 3$ be a prime and let $b$
 be  a $p$-block of $S$ with non-cyclic defect groups. Then
 $k_{\Aut(S)}(b)\geq 3$.
 In particular, Theorem~{\rm\ref{newmandi}} holds for these groups.
\end{prop}

\begin{proof}
Note that the Schur multiplier of $\bar S$ is trivial or $S$ was considered
already in Proposition \ref{prop:sporadic}, so we assume $\bar S=S$. First
suppose $p\mid q^2$, so that $S=\tw{2}\type{G}_2(q^2)$ or $\tw{3}\type{D}_4(q)$.
Then by a theorem of Humphreys \cite{Hum}, $S$ has exactly two
blocks, namely the principal block and a block of defect zero containing only
the Steinberg character. By observing the list of unipotent character degrees in
\cite[Sec.~13.9]{Carter}, we see that there are at least 3 distinct character
degrees in the principal block.

Now suppose $3\le p\nmid q^2$. Then Sylow $p$-subgroups of $\tw{2}\tB_2(q^2)$
and $\tw{2}\type{G}_2(q^2)$ are cyclic.   
%So, we let $S=\tw{2}\tF_4(q^2)$ with $q^2=2^{2f+1}$ or $S=\tw{3}\tD_4(q)$.  
Let $S=\tw{2}\tF_4(q^2)$, with $q^2=2^{2f+1}$. Here from \cite[Bem.~1]{Ma90},
if $p\nmid (q^2-1)$, then each semisimple $p'$-element $s$ in the dual group
$S^\ast$ defines a unique block of positive defect containing the Lusztig
series $\cE(S,s)$. 
If $p\mid (q^2-1)$, such an $s$ defines a unique block of positive defect if
$s$ is of class type $t_2$ or $t_3$, and three such blocks for $s=1$ or of type
$t_1$ in the notation of \cite{Sh75}. In the latter two cases, only one of the
three blocks has non-cyclic defect. In each
case, we can see from the centraliser structures and conjugacy class types in
\cite{Sh75}, together with the fact that field automorphisms permute Lusztig
series via $\cE(S,s)^\alpha=\cE(S,s^{\alpha^\ast})$ for
$\alpha^\ast\in\Aut(S^{\ast})$ dual to $\alpha$ (see \cite[Cor.~2.4]{NTT08})
that there are at least three characters in the relevant blocks that cannot be
conjugate under field automorphisms of $S$. In the case that $S=\tw{3}\tD_4(q)$,
we may argue similarly to above, taking into consideration the results of
\cite[Lemma 5.9]{DM87} and the structure of centralisers and tori discussed in
loc.~cit.
\end{proof}

%%%%%%%%%%%%%%%%%%%%%%%%%%%%%%%%%%%%%
\subsection{Reducing to quasi-isolated blocks}\label{sec:redisol}
Let $p$ be an odd prime and let $\bG$ be a simple algebraic group of simply
connected type over $\overline{\FF}_r$ for some prime $r$. Let $q$ be a
power of $r$ and $F\colon\bG\rightarrow\bG$ a Frobenius endomorphism with
respect to an $\FF_q$-structure. Write  $G:=\bG^F$ for the corresponding finite
reductive group. 

Let $\iota\colon\bG\hookrightarrow \tbG$ be a regular embedding as in
\cite[(15.1)]{CE04} (see also \cite[Sec.~1.7]{GM20}), and write
$\wt{G}:=\tbG^F$. Let $(\bG^\ast, F)$ be dual to $(\bG, F)$, so that $\bG^\ast$
is simple of adjoint type, and let $(\tbG^\ast, F)$ be dual to $(\tbG, F)$.
We will write $G^\ast:=\bG^{\ast^F}$ and $\wt{G}^\ast:=\wt{\bG}^{\ast F}$.
We further have a dual $F$-equivariant epimorphism
$\iota^\ast\colon\tbG^\ast\twoheadrightarrow\bG^\ast$ induced by $\iota$.
Now, with this setting, $\Aut(G)$ is induced by $\wt{G}\rtimes \cD$,
where $\cD$ is the group generated by appropriately chosen graph and field
automorphisms (see e.g. \cite[Thm~2.5.1]{GLS3}). 

The set $\irr{G}$ is a disjoint union of Lusztig series $\cE(G, s)$ (also
called rational series),
where $s$ runs over $G^\ast$-conjugacy class representatives of semisimple
elements of $G^\ast$. If $\wt{s}\in\wt{G}^\ast$ is such that
$\iota^\ast(\wt{s})=s$, then the series $\cE(G,s)$ consists of the constituents
of the restrictions of the characters in $\cE(\wt{G}, \wt{s})$ to $G$ (see
\cite[Prop.~15.6]{CE04}).

We next consider the case of groups of Lie type in defining characteristic,
i.e., when $r=p$. 

\begin{prop}   \label{prop:definingchar}
 Assume $S$ is quasi-simple such that $S/\zent{S}$ is a simple group of Lie
 type defined in characteristic $p\geq 3$. Let $b$ be a $p$-block of $S$ with
 non-cyclic defect groups. Then $k_{\Aut(S)}(b)\geq 3$. In particular,
 Theorem~{\rm\ref{newmandi}} holds for these groups.
\end{prop}

\begin{proof}
Let $G=\bG^F$ be as above such that $\bar{S}:=S/\zent{S}=G/\zent{G}$.  We may
assume $G$ is the full covering group of $\bar S$, as the exceptional covers
have been discussed in Proposition~\ref{prop:sporadic}, and that $\bar{S}$ is
not of Suzuki or Ree type, from Proposition \ref{suzree}.  By applying
\cite[Thm~9.9]{N}, it suffices to prove the statement for $G$, since
$p\nmid|\zent{G}|$. Now, every $p$-block of~$G$ is either maximal defect or
of defect zero \cite{Hum}, and the blocks of maximal defect are in bijection
with the characters of $\zent{G}$, via their central character. Let
$\theta\in\Irr(\zent{G})$ correspond to $b$. 

Given this, and by inspection of the character tables for $\SL_2(q)$,
$\SL_3(\eps q)$, and $\Sp_4(q)$ (see \cite[Tab.~2.6 and references in
Tab~2.4]{GM20}), we see there are at least three character degrees for each
block of positive defect in these cases, so we further assume that $G$ is not 
one of these groups. Furthermore, we may assume $\zent{G}\ne1$ as any
non-solvable group has at least four irreducible characters.

Let $T\le G$ be a maximal torus. We claim that there is $s\in T^*$, a maximal
torus of~$G^*$ in duality with $T$, such that the characters in the (rational)
Lusztig series $\cE(G,s)$ lie above $\theta$. Indeed, as $\zent{G}\le T$ there
is $\tilde\theta\in\Irr(T)$ with $\tilde\theta|_\zent{G}=\theta$, and then
by the character formula, also
$\RTG(\tilde\theta)|_\zent{G}=\RTG(\tilde\theta)(1)\theta$. So taking
$(T^*,s)$ dual to $(T,\theta)$ we find $s$ as claimed. Again by the character
formula, $\RTG(\mu\tilde\theta)|_\zent{G}=\RTG(\mu\tilde\theta)(1)\theta$ for
all $\mu\in\Irr(T)$ with $\mu|_\zent{G}=1$, and so by duality
$R_{\bT^*}^{\bG^*}(st)|_\zent{G}=R_{\bT^*}^{\bG^*}(st)(1)\theta$ for
all $t\in T^*\cap[G^*,G^*]$ (see the proof of \cite[Lemma~4.4(ii)]{NT2}).
Note we may choose $s=:s_1$ to have order only divisible by primes dividing
$|G^*:[G^*,G^*]|=|\bZ(G)|$. 

As $G$ is not of types $A_1,A_2,\tw2A_2,B_2$ there are at least two Zsigmondy
primes $\ell_2$, $\ell_3$ dividing $|G^*|$ but not $|\bZ(G)|$. Applying the
above claim to suitable maximal tori $T_i \leq G$, $i=2,3$, we may choose 
semisimple elements $s_i$ with order divisible by $\ell_i$ and possibly by some
primes dividing $|G^*/[G^*,G^*]|$. This ensures that $|s_i|$, $i = 1,2,3$, are
pairwise distinct. Hence the disjoint series $\cE(G,s_i)$, $i=1,2,3$,
contain distinct characters above $\theta$ not conjugate under $\Aut(G)$
by \cite[Cor.~2.4]{NTT08}, completing the proof.
\end{proof}

Given Proposition \ref{prop:definingchar}, we now assume that $r\neq p$ for the
remainder of Section~\ref{sec:ThmB}.
\medskip

If $s\in G^*$ is a $p'$-element, as is customary we write $\cE_p(G,s)$ for the
union $\bigcup_t \cE(G,st)$ where $t$ ranges over the $p$-elements of $G^\ast$
commuting with $s$. Then $\cE_p(G,s)$ is a union of $p$-blocks of $G$, and for
each block $B$ with $\Irr(B)\subseteq \cE_p(G,s)$, we have
$\Irr(B)\cap \cE(G,s)$ is nonempty (see \cite[Thm~9.12]{CE04}). At times, we
will write $\cE(G,p')$ to denote the union of the series $\cE(G,s)$ for
$s\in G^\ast$ ranging over semisimple $p'$-elements.
  
A fundamental result of Bonnaf{\'e}--Rouquier \cite{BR03} yields that the blocks
in $\cE_p(G,s)$ are Morita equivalent to so-called {quasi-isolated} blocks of
suitable Levi subgroups.  In the more general setting of a finite reductive
group $H:=\bH^F$, that is, the fixed points of a connected reductive group
$\bH$ under a Frobenius morphism $F\colon\bH\rightarrow\bH$ endowing $\bH$ with
an $\FF_q$-rational structure, a block of $H$ is called \emph{quasi-isolated} if
it lies in $\cE_p(H,s)$ for a semisimple $p'$-element $s\in H^\ast$ such that
$\cent{\bH^\ast}{s}$ is not contained in any proper $F$-stable Levi subgroup of
$\bH^\ast$.  (In such a situation, $s$ is also called \emph{quasi-isolated}.)

The following setup will be useful. Let $\bH$ be an $F$-stable Levi subgroup
of $\bG$ and $H:=\bH^F$. Let $\bH_0:=[\bH, \bH]$ and $H_0:=\bH_0^F$, so that
$\bH_0$ is semisimple of simply connected type, by \cite[Prop.~6.20(c) and
12.14]{MT}. Then by \cite[Cor.~1.5.16]{GM20}, $H_0$ is isomorphic to a direct
product $\prod_{i=1}^k \bH_i^{F_i}$, where each $\bH_i$ is simple of simply
connected type and $F_i$ is a Frobenius morphism obtained as some power of $F$.
Write $H_i:=\bH_i^{F_i}$. Let $B$ be a $p$-block of $H$ and let $B'$ be a block
of $H_0$ covered by $B$.  Then $B'$ is isomorphic to a tensor product
$\otimes_{i=1}^k B_i$, where $B_i$ is a block of $H_i$ for each $1\leq i\leq k$.

Recall from \eqref{eq:setE} that $\fE=\{\SL_2(2), \SU_3(2), \Sp_4(2)\}$.
Note that $H_i$ is perfect (and hence quasi-simple) unless
$H_i\in\fE \cup\{\SL_2(3)\}$
(see \cite[Thm~24.17]{MT}). (Indeed, note that the excluded groups 
$\type{G}_2(2)$, $\tw{2}\tB_2(2)$, $\tw{2}\type{G}_2(3)$, and $\tw{2}\tF_4(2)$
cannot occur as an $H_i$ since $F$ is a Frobenius endomorphism and $\type{G}_2$
does not occur as a component of a Levi subgroup of the other types.) Since $G$
is defined in characteristic distinct from $p$ and $|\SL_2(3)|$ is only
divisible by the primes $2$ and $3$, we further will not need to consider
$\SL_2(3)$ in what follows.  In the following, we will write
$$\bar{k}_{\Aut(H_i)}(B_i):=\begin{cases}
  k_{\Aut(H_i)}(B_i)& \text{if $H_i\not\in\fE$,}\\
  |\cd(B_i)|& \text{if $H_i\in\fE$.}\end{cases}$$

\begin{hyp}   \label{hyp:orbisol}
Keep the notation and situation of the previous paragraph. For a positive
integer~$c$, let $B$ be a $p$-block of $H$ such that there exists a block $B'$
as above that satisfies at least one of the following three conditions:
\begin{enumerate}
 \item[\namedlabel{hyp1}{(3.8.1)}] there exists some $i$ with $1\leq i\leq k$
  such that $\bar{k}_{\Aut(H_i)}(B_i)\geq c$;
 \item[\namedlabel{hyp2}{(3.8.2)}] there exist $i,j$ with $1\leq i\neq j\leq k$
  such that $\bar{k}_{\Aut(H_i)}(B_i)\geq c-1$ and, furthermore, 
  $\bar{k}_{\Aut(H_j)}(B_j)\geq 2$;
 \item[\namedlabel{hyp3}{(3.8.3)}] $B'$ is cyclic of positive defect, and there
  exists some $i$ with $1\leq i\leq k$ such that $B_i$ is cyclic and
  $\bar{k}_{\Aut(H_i)}(B_i)\geq c-1$.
\end{enumerate}
\end{hyp}

\begin{prop}   \label{prop:hyponBi}
 Let $G:=\bG^F$ be as above. Let $\bH$ be an $F$-stable, proper Levi subgroup
 of $\bG$ and $H:=\bH^F$. Let $p$ be an odd prime not dividing $q$ and $B$ be a
 non-cyclic $p$-block of $H$ with $|\cd(B)|>1$. Then:
 \begin{enumerate}[\rm(a)]
  \item In the notation above, any block $B'$ of $H_0$ covered by $B$ has
   positive defect.
  \item If $B$ satisfies Hypothesis~{\rm\ref{hyp:orbisol}} for some
   positive integer $c$ then $k_{\Aut(H)}(B)\geq c$.
 \end{enumerate}
\end{prop}

\begin{proof}
Keep the notation above and let $B'$ be a block of $H_0$ covered by $B$. First,
assume $B'$ is defect zero. Write
$\Irr(B')=\{\theta_1\}$ and let $T:=H_{\theta_1}$ be the inertia subgroup for
$\theta_1$ in $H$. Since $H/H_0$ is abelian and restrictions from $H$ to $H_0$
are multiplicity-free (a result of Lusztig, see \cite[Thm~15.11]{CE04}),
Gallagher's theorem and Clifford correspondence imply that every element of
$\irr{H\mid\theta_1}$ is of the form $(\beta\hat\theta_1)^H$ where
$\hat\theta_1$ is an extension of $\theta_1$ to~$T$ and $\beta$ is a (linear)
character of $T/H_0$.  In particular, every member of $\Irr(B)$ is of this
form, and hence $|\cd(B)|=1$, a contradiction. This shows (a).

We now assume that $B$ satisfies Hypothesis \ref{hyp:orbisol} for $c$ and aim
to show 
$$k_{\Aut(H)}(B)\geq c.$$ 
Note that this is trivial when $c=1$, so we assume throughout that $c\geq 2$.
We claim that it suffices to show that 
\begin{equation}\label{eq:blockB}
  k_{\Aut(H_0)}(B')\geq c.
\end{equation}
Indeed, if this is the case, write $\theta_1,\ldots,\theta_c$ for
representatives in $\irr{B'}$ of $c$ distinct orbits. Then by
\cite[Thm~9.4]{N}, there must be at least $c$ characters $\chi_1,\ldots,\chi_c$
in $B$, lying above $\theta_1,\ldots,\theta_c$, respectively. Now $H_0$ is
characteristic in $H$ (indeed, it is generated by all unipotent elements of $H$
by \cite[Rem.~1.5.13]{GM20}, as $\bH_0$ is simply connected since $\bH$ is a
Levi subgroup of the simply connected group $\bG$), so is stabilised by any
automorphism of $H$.
Then, if $\chi_i$ and $\chi_j$ are $\Aut(H)$-conjugate for $i\neq j$, then
$\theta_i$ and $\theta_j$ are $\Aut(H)$-conjugate and hence
$\Aut(H_0)$-conjugate, a contradiction. Hence, we will prove show, at least in
the cases~\ref{hyp1} and~\ref{hyp2}, that \eqref{eq:blockB} holds.

Now, in the case of \ref{hyp1}, we may without loss assume
$\bar k_{\Aut(H_1)}(B_1)\geq c$. Let $\chi_1,\chi_2\in\Irr(B_1)$ lie in
distinct $\Aut(H_1)$-orbits on $\Irr(B_1)$ if $H_i\not\in\fE$ or have
distinct degrees if $H_i\in\fE$. First assume that $H_1$ is not isomorphic to
any $H_i$ for $2\leq i\leq k$. Let $X:=H_2\times\cdots\times H_k$ and let
$\alpha\in\Aut(H_0)$.
By \cite{BCM06}, for $(h, 1_X)\in H_1\times X$ with $h\in H_1$ we have
$\alpha(h, 1_X)=(\alpha_1(h), \alpha_1'(h))$ where $\alpha_1\in\Aut(H_1)$ and
$\alpha_1'\in\Hom(H_1,\zent{X})$ are defined by~$\alpha$.
Now, assume $(\chi_1\otimes \vhi)^\alpha =\chi_2\otimes\vhi$, where
$\vhi\in\Irr(B_2\otimes\cdots\otimes B_k)$. Then considering elements of the
form $(h, 1_X)$ with $h\in H_1$, we see that
$\chi_1^{\alpha_1}(h)\theta(h)=\chi_2(h)$ for some linear character
$\theta\in\Irr(H_1)$.  (Namely, $\theta$ is the composition of $\alpha_1'$ with
the unique irreducible constituent of $\vhi|_{\zent{X}}$.)  If $H_1\not\in\fE$,
this contradicts that $\chi_1$ and $\chi_2$ are not conjugate under
$\Aut(H_1)$, since $\theta$ must be trivial.  If $H_1\in\fE$, it contradicts
that $\chi_1(1)\neq \chi_2(1)$.  Hence in either case $\chi_1\otimes \vhi$
cannot be $\Aut(H_0)$-conjugate to $\chi_2\otimes\vhi$.  This shows
$k_{\Aut(H_0)}(B')\geq c$. In the case that $H_0$ contains multiple isomorphic
copies of $H_1$, a similar argument holds, taking instead $X$ to be the
(possibly trivial) product of those $H_i$ such that $B_i\not\cong B_1$ under
this isomorphism and $\chi_1\otimes\cdots\otimes\chi_1$ (one for each copy of
$B_1$) in place of $\chi_1$. 

Arguing similarly, in case \ref{hyp2}, we obtain $k_{\Aut(H_0)}(B')\geq c$, and
in case \ref{hyp3} we obtain $k_{\Aut(H_0)}(B')\geq c-1$.

Now, assume we are in the situation of \ref{hyp3}, so $B'$ is cyclic but $B$ is
not, and $B'$ has positive defect.  Then $p$ divides $|H/H_0|$, and we let
$H_0\leq H_p\leq H$ be such that $H_p/H_0\in\Syl_p(H/H_0)$. 
Let $B_p$ be the (unique) block of $H_p$ above $B'$.  Then we have $c-1$
non-$\Aut(H)$-conjugate characters in $B_p$ lying above the
non-$\Aut(H_0)$-conjugate characters $\theta_1, \ldots, \theta_{c-1}$ of
$\irr{B'}$. We claim that there is at least one more character in $\irr{B_p}$
that is not $\Aut(H)$-conjugate to these.

Let $(\bH^\ast, F)$ be dual to $(\bH, F)$ and write $H^\ast:=\bH^{\ast F}$.
Let $s\in H^\ast$ be a semisimple $p'$-element such that
$\Irr(B)\subseteq\cE_p(H, s)$. Let $\bH_0\hookrightarrow \bH$ be the inclusion
map and let $s_0$ be the image of~$s$ under the induced dual epimorphism
$\bH^\ast\twoheadrightarrow\bH_0^\ast$.  Then
$\Irr(B')\subseteq \cE_p(H_0,s_0)$.  Then at least one of these $c-1$
characters, say $\theta_1$, can be assumed to lie in $\cE(H_0, s_0)$ using
\cite[Thm~9.12]{CE04}.

Let $\iota\colon\bH\hookrightarrow\wt{\bH}$ be a regular embedding, as in
\cite[(15.1)]{CE04} and write $\wt{H}:=\wt{\bH}^F$. Note that $\iota|_{\bH_0}$
is also a regular embedding of $\bH_0$ into $\wt{\bH}$ and we have
$\bH_0=[\wt{\bH}, \wt{\bH}]$ and $H_0\lhd H\lhd \wt{H}$.
Let $\wt{s}\in\wt{H}^\ast$ be a semisimple $p'$-element such that
$\iota^\ast(\wt{s})=s$. Then by \cite[Prop.~15.6]{CE04}, $\cE(H, s)$ is the set
of constituents of the restrictions to~$H$ of characters in $\cE(\wt{H},
\wt{s})$ and we may further define $\cE(H_p, s)$ to be the set of constituents
of restrictions of characters from $\cE(H, s)$ to $H_p$. Similarly, for $\wt x$
any semisimple element of $\wt H^\ast$ and $x=\iota^\ast(\wt x)$, we define
$\cE(H_p, x)$ to be the set of constituents of restrictions of characters from
$\cE(\wt H, \wt x)$ to $H_p$.

Then by \cite[Prop.~1.3]{CE99}, $\theta_1$ extends to a character $\chi$ in
$\cE(H_p, s)$, which is hence a member of $\irr{B_p}\cap \cE(H_p, s)$.  Then
$\chi \beta\in\irr{B_p}$ for every $\beta\in\irr{H_p/H_0}$. Now, recall that
characters of $\wt{H}/H_0$ are in bijection with elements of
$\zent{\wt{H}^\ast}$ (see \cite[(15.2)]{CE04}). We write $\hat{\wt z}$ for the
character of $\wt{H}/H_0$ corresponding to $\wt z\in\zent{\wt{H}^\ast}$.
Choose $\beta:=\hat{\wt z}$ for $1\neq \wt z\in\zent{\wt{H}^\ast}$ of $p$-power
order. Then $\chi\beta\in\cE(H_p, sz)$ where $z:=\iota^\ast(\wt z)$ and is not
$\wt{H}$-conjugate to $\chi$ by definition of $\cE(H_p,sz)$ and $\cE(H_p,s)$.
Together with the fact that $sz$ cannot be $H^\ast$-conjugate to $\vhi^\ast(s)$
for any automorphism $\vhi^\ast$ since $s$ is $p'$ and $z$ is a $p$-element,
this tells us that $\chi$ cannot be $\Aut(H)$-conjugate to $\chi \beta$ (see
\cite[Cor.~2.4]{NTT08}). 

Further, since $\chi \beta$ lies above $\theta_1$, it is not $\Aut(H)$-conjugate
to a character above $\theta_i$ for $i\neq 1$. Then letting $\chi_c\in \Irr(B)$
above $\chi \beta$ and $\chi_1,\ldots,\chi_{c-1}\in\Irr(B)$ above
$\chi, \theta_2,\ldots,\theta_{c-1}$, respectively, $\chi_1,\ldots,\chi_c$ are
non-$\Aut(H)$-conjugate members of $\Irr(B)$, as desired. 
\end{proof}

\begin{rem}
In our application of Proposition \ref{prop:hyponBi} and Lemma \ref{lem:redtoisol} below, we really only require $k_{\Aut(G)_H}(B)\geq c$, rather than $k_{\Aut(H)}(B)\geq c$.  Then we remark that since the automorphisms of $G$ respect the product structure of $H_0$, we could replace each $\bar{k}_{\Aut(H_i)}(B_i)$ with $k_{\Aut(H_i)}(B_i)$ in Hypothesis \ref{hyp:orbisol}, and then the statements of Proposition \ref{prop:hyponBi} and Lemma \ref{lem:redtoisol} hold with the condition $k_{\Aut(H)}(B)\geq c$ replaced with $k_{\Aut(G)_H}(B)\geq c$.
\end{rem}

In the notation above, note that if $B$ is quasi-isolated, then so is each
$B_i$. Indeed, let $(\bH^\ast, F)$ be dual to $(\bH, F)$ and write
$H^\ast:=\bH^{\ast F}$.  Let $\Irr(B)\subseteq\cE_p(H, s)$ for a quasi-isolated
semisimple $p'$-element $s\in H^\ast$.  Let $\bH_0=[\bH, \bH]$, and keep the
rest of the notation from the paragraph before Hypothesis \ref{hyp:orbisol}.
Note that the inclusion map $\bH_0\hookrightarrow \bH$ is a central isotypy in
the sense of \cite[Def.~2.A]{Bo05} and \cite[1.3.21]{GM20}, and so is the
induced dual epimorphism $\bH^\ast\twoheadrightarrow \bH_0^\ast$, by
\cite[1.7.11]{GM20}. Then if $s_0\in H_0^\ast$ is the image of $s$ under the
latter map, \cite[Prop.~2.3]{Bo05} yields that $s_0$ is also quasi-isolated. In
particular, any block $B'$ of $H_0$ covered by $B$ is quasi-isolated, and hence
so are the blocks $B_1,\ldots, B_k$. (Indeed, if $s_0$ is quasi-isolated, let
$s_0$ correspond to $\prod s_i$ under the isomorphism with $\prod \bH_i^{F_i}$.
If $s_i$ is not quasi-isolated in $\bH_i$ then neither is its preimage in the
corresponding $F$-simple factor of~$\bH_0$, in the notation of
\cite[1.5.14]{GM20}. But this would contradict that $s_0$ is quasi-isolated.)

\begin{rem}   \label{rem:maxquasiisol}
With this and Proposition \ref{prop:hyponBi}, note that if each
(quasi-isolated) $B_i$ satisfies 
\begin{equation}   \label{eq:condorbs}
\bar{k}_{\Aut(H_i)}(B_i)\geq \begin{cases} c& \hbox{ if $B_i$ is non-cyclic,}\\
  c-1 & \hbox{ if $B_i$ is cyclic,}\end{cases}
  \end{equation}
then our (quasi-isolated) block $B$ will satisfy $k_{\Aut(H)}(B)\geq c$. 
\end{rem}

Throughout, when $\bL$ is an $F$-stable Levi subgroup of $\bH$, we use $\RLH$
to denote Lusztig's twisted induction with respect to any parabolic subgroup
$\bP$ of $\bH$ containing $\bL$ as a Levi complement. In our situation, since
we may exclude groups considered in Proposition \ref{prop:sporadic} and
$\type{E}_7(2)$ and $\type{E}_8(2)$ since they have trivial outer automorphism
group, we have twisted induction is independent of the choice of $\bP$ by
\cite[Thm~3.3.8]{GM20}. Hence, as is customary, we
will suppress the parabolic subgroup from the notation.

\begin{lem}   \label{lem:redtoisol}
 As above, let $G:=\bG^F$ and let $p\nmid q$ be an odd prime. Assume that
 $k_{\Aut(H)}(B)\geq c$ for each $F$-stable Levi subgroup $\bH$ of $\bG$ and
 each non-cyclic quasi-isolated $p$-block $B$ of $\bH^F$ with $|\cd(B)|>1$. 
 Then $k_{\Aut(G)}(b)\geq c$ for each non-cyclic $p$-block $b$ of~$G$ such
 that $|\cd(b)|>1$.
\end{lem}

\begin{proof}
Let $b$ be a non-cyclic block of $G$ such that $|\cd(b)|>1$. If $b$ is
quasi-isolated, our assumption yields that the statement holds for $b$, with
$H=G$. Otherwise, by \cite{BR03}, $b$ is Morita equivalent to a quasi-isolated
block of a proper Levi subgroup $H:=\bH^F$ of $G$. In particular, this
Bonnaf\'e--Rouquier Morita equivalence is induced by the map $\RHG$.
Let $B$ be the Bonnaf\'e--Rouquier
correspondent for $b$ in~$H$. Note that by \cite[Thm~1.3]{KM13}, $B$ is
non-cyclic if and only if $b$ is. Further, by the character formula
\cite[Thm~8.16]{CE04} for $\RHG$, we have $|\cd(B)|>1$ if and only if
$|\cd(b)|>1$. 
 
Now, by the proof of \cite[Lemma 3.5]{MNS18}, we have
$\Aut(G)_b=\Inn(G)\Aut(G)_{H,B}$, and $\RHG$ is $\Aut(G)_{H,B}$-equivariant.
By assumption, $k_{\Aut(H)}(B)$, and hence $k_{\Aut(G)_{H,B}}(B)$, is at
least~$c$.  Then we have $k_{\Aut(G)_{H,B}}(b)\geq c$, proving that also
$k_{\Aut(G)}(b)\geq c$.
\end{proof}

%%%%%%%%%%%%%%%%%%%%%%%%%%%%%%%%%%%%%
\subsection{$e$-Harish-Chandra theory and blocks}
In this subsection, we allow $\bG$ to be any  Levi subgroup of a simple
algebraic group of simply connected type and $F:\bG\rightarrow\bG$ a Frobenius
endomorphism with respect to an $\FF_q$-rational structure. Thanks to the work
of Brou{\'e}--Malle--Michel \cite{BMM}, Cabanes--Enguehard \cite{CE99}, and
Kessar--Malle \cite{KM15}, we have a parametrisation of blocks of $G:=\bG^F$
in terms of $e$-Harish-Chandra theory. (See, e.g., \cite[\S3.5]{GM20} for the
notions of $e$-torus, $e$-split Levi subgroups and $e$-cuspidal characters
of $(\bG,F)$.)

Namely, by \cite[Thm~A]{KM15}, if $p\geq 3$ is a prime not dividing $q$ and $e$
is the order of $q$ modulo $p$, then 
there is a bijection from the set of $G$-conjugacy classes of $e$-Jordan
quasi-central cuspidal pairs $(\bL,\la)$ of $\bG$ with $\la\in\cE(\bL^F,p')$ to
the set of $p$-blocks of $G$. (See \cite[Def.~2.1, 2.12]{KM15} for the
definitions.)  We will write $b_G(\bL,\la)$ for the block corresponding to
$(\bL,\la)$. Then all irreducible constituents of $\RLG(\la)$ lie in
$b_G(\bL,\la)$ by \cite[Thm~A]{KM15}. The next lemma allows us to say more
about characters lying in $b_G(\bL,\la)$ and could be useful for other
applications. Here we write $d^1$ for the map on class functions given by
composition with the characteristic function on $p'$-elements of a group. 

\begin{lem}   \label{lem:p-elt}
 Let $b=b_{G}(\bL,\la)$ be a $p$-block of $G=\bG^F$ in $\cE_p(G,s)$, for
 $s\in\bL^{*F}$ a semisimple $p'$-element. Let $t\in \bZ(\bL^*)^F$ be a
 $p$-element. Then $\Irr(b)\cap\cE(G,st)\ne\emptyset$.
\end{lem}

\begin{proof}
By what we recalled above, all constituents of $\RLG(\la)$ lie in $\Irr(b)$, so
$0\ne d^1(\RLG(\la))$ has some non-zero constituent in~$b$. Let $\hat t$ denote
the linear character of $\bL^F$ corresponding to $t\in \bZ(\bL^*)^F$
(see \cite[Prop.~2.5.20]{GM20}). Since $d^1$ commutes with Lusztig induction
\cite[Prop.~3.3.17]{GM20}
and $|\hat t|=|t|$ is a $p$-power,
$$d^1(\RLG(\hat t\la))=\RLG(d^1(\hat t\la))=\RLG(d^1(\la))=d^1(\RLG(\la))\ne0,$$
so $d^1(\RLG(\hat t\la))$ has a component in $b$ as well, which means that
$\RLG(\hat t\la)$ has a constituent lying in $\Irr(b)$. But all constituents
of $\RLG(\hat t\la)$ are contained in $\cE(\bG^F,st)$ (see
\cite[Prop.~3.3.20]{GM20}).
\end{proof}

For $\bG$ a connected reductive group such that $[\bG,\bG]$ is simply
connected, we say that a block of $G=\bG^F$ is of \emph{quasi-central defect}
if it covers a block of $[\bG,\bG]^F$ that has a defect group contained in
$\zent{[\bG,\bG]^F}$. (In particular, when $\bG$ is simple, being of
quasi-central defect is equivalent to having central defect.)

\begin{cor}   \label{cor:ge2}
 Let $G=\bG^F$ be as above and let $b$ be a block in $\cE_p(G,s)$ with
 non-trivial defect. Then $\Irr(b)$ contains at least two characters not
 conjugate under $\Aut(G)$, where one of the characters lies in
 $\cE(G,s)$ and the other one outside of $\cE(G, p')$.
\end{cor}

\begin{proof}
First, by the result of Hiss \cite[Thm~9.12]{CE04} we have
$\Irr(b)\cap\cE(G,s)\neq\emptyset$. Now write $b=b_G(\bL,\la)$ with $\bL\le\bG$
an $e$-split Levi subgroup whose dual $\bL^*\le\bG^*$ contains $s$. If $\bL$ is
proper in $\bG$, it centralises a non-trivial $e$-torus of $\bG$ and thus its
dual centralises a non-trivial $e$-torus of $\bG^*$. Thus there is a
$p$-element $1\ne t\in \bZ(\bL^*)^F\le \bC_{G^*}(s)$, and
$\Irr(b)\cap\cE(G,st)\neq\emptyset$ by Lemma~\ref{lem:p-elt}. If $\bL=\bG$, the
quasi-central defect group $D$ of $b$ is normal in $G$, hence there exist
characters in $b$ non-trivial on $D$ (see e.g. \cite[Thm~9.4]{N}) which again
cannot lie in $\cE(G,p')$ (see e.g.~\cite[Prop.~1.2(v)]{CE99}). Arguing as
before, characters in $\cE(G,p')$ and $\cE(G,st)$ cannot be $\Aut(G)$-conjugate.
\end{proof}

Note that in the context of Theorem \ref{newmandi}, we are interested in the
case $c=3$ of Hypothesis \ref{hyp:orbisol}, and hence Corollary \ref{cor:ge2}
implies that we now only need to deal with non-cyclic blocks.

The next observation deals with unipotent blocks, and may be of interest for
other applications. For this statement, we relax the assumption that $p$ is
odd, and if $p=2$ we define $e$ to be the order of $q$ modulo $4$.
We remark that for $\eps\in\{\pm1\}$, we write $\PSL_n(\eps q)$ for the group
$\PSL_n(q)$ when $\eps=+1$ and $\PSU_n(q)$ when $\eps=-1$, with analogous
conventions for related groups $\SL_n(\eps q)$, $\PGL_n(\eps q)$, and
$\GL_n(\eps q)$.

\begin{lem}   \label{lem:unipsdistinctrpart}
 Let $\bG$ be a connected reductive group in characteristic $r$ such that
 $[\bG,\bG]$ is simply connected, and let $F\colon \bG\rightarrow\bG$ be a
 Frobenius endomorphism. Let $p$ be a prime and let $b$ be a unipotent $p$-block
 of $G:=\bG^F$ that is not of quasi-central defect. Then $\Irr(b)$ contains two
 unipotent characters whose degrees have different $r$-parts.
 In particular, these cannot be $\Aut(G)$-conjugate.
\end{lem}

\begin{proof}
We first assume that $\bG=[\bG,\bG]$ is simple of simply connected type.
Suppose $G$ is of exceptional type.
Let $e$ be as defined above and let $b=b_G(\bL,\la)$ for some unipotent
quasi-central $e$-cuspidal pair $(\bL,\la)$ (see
\cite[Thm~A]{KM15}). Note that our assumption $b$ has non-central defect
groups and $\bG$ is simple means we have $G\neq L:=\bL^F$. 
Then \cite[Sec.~13.9]{Carter} and
\cite[Tab.~1 and~2]{BMM} include the decomposition of $\RLG(\la)$ and the
relative Weyl group $W_G(\bL, \la)$ for many of the relevant blocks and
\cite[Tab.~3]{BMM} includes information about remaining cases where $\bL$ is a
torus and $W_G(\bL, \la)$ is non-cyclic.  By studying the characters in these
decompositions, whose degrees are available in \cite[Sec.~13.9]{Carter}, and at
times applying Ennola duality (see \cite[Thm~3.3]{BMM}), we see there are at
least two characters in
$\RLG(\la)$ whose degrees have distinct $r$-parts, except possibly
when $\bL$ is a torus.  In the latter situation, we have $\bL$ is the
centraliser of a Sylow $e$-torus, $e$ is regular for $\bG$, and $\la=1$. Then
$b$ is the principal block of $G$ and is the unique block containing
$p'$-degree unipotent characters.  Hence $b$ contains at least $1_G$ and
$\mathrm{St}_G$, and the claim holds.

Now suppose $G$ is of classical type. Note that by \cite[Thm~(i)]{CE94}, we
may replace $G$ with a group of the same rational type. If $p=2$, we are done
taking $1_G$ and St$_G$, since there is a unique unipotent block in this case
(see \cite[Thm~21.14]{CE04}). So, we also assume $p\geq 3$.

First, let $G=\SL_n(\eps q)$.  Let $e'$ be the order of $\eps q$ modulo~$p$.
The unipotent blocks of $G$ are parametrised by $e'$-cores of partitions
of~$n$. Assume $b$ is labelled by the $e'$-core $\la$, with $n=e'w+|\la|$.
As $b$ has non-central defect, $w\ge1$.
Let $\mu_1=(\la_1+e'w,\la_2,\ldots, \la_k)$ and
$\mu_2=(\la_1,\ldots,\la_k,1^{e'w})$, two partitions of $n$ labelling distinct
unipotent characters $\chi_1,\chi_2$ in~$b$. Then the degree formulas
\cite[Props.~4.3.1 and~4.3.5]{GM20} show that $\chi_1(1)$ and $\chi_2(1)$ have
distinct $r$-part.

Now let $G$ be one of $\Sp_{2n}(q)$, $\SO_{2n+1}(q),$ or $\SO_{2n}^\eps(q)$ and
let $e'$ be the order of $q^2$ modulo~$p$. Now the unipotent blocks of $G$ (or
in the case $G=\SO_{2n}^\eps(q)$, the blocks of $\GO_{2n}^\eps(q)$ lying above
unipotent blocks of $G$) are parametrised by $e'$-cores (if $p\mid (q^{e'}-1)$)
or cocores (if $p\mid (q^{e'}+1)$) of appropriate symbols
(see \cite[pp.~48--52]{BMM} and \cite[Thm]{CE94}).
Suppose $b$ (or a block above in $\GO_{2n}^\eps(q)$) is labelled by the
$e'$-core $(X;Y)$, with $n=e'w+\rnk(X; Y)$. (Here $\rnk(X; Y)$ is the rank of
the symbol, defined in \cite[p.~301]{GM20}.) Let $X=(x_1,\ldots, x_a)$ and
$Y=(y_1,\ldots,y_c)$ with $x_i<x_{i+1}$ and $y_i<y_{i+1}$ for each $i$.  
Again we have $w\geq 1$, and we may consider the two unipotent characters in
$\Irr(b)$ labelled by 
$$\begin{array}{r}(0,1,2,\ldots,e'w-1, x_1+e'w,x_2+e'w,\ldots,x_a+e'w;\\
    1, 2, \ldots, e'w, y_1+e'w,y_2+e'w,\ldots,y_c+e'w)\end{array}$$ and 
$$(x_1,\ldots,x_{a-1},x_a+e'w; Y).$$ 
Here the degree formula \cite[Prop.~4.4.7]{GM20} shows that these two
characters will have distinct $r$-part.

If $b$ is instead labelled by an $e'$-cocore, Olsson's process of $e'$-twisting
of symbols (\cite[p.~235]{ols86}) shows there is a bijection between
symbols $(X';Y')$ with $e'$-core $(X;Y)$ and symbols $(\wt{X}';\wt{Y}')$ with
$e'$-cocore $(\wt{X};\wt{Y})$, where $(\wt{X};\wt{Y})$ is the $e'$-twist of
$(X;Y)$. Further, $e'$-twisting does not change the entries of the symbol,
i.e., $(X';Y')$ and its $e'$-twist $(\wt{X}';\wt{Y}')$ satisfy
$X'\cup Y'=\wt{X}'\cup\wt{Y}'$. Hence from the formula
\cite[Prop.~4.4.7]{GM20}, we see that the $r$-part of the characters
corresponding to the symbols $(X';Y')$ and  $(\wt{X}';\wt{Y}')$ are the same,
and we are done with the case $\bG=[\bG,\bG]$ is simple.

In the general case, we have $b$ lies above a block $B$ of $[\bG,\bG]^F$ whose
defect groups are non-central. We may write $[\bG, \bG]^F$ as a direct product
of groups of the form $G_i=\bG_i^{F_i}$, where $\bG_i$ is simple of
simply connected type and $F_i$ is some power of $F$.  Then $B$ is a tensor
product of blocks $B_i$ of $G_i$, at least one of which, say $B_j$,
must have non-central defect groups.   Then from above, $\Irr(B_j)$ contains at
least two unipotent characters with different $r$-parts, and therefore the
statement also holds for $B$. Since the unipotent characters in~$b$ are
extensions of those in $B$, this completes the proof.
\end{proof}

%%%%%%%%%%%%%%%%%%%%%%%%%%%%%%%%%%%%%
\subsection{Type $\tA$}

We begin by considering the case of finite linear and unitary groups.  In this
subsection, we fix $\bar{S}=\PSL_n(\eps q)$, $G=\SL_n(\eps q)$,
$G^\ast=\PGL_n(\eps q)$, and $\wt{G}=\GL_n(\eps q)\cong\wt{G}^\ast$.

The blocks of $\wt{G}$ have been well-studied, with the parametrisation of the
blocks of $\wt{G}$ given in \cite{FS82} and a reduction to smaller-rank linear
and unitary groups given in \cite{MO83}.  We use these to prove that Theorem
\ref{newmandi} holds in the case that $\bar S=S/\zent{S}=\PSL_n(\eps q)$ and
that Hypothesis~\ref{hyp:orbisol} holds for $c=3$ and blocks of $G$.  Again in
this case, we will provide a slightly more general result.

\begin{prop}   \label{prop:TypeA}
 Let $G=\SL_n(\eps q)$ such that $\PSL_n(\eps q)$ is simple, and let $p$ be an
 odd prime with $p\nmid q$. 
 \begin{enumerate}[\rm(a)]
  \item Let $B$ be a $p$-block of $G$ with non-cyclic defect. Then
   $k_{\Aut(G)}(B)\geq 3$. 
  \item  Theorem~{\rm\ref{newmandi}} holds when $S/\zent{S}=\PSL_n(\eps q)$.
 \end{enumerate}
\end{prop}

\begin{proof}
Keep the notation from above.  Note that in this case,
$\wt{G}^\ast\cong\wt{G}=\GL_n(\eps q)$ and
$G^\ast=\PGL_n(\eps q)\cong \wt{G}/\zent{\wt{G}}$ and recall from the
discussion in Section \ref{sec:redisol} that
$\wt{G}\rtimes \cD$ induces all automorphisms on $\PSL_n(\eps q)$. From
Section~\ref{sec:othergroups}, we may assume that $G$ is the (non-exceptional)
Schur covering group for $\PSL_n(\eps q)$. Let $e'$ be the order of
$\eps q$ modulo $p$.

Note that in the situation of (b), we have $S\cong G/Z$ where $Z\leq \zent{G}$
is some central subgroup.  Further, writing $Z_p$ for the Sylow $p$-subgroup
of~$Z$, we may identify the blocks of $G/Z$ with blocks of $G/Z_p$, using
\cite[Thm~9.9]{N}.  Hence we may assume $Z=Z_p$ is a $p$-subgroup
in~ $\zent{G}$. Then it suffices to show (b) in the case that $e'=1$ with
$S=G/Z_p$ a quotient by some non-trivial central $p$-group and to show (a). 

Let $B$ be a non-cyclic block of $G$ and suppose $s$ is a semisimple
$p'$-element of $G^\ast$ such that $B$ lies in $\cE_p(G, s)$. In the context
of~(b), we further assume here that $\bar{B}$ is a block of $G/Z_p$ with
non-cyclic defect groups dominated by $B$, so that $\irr{\bar B}$ is comprised
of those members of $\irr{B}$ that are trivial on~$Z_p$. 

Applying Corollary \ref{cor:ge2}, it suffices for (a) to show that there are at
least two non-$\Aut(G)$-conjugate members of $\Irr(B)\cap\cE(G,s)$ or two
non-$\Aut(G)$-conjugate members of $\Irr(B)$ outside of $\cE(G,p')$.  For (b),
note that every member of $\cE(G,s)$ lies above the same character $\omega_s$
of $\zent{G}$, by \cite[11.1(d)]{bon06}, and that $\omega_s$ must be trivial on
$Z_p$ (and in fact on $\zent{G}_p\in\mathrm{Syl}_p(\zent{G})$) since $s$ is a
$p'$-element. In particular, every member of $\Irr(B)\cap\cE(G,s)$ may be
viewed as a character of $\bar{B}$.  Hence for~(b), it suffices to show that
there are at least two non-$\Aut(G)$-conjugate members of
$\Irr(B)\cap\cE(G,s)$, or that there are at least two non-$\Aut(G)$-conjugate
members of $\Irr(B)$ lying outside of $\cE(G,p')$ and trivial on $\zent{G}_p$. 

Let $\wt{s}\in\wt{G}^\ast\cong \wt{G}$ be a semisimple $p'$-element mapping to
$s$ under the natural epimorphism $\wt{G}^\ast\twoheadrightarrow G^\ast$, and
let $\wt{B}$ be a block of $\wt{G}$ in $\cE_p(\wt{G},\wt{s})$ covering $B$.
(Indeed, such a setup exists using \cite[Thm~9.12 and Prop.~15.6]{CE04}.) Let
$\wt{D}$ and $D$ be defect groups of $\wt{B}$ and $B$, respectively,
satisfying $D=\wt{D}\cap G$. Now,
\[\cent{\wt{G}^\ast}{\wt{s}}
  \cong \prod_{i=1}^k \GL_{m_i}((\eps q)^{\delta_i}),\]
where $\delta_i$ and $m_i$ are positive integers. (See \cite[Sec.~1]{FS82} for
details.) Note that 
$q-\eps$ divides $(\eps q)^{\delta_i}-1$ for each~$i$.  Write
$G_i:=\GL_{m_i}((\eps q)^{\delta_i})$ for $1\leq i\leq k$.

According to \cite[Thm~(7A)]{FS82}, Jordan decomposition maps
$$\Irr(\wt{B})\cap\cE(\wt{G},\wt{s})$$ 
to 
$$\Irr(B')\cap \cE(\cent{\wt{G}^\ast}{\wt{s}}, 1),$$ 
where
$B'=\prod_{i=1}^k B_i$ is some unipotent block of $\cent{\wt{G}^\ast}{\wt{s}}$
with $B_i$ a unipotent block of $G_i$. 
Let $e_i$ be the order of $(\eps q)^{\delta_i}$ modulo $p$.  Then $\wt{B}$
(and $B'$) is determined by certain $e_i$-core partitions $\la_i$.  Namely,
writing $m_i:=e_iw_i+|\la_i|$ for each $1\leq i\leq k$, \cite[Thm~(1.9)]{MO83}
yields that the number of characters in $\irr{\wt{B}}$ and the defect group of
$\wt{B}$ are the same as for the principal block of
$\prod_{i=1}^k \GL_{e_iw_i}((\eps q)^{\delta_i})$. 

In particular, note that if some $B_i$ has non-central defect, then $B_i$, and
hence $B'$, contains at least two unipotent characters of distinct degree, by
Lemma~\ref{lem:unipsdistinctrpart}.  In this case, $\Irr(\wt{B})\cap
\cE(\wt{G}, \wt{s})$ contains at least two characters of distinct degree, which
must lie above two characters in $\Irr(B)\cap \cE(G,s)$ lying in distinct
$\Aut(G)$-orbits.

To complete the proof of~(a) and~(b) it therefore suffices to assume that each
$B_i$ has central defect, which forces $m_i=e_iw_i=1$ for each $1\leq i\leq l$,
where we without loss assume that $B_i$ has trivial defect groups for
$l<i\leq k$.  Here note that $G_i\cong C_{(\eps q)^{\delta_i}-1}$ is cyclic for
$1\leq i\leq l$ and a defect group for $\wt{B}$ is isomorphic to a Sylow
$p$-subgroup of $\prod_{i=1}^l C_{(\eps q)^{\delta_i}-1}$. Since $\wt{B}$ is
non-cyclic, note that $l\geq 2$.

Suppose $\chi_i\in\Irr(B)$ lies under members of $\cE(\wt{G},\wt{s}\wt{t}_i)$,
$i=1,2$, for $p$-elements $\wt{t}_i\in\cent{\wt{G}^\ast}{\wt{s}}$, and assume
$\chi_1$ is $\Aut(G)$-conjugate to $\chi_2$.  Then $(\wt{s}\wt{t}_1)^\alpha$
is $\wt{G}^\ast$-conjugate to $\wt{s}\wt{t}_2z$ for some
$z\in\zent{\wt{G}^\ast}$ and $\alpha\in\cD$.  It follows that $\wt{t}_1^\alpha$
is conjugate to $\wt{t}_2z_p$, and hence they share the same set of
eigenvalues, where $z_p$ is the $p$-part of $z$. To complete the proof of~(a),
we therefore aim to exhibit $\wt{t}_1$ and $\wt{t}_2$ such that this cannot be
the case.

When $l\geq 3$, elements of $\wt{G}^\ast$ corresponding to $(x,1,\ldots,1)$ and
$(x,y, 1,\ldots,1)$ for $p$-elements $1\neq x\in G_1$, $1\neq y\in G_2$ have
different multiplicities for the eigenvalue 1, and hence we obtain
$\wt t_1, \wt t_2$ satisfying $\wt t_1^\alpha$ is not $\wt{G}^\ast$-conjugate
to $\wt t_2 z$ for any $z\in\zent{\wt{G}^\ast}$ and $\alpha\in \cD$.

Now suppose $l=2$. If $p\nmid(q-\eps)$, then $p$-elements corresponding to
$(x,1)$, $(x,y)$ with $1\neq x\in G_1$, $1\neq y\in G_2$ again satisfy the
claim, since $p\nmid |\zent{\wt{G}^\ast}|$.  Hence we assume that
$p\mid (q-\eps)$.  Note that this forces $p\mid ((\eps q)^{\delta_i}-1)$ for
each $i=1,2$.

If $p^2\mid((\eps q)^{\delta_1}-1)$, say, then elements $\wt t_1, \wt t_2$
corresponding to $(x,1)$ and $(x^p,1)$ with $|x|=p^2$ have the desired
property. We therefore can assume that $p\mid\mid((\eps q)^{\delta_1}-1)$, so
that also $p\mid\mid(q-\eps)$, in which case $D=\wt{D}\cap G$ is cyclic,
and we are done with~(a).

To complete the proof of (b), we wish to show that such $\wt{t}_1$ and
$\wt{t}_2$ exist such that $\omega_{\wt{t}_i}$ are further trivial on
$\zent{\wt{G}}_p$. For this, it suffices to find $\wt{t}_1$ and $\wt{t}_2$
lying in $[\wt{G}^\ast,\wt{G}^\ast]\cong G$.  Recall that here
$p\mid \gcd(n,q-\eps)$. In this case, note that $e_i=1$ for $1\leq i\leq k$
since $(q-\eps)\mid ((\eps q)^{\delta_i}-1)$, so that $B_i$ is the unique
block in $\cE(G_i, 1)$ for each $1\leq i\leq k$, and hence $\wt{B}$ is the
unique block in $\cE_p(\wt{G}, \wt{s})$.  This also forces $l=k$, and
$\cE(\wt{G}, \wt{s})$ contains only one character.

Recall that $\wt{D}$ is isomorphic to a Sylow $p$-subgroup of the (in this case
abelian) group $\cent{\wt{G}^\ast}{\wt{s}}$ and in fact by
\cite[Thm~(3C)]{FS82}, if we identify $\wt{G}$ with $\wt{G}^\ast$, we may take
$\wt{D}$ to be a Sylow $p$-subgroup of $\cent{\wt{G}^\ast}{\wt{s}}$. This way,
we also identify $G$ with
$[\wt{G}^\ast, \wt{G}^\ast]$ and have $D=\wt{D}\cap G$ is viewed as a subgroup
of $\cent{\wt{G}^\ast}{\wt{s}}\cap[\wt{G}^\ast,\wt{G}^\ast]$.  Let
$\bar{D}=D/Z_p$ be a defect group for $\bar B$.  It suffices to argue that
$D/\zent{G}_p$ contains at least two non-trivial non-$\Aut(G)$-conjugate
elements when $|\cd(B)|>1$. (Indeed, this would yield $p$-elements of
$\cent{\wt{G}^\ast}{\wt{s}}$ that lie in $G\cong[\wt{G}^\ast,\wt{G}^\ast]$ and
are not $\Aut(G)$-conjugate to $\zent{\wt{G}^\ast}_p$-multiples of one
another.) Hence we may assume $D/\zent{G}_p$ is elementary abelian, as
otherwise two non-identity elements of $D/\zent{G}_p$ with distinct orders
satisfy the claim.

If $k\geq 4$, then $D/\zent{G}_p$ is generated by at least two elements, and as
before we may find such elements whose eigenvalue structures do not allow them
to be $\Aut(G)$-conjugate.
 
If $k=2$ and $\delta_1=\delta_2$, then this forces $p\mid \delta_1$, since
$p\mid n=2^c\delta_1$ for some positive integer~$c$. Further, note that since
$D$ is not cyclic, we know $((\eps q)^{\delta_1}-1)_p>(q-\eps)_p$.  Then
choosing the embeddings into $\wt{G}$ of elements of the form $(y, y^{-1})$
with $|y|=((\eps q)^{\delta_1}-1)_p$ and $(x,1)$ with $|x|=(q-\eps)_p$, we see
these elements lie in $G$ (since $(x,1)$ is embedded into $\wt{G}$ with
eigenvalues
$$(x,x^{\eps q},\ldots,x^{(\eps q)^{\delta_1-1}},1,\ldots,1)
  =(x,\ldots,x,1,\ldots,1),$$ 
and therefore has determinant 1) and cannot be $\Aut(G)$-conjugate to
$\zent{G}_p$-multiples of one another.  Hence we may assume that
$\delta_1>\delta_2$.  Then note that, again studying the embedding of elements
$(s_1, s_2)\in C_{(\eps q)^{\delta_1}-1}\times C_{(\eps q)^{\delta_2}-1}$ into
$\wt{G}^\ast$, we see that $\wt{s}$ cannot be $\wt{G}^\ast$-conjugate to
$\wt{s}z$ for any $1\neq z\in\zent{\wt{G}^\ast}$. Hence the unique character of
$\cE(\wt{G}, \wt{s})$ restricts irreducibly to $G$, and similarly every member
of $\cE(\wt{G}, \wt{s}\wt{t})$ for $\wt{t}\in\cent{\wt{G}^\ast}{\wt{s}}$ a
$p$-element also restricts irreducibly.  Since
$\cent{\wt{G}^\ast}{\wt{s}\wt{t}}=\cent{\wt{G}^\ast}{\wt{s}}$ for each such
$\wt{t}$, we see every element of $\Irr(\wt{B})$, and hence of $\Irr(B)$, will
have the same degree, a contradiction.
 
So, we finally assume $k=3$.  Arguing like above, we may assume
$\delta_1=\delta_2=\delta_3$, so $\cent{\wt{G}^\ast}{\wt{s}}\cong G_1^3$. If
$p\mid\mid ((\eps q)^{\delta_1}-1)$, then $p\mid\mid(q-\eps)$, so we have
$Z_p=\zent{G}_p$ and $\bar{D}$ is cyclic, a contradiction. So we see
$p^2\mid ((\eps q)^{\delta_1}-1)$.  But then we also have $p^2\mid |\zent{G}_p|$, since otherwise $D/\zent{G}_p$ contains elements of order $p^2$,
contradicting our assumption that $D/\zent{G}_p$ is elementary abelian. In
particular, $p^2\mid n$.  Then here since $n=2^c\cdot 3\delta_1$, we have
$p\mid\delta_1$.  Hence we are done by considering elements $(y,y^{-1},1)$ and
$(x,1,1)$ with $x,y$ analogous to the case $k=2$ above.
\end{proof}

%%%%%%%%%%%%%%%%%%%%%%%%%%%%%%%%%%%%%
\subsection{Other classical groups}

For $(\bG, F)$ a connected reductive group and Frobenius morphism $F$ emitting
an $\FF_q$-rational structure and $s\in \bG^\ast$ semisimple, we write
\[J_s^\bG\colon \ZZ\cE(\bG^F, s)\rightarrow \ZZ\cE(\cent{\bG^\ast}{s}^F, 1)\]
for a Jordan decomposition as in \cite[Thm~2.6.22]{GM20}.
Recall from \cite[Prop.~3.3.20]{GM20} that for $\bK\leq \bG$ an $F$-stable Levi
subgroup with dual $\bK^\ast\le\bG^*$ and $s\in \bK^\ast$, Lusztig induction
induces a map $\hc_\bK^\bG\colon\ZZ\cE(\bK^F, s)\rightarrow\ZZ\cE(\bG^F, s)$.

We next consider the case of other classical groups, specifically for
quasi-isolated blocks.

\begin{prop}   \label{prop:typesBCD}
 Let $\bG$ be simple of simply connected type such that $G=\bG^F$ is
 quasi-simple of type $\tB_n(q)$ with $n\geq 2$, $\type{C}_n(q)$ with
 $n\geq 3$, $\tD_n(q)$ with $n\geq 4$ or $\tw{2}\tD_n(q)$ with
 $n\geq 4$. Let $p$ be an odd prime not dividing $q$. Then if $B$ is a
non-cyclic quasi-isolated $p$-block of $G$, we have $k_{\Aut(G)}(B)\geq 3$. 
\end{prop}

\begin{proof} 
As before, let $\iota\colon \bG\hookrightarrow\tbG$ be a regular embedding,
even taking the embedding as in \cite[Ex.~1.7.4]{GM20}, and
write $\wt{G}=\tbG^F$. Let $B$ be a non-cyclic quasi-isolated block in
$\cE_p(G, s)$ and let $\wt{B}\subseteq \cE_p(\wt{G}, \wt{s})$ be a block of
$\wt{G}$ lying above $B$, where
$\iota^\ast(\wt{s})=s$ is a quasi-isolated $p'$-element of $G^\ast$. 
Let $e$ be the order of $q$ modulo $p$. By what we recalled from \cite{KM15},
let $(\bL, \la)$ be the $e$-Jordan cuspidal pair of $\tbG$ such that $\wt{B}$
is the unique block of $\wt{G}$ containing all constituents of
$\hc_\bL^{\tbG}(\la)$.

Note that we may assume $\wt{s}\in\bL^\ast$ for a dual Levi subgroup
$\bL^\ast\le\tbG^\ast$ of $\bL$. Then since $F$ is a Frobenius endomorphism
and $\cent{\tbG^\ast}{\wt{s}}$ is connected and has only components of
classical type, by \cite[Thm~3.3.7]{GM20} the Mackey formula holds, and by
\cite[Thm~4.7.2]{GM20},
\begin{equation}\label{eqn:jordlus}
  J_{\wt{s}}^{\bK}\circ\hc_{\bK_1}^{\bK}=
  \hc_{\cent{\bK_1^*}{\wt{s}}}^{\cent{{\bK}^*}{\wt{s}}}\circ J_{\wt{s}}^{\bK_1},
\end{equation}
for any $F$-stable Levi subgroups $\bK_1^\ast\leq \bK^\ast\leq \tbG^\ast$
containing $\wt{s}$, so in particular for $\bK_1^*=\bL^*$. 

Also note that $\cent{\bL^\ast}{\wt{s}}$ is an $e$-split Levi subgroup of
$\cent{\tbG^\ast}{\wt{s}}$ (by \cite[Prop.~2.12]{Ho22}) and (by definition
of $e$-Jordan cuspidality), $(\cent{\bL^\ast}{\wt{s}}, J_{\wt{s}}^\bL(\la))$ is
an $e$-cuspidal pair.
Write $\bH^\ast:=\cent{\tbG^\ast}{\wt{s}}$ and
$\bM^\ast:=\cent{\bL^\ast}{\wt{s}}$, a Levi subgroup of $\bH^\ast$, and let
$\bH,\bM\le\tbG$ be dual to
$\bH^\ast$ and $\bM^\ast$, respectively. Then $(\bM, \psi)$, with
$\psi:= J_1^{\bM^\ast}\circ J_{\wt{s}}^\bL(\la)$, is an $e$-Jordan cuspidal
pair for $\bH$ defining a unipotent block $b$ of $H:=\bH^F$. 
Then by \cite[Thm~3.11]{BMM} and \cite[Thm~A]{KM15}, we have
$\irr{b}\cap \cE(H, 1)$ is the set of constituents of $\hc_\bM^\bH(\psi)$.
From this, we see that Jordan decomposition induces an injection
$$\irr{b}\cap \cE(H, 1)\hookrightarrow \irr{\wt{B}}\cap \cE(\wt{G}, \wt{s}).$$

Let $t\in\cent{\wt{G}^\ast}{\wt{s}}=(\bH^\ast)^F$ be a $p$-element and let
$\wt{\bG}(t)\leq\wt{\bG}$ and $\bH(t)\leq\bH$ be $F$-stable Levi subgroups in
duality with $\cent{\wt{\bG}^\ast}{t}$ and $\cent{\bH^\ast}{t}$, respectively.
Applying \cite[Thm (iii)]{CE94} to both $\wt{B}$ and $b$ and applying
\eqref{eqn:jordlus} to $(\wt{\bG},\bG(t),\wt{s})$ as well as to its analogue
for $(\bH^\ast,\bH(t)^\ast,1)$, we obtain further that
$\Irr(\wt{B})\cap\cE(\wt{G},\wt{s}t)$ is non-empty if and only if
$\Irr(b)\cap \cE(H,t)$ is non-empty.  More specifically,
$\chi\in\cE(\wt{G},\wt{s}t)$ is in $\irr{\wt{B}}$ if and only if
$\psi\in\cE(H, t)$ lies in $\irr{b}$, where $\psi$ is the character such that
$J_{\wt{s}t}^{\wt{\bG}}(\chi)=J_{t}^{\bH}(\psi)$.
Further, by \cite[Prop.~5.1]{CE99}, $b$ and $\wt{B}$ have isomorphic defect
groups. Let $D_b\cong D_{\wt{B}}$ and $D_B$ be defect groups of $b, \wt{B}$
and~$B$, respectively, chosen so that $D_B=D_{\wt{B}}\cap G$.

Now, by \cite[(2.2)]{Bo05}, we have
$\iota^\ast\left(\bH^\ast\right)=\cent{\bG^\ast}{s}^\circ$. Since $B$ is
quasi-isolated, the latter has only classical components, of the forms listed
in \cite[Tab.~II]{Bo05}. Hence the components of $\bH$ are of the forms dual
to these.

By \cite[Thm~(i)]{CE94}, the unipotent characters in $b$ are independent of
isogeny type.  Let $H_0:=[\bH,\bH]^F$ and let $b'$ be the unipotent block of
$H_0$ below $b$ and $b''$ the corresponding unipotent block of
$[\bH, \bH]_\SC^F$, where $[\bH, \bH]_\SC$ denotes the simply
connected covering of the semisimple group $[\bH, \bH]$. Applying
Corollary~\ref{cor:ge2} and Lemma~\ref{lem:unipsdistinctrpart}, we may assume
$\zent{[\bH,\bH]_\SC^F}_p$
is a defect group of $b''$, and each of $b,b'$ and $b''$ contain exactly one
unipotent character. Assume $[\bH, \bH]_\SC^F$ has a factor of type
$\type{A}_m(\eps q^\delta)$ for some $m,\delta\geq 1$ and $\eps\in\{\pm1\}$
with $d_p(\eps q^\delta)=1$, where $d_p(\eps q^\delta)$ denotes the order of
$\eps q^\delta$ modulo $p$. Then, $\type{A}_m(\eps q^\delta)$ has
a unique unipotent $p$-block, which must contain the trivial and Steinberg
characters, contradicting that $b$ contains a unique unipotent character.
Thus, since all components of $\bH$ are classical, $\zent{[\bH, \bH]_\SC^F}_p$
is trivial, so $p\nmid |\zent{H_0}|$, and we obtain that $D_b=\zent{H}_p$ from
\cite[Thm~(ii)]{CE94}. Then we see using \cite[Lemma~4.16, Def.~4.3]{CE99} that
$D_{\wt{B}}$ can be taken to be equal to $D_b=\zent{H}_p$.

Since $\wt{G}/G\zent{\wt{G}}$ is a $2$-group, we know $\zent{\wt{G}}_p\cap G=1$, so $D_{\wt{B}}$ is a direct product of $D_B$ and $\zent{\wt{G}}_p$, so $D_B\cong D_{\wt{B}}/\zent{\wt{G}}_p$ and $\zent{H}_p/\zent{\wt{G}}_p$ must be non-cyclic.
However, from our list of possible structures in \cite[Tab.~II]{Bo05}, we see
that $\zent{H}_p/\zent{\wt{G}}_p$ is trivial or cyclic, unless $\bG$ is of type
$\type{D}_n$ and $\bH$ is of type $\type{A}_{n-3}$. In this case, although $s$
is quasi-isolated, $\wt{s}$ is not, and $\bH$ is a Levi subgroup of
$\wt{\bG}$.  From our discussions above and using, for example, the
descriptions of possible $(\cent{\bG^\ast}{s}^\circ)^F$ in
\cite[Table~2]{mallecusp}, we have $H_0=\type{A}_{n-3}(\eps q)$ with
$d_{p}(\eps q)\neq 1$, and $\zent{H}_p/\zent{\wt{G}}_p\cong C_{(q+\eps)_p}^2$
coming from a torus $C_{q+\eps}^2\leq H$.

Let $\theta$ be the unique unipotent character in $b$.
%Then since unipotent
%characters are trivial on the center, we have  $\theta$ is the canonical
%character of $b$ by \cite[Thm~9.12, pp. 204]{N}. Since $D_b$ is central, we
%have $\Irr(b)=\{\theta\xi\mid \xi\in\Irr(\zent{H}_p)\}$ in the notation of
%\cite[Thm~9.12]{N}, but is also 
We have $\Irr(b)=\{\theta\hat{t}\mid t\in \zent{\bH^\ast}^F_p\}$, where
$t\mapsto \hat t$ is the isomorphism $\zent{\bH^\ast}^F\rightarrow\Irr(H/H_0)$
guaranteed by \cite[Prop.~2.5.20]{GM20}. Now, a construction for
$\wt{\bG}\cong\wt{\bG}^\ast$ is presented in \cite[Ex.~1.7.4]{GM20} using tori
constructed in \cite[Ex.~1.5.6]{GM20}. The element $s$ and its centralizer in
$G^\ast$ are described in terms of root systems in \cite{Bo05}. With this, we
can see $\zent{\bH^\ast}$ through this construction, and we calculate that
there are elements $t_1,t_2$ of order $p$ in
$\zent{\bH^\ast}^F\setminus\zent{\wt{G}^\ast}$ such that $t_1$ is not
$\wt{G}^\ast\cD$-conjugate
to $t_2z$ for any $z\in \zent{\wt{G}^\ast}_p$. Since $\wt{G}\rtimes \cD$
induces $\Aut(G)$ and \cite[Thm~3.1]{CS13} tells us Jordan decomposition
for $\wt{G}$ can be chosen to be $\cD$-equivariant, we then obtain as in the
proof of Proposition \ref{prop:TypeA} two non-$\Aut(G)$-conjugate characters in
$\wt{B}\setminus\cE(\wt{G}, p')$ lying above non-$\Aut(G)$-conjugate members of
$\Irr(B)$.
\end{proof}

%%%%%%%%%%%%%%%%%%%%%%%%%%%%%%%%%%%%%
\subsection{Remaining quasi-isolated blocks}

\begin{prop}   \label{prop:ExceptTypes}
 Let $\bG$ be such that $G:=\bG^F$ is quasi-simple of type $\tE_n(q)$ with
 $6\leq n\leq 8$, $\tw{2}\tE_6(q)$, $\tF_4(q)$ or $\type{G}_2(q)$. Let $p$
 be an odd prime not dividing $q$. Then if $B$ is a non-cyclic quasi-isolated
 $p$-block of $G$, we have $k_{\Aut(G)}(B)\geq 3$.
\end{prop}

\begin{proof}
Let $B$ lie in $\cE_p(G, s)$, where $s$ is a quasi-isolated semisimple
$p'$-element of $G^\ast$. Let $e$ be the order of $q$ modulo $p$. In our
situation, thanks to \cite[Rem.~2.2]{KM15}, the notions of $e$-Jordan
cuspidality and $e$-cuspidality coincide. When $p$ is good for $G$, the
$e$-Jordan quasi-central cuspidal pairs are the same as the
$e$-Jordan cuspidal pairs by \cite[Thm~A(d,e)]{KM15}.

Assume $(\bG,\chi)$ is quasi-central $e$-cuspidal. Then since
$\bG=[\bG, \bG]$, \cite[Prop.~2.5]{KM13} implies that $b_G(\bG,\chi)$ is of
central defect. However, then $b_G(\bG,\chi)$ is cyclic.

Hence from now on, we may assume $B=b_G(\bL,\la)$ with $(\bL,\la)$ an
$e$-Jordan quasi-central cuspidal pair of $G$ with $G\neq\bL^F$ and
$\la\in\cE(\bL^F,s)$.  Further, from Lemma~\ref{lem:unipsdistinctrpart} and
Corollary~\ref{cor:ge2}, we may assume $B$ is not unipotent. From here, we
consider separately the cases that $p$ is bad and that $p$ is good.

First let $p\geq 3$ be bad for $\bG$.
Here the block distributions for $\cE(G,s)$ are given in
\cite[Tab.~2--4,6--9]{KM13}, or by Ennola duality with those results.
(Namely, when $e=2$, one formally replaces $q$ with $-q$ in the results for
$e=1$ --- see \cite[p.~16]{KM13}; also, the results for $\tw{2}\type{E}_6(q)$
are obtained from those of $\type{E}_6(q)$ by switching the roles of $e=1$ and
$e=2$ --- see \cite[pp. 21]{KM13}.)
From this and the knowledge of the degrees of characters in $\cE(G,s)$ for
each $\cent{G^\ast}{s}$ listed, (obtained by Jordan decomposition from the
unipotent character degrees of groups of small rank), we see that
$B\cap \cE(G,s)$ contains at least two characters that are not
$\Aut(G)$-conjugate in the cases with non-cyclic defect. 
(Here we may use \cite[Prop.~2.7]{KM13} to understand the defect groups.)
Hence we are done in this case, by applying Corollary~\ref{cor:ge2}. 

Next assume $p\geq 5$ is a good prime for $\bG$. First suppose $e\geq 3$. In
this case, the block distributions for $\cE(G, s)$ are given in
\cite[Tab.~2, 3, 5, 7, 8]{Ho22}, through a description of the decompositions of
$\RLG(\la)$.
With this information, combined with the knowledge of the character degrees
in $\cent{G^\ast}{s}$ and again using Corollary \ref{cor:ge2} we see there
are at least three non-$\Aut(G)$-conjugate characters when the defect group is
non-cyclic. 

Finally, if $p\geq 5$ is a good prime for $G$ with $e\in\{1,2\}$, the
decompositions of $\RLG(\la)$ and the groups $W_G(\bL, \la)$ are the same as
those given in \cite[Tab.~2--4,6--9]{KM13} for the bad prime case. But, note
that now each $(\bL,\la)$ gives a distinct block by \cite[Thm~A]{KM15}, and the
same considerations as before complete the proof.
\end{proof}

\begin{cor}   \label{cor:isolsc}
 Let $\bG$ be simple of simply connected type such that $G:=\bG^F$ is
 quasi-simple. Let $p\nmid q$ be an odd prime and let $B$ be a quasi-isolated
 $p$-block of $G$. Then $k_{\Aut(G)}(B)\geq 3$ if $B$ is non-cyclic and
 $k_{\Aut(G)}(B)\geq 2$ if $B$ is cyclic of positive defect.
\end{cor}

\begin{proof}
This now follows from Corollary \ref{cor:ge2} and Propositions \ref{suzree},
\ref{prop:TypeA}, \ref{prop:typesBCD}, and \ref{prop:ExceptTypes}.
\end{proof}

\begin{prop}   \label{prop:isolatedmax}
 Let $\bG$ be simple of simply connected type such that $G:=\bG^F$ is
 quasi-simple, and let $p\nmid q$ be an odd prime. Let $\bH$ be an $F$-stable
 Levi subgroup of $\bG$ and let $B$ be a non-cyclic quasi-isolated $p$-block of
 $H:=\bH^F$ with $|\cd(B)|>1$. Then $k_{\Aut(H)}(B)\geq 3$.
\end{prop}

\begin{proof}
Using Corollary \ref{cor:isolsc} and Proposition \ref{prop:nonquasisimple}, we
have $H$ satisfies Hypothesis~\ref{hyp:orbisol} with respect to $c=3$ (see also
Remark~\ref{rem:maxquasiisol}).
Hence by Proposition~\ref{prop:hyponBi}, the statement holds.
\end{proof}

%%%%%%%%%%%%%%%%%%%%%%%%%%%%%%%%%%%%%
\subsection{Proof of Theorem B}

We are now ready to complete the proof of Theorem B. Recall that, thanks to
Section~\ref{sec:othergroups} and Proposition~\ref{prop:definingchar}, we only
need to consider groups of Lie type in non-defining characteristic. 

\begin{thm}   \label{thm:nondefiningchar}
 Theorem~{\rm\ref{newmandi}} holds when $S$ is quasi-simple such that
 ${S}/\zent{S}$ is a simple group of Lie type defined in characteristic
 $r\neq p$.
\end{thm}

\begin{proof}
By Section \ref{sec:othergroups}, we may assume that the simple group
$\bar{S}:=S/\zent{S}$ has a non-exceptional Schur multiplier, and that $S$ is
not of Suzuki or Ree type. 

Let $G$ be the Schur cover of $\bar S$, so that $G=\bG^F$ is the group of fixed
points of some simple, simply connected algebraic group $\bG$ in
characteristic~$r$ under some Frobenius morphism $F$. Assume first that
$p\nmid|\zent{G}|$. Then it suffices to show the statement for~$G$, since the
irreducible characters in a block of $S$ will be the same as those in the block
of $G$ dominating it, viewed via inflation.
As in the proof of Lemma~\ref{lem:redtoisol}, we have $b$ is Morita equivalent
to a quasi-isolated block $B$ of $H$, where $H=\bH^F$ for an $F$-stable Levi
subgroup $\bH$ of $\bG$. Now, by Proposition~\ref{prop:isolatedmax} together
with Lemma~\ref{lem:redtoisol}, we have $k_{\Aut(G)}(b)\geq 3$.

Now suppose that $p$ divides $|\zent{G}|$. Then since $p$ is odd, we have $G$
is $\SL_n(\eps q)$ or $\tE_6(\eps q)$ for some $\eps\in\{\pm1\}$ and some power
$q$ of $r$. In the first case, $p\mid \gcd(n, q-\eps)$ and
Proposition~\ref{prop:TypeA}(b) finishes the proof. In the second case, 
$p=3\mid (q-\eps)$.

So, we finally assume that $\bar S=\tE_6(\eps q)=G/\zent{G}$ and
$p=3\mid(q-\eps)$. Let $\bar{b}$ be a 3-block of $\bar S$ with non-cyclic
defect groups such that $|\cd(\bar{b})|>1$, and let $b$ be the block of $G$
dominating $\bar{b}$ lying in $\cE_3(G,s)$ for some semisimple $3'$-element
$s\in G^\ast$. Let $\bH$ be an $F$-stable Levi subgroup of $\bG$ minimal with
the property that $\cent{\bG^\ast}{s}\leq \bH^\ast$, so that $s$ is
quasi-isolated in $\bH^\ast$. Let $B$ be the block of $H=\bH^F$ in
Bonnaf\'e--Rouquier correspondence with $b$ and let $\bar{B}$ be its image in
$H/\zent{G}$. Then $\bar{B}$ is also not cyclic, by \cite[Thm~7.16]{KM13}.

Arguing exactly as in the fourth paragraph of \cite[p.~13]{KM17}, either
$\cent{\bG^\ast}{s}=\bC_{\bG^\ast}^\circ(s)=\bH^\ast$ consists only of
components of type $\tA$, or $\bH_0:=[\bH,\bH]$ is of type $\tD_4$ or $\tD_5$.

In the first case, $B$ is the tensor product of some unipotent block of $H$
with the linear character $\hat{s}$ corresponding to $s\in\zent{\bH^\ast}^F$. 
Let $B'$ be the unipotent $3$-block of $H_0$ covered by $B\otimes \hat{s}^{-1}$. Using Lemma \ref{lem:unipsdistinctrpart} or in fact checking directly for groups of type $\tA$ of rank at most~5, we see $B'$ contains at least two unipotent characters of distinct degree unless either $\cent{G^\ast}{s}$ is abelian or $H_0$ contains only components of the form $\tA_2(-\eps q)$. If $\cent{G^\ast}{s}$ is abelian, then every member of $\cE_3(H,s)$, hence $\Irr(B)$, and hence $\Irr(b)$, has the same degree. Now, the unipotent blocks of $\tA_2(\eps q)$ consist of one defect zero block and the principal block containing the trivial and Steinberg characters. Then if $H_0$ contains only components of the form $\tA_2(-\eps q)$ and $B'$ does not contain two characters of distinct degree, then $B'$ has trivial defect groups, which contradicts Proposition \ref{prop:hyponBi}(a) unless again $|\cd(B)|=|\cd(b)|=1$.

So we may assume $B'$ contains at least two unipotent characters of distinct
degree, so $\Irr(B)\cap \cE(H,s)$, and hence $\Irr(b)\cap \cE(G,s)$, contains
at least two characters of distinct degrees. Recall that the members of
$\Irr(b)\cap\cE(G,s)$ are trivial on $\zent{G}$, since $s$ is $3'$, and can
therefore be viewed as characters in $\Irr(\bar{b})$.
Further, the images of the characters in $\Irr(B)\cap\cE(H,s)$ under $d^1$ are
linearly independent by \cite[Thm~1.7]{CE99}. Hence there must be at least
one more member of $\Irr(\bar{B})$, and therefore at least one member of
$\Irr(\bar{b})$ lying outside of $\cE(G,s)$, which as before is not
$\Aut(G)$-conjugate to the members of $\cE(G,s)$. 

Now consider the case $\bH_0=[\bH, \bH]$ is of type $\tD_4$ or $\tD_5$. Note
that $H_0=\bH_0^F$ is simply connected of type $\tD_4$ or $\tD_5$, and hence
has centre of $2$-power order. Then a (quasi-isolated) block $B'$ of
$H_0\zent{G}/\zent{G}\cong H_0$ lying under $\bar{B}$ contains at least three,
respectively two, characters in distinct $\Aut(H_0)$-orbits from
Proposition~\ref{prop:typesBCD} and Corollary~\ref{cor:ge2} if $B'$ is
non-cyclic, respectively cyclic. If $B'$ is non-cyclic, then arguing exactly as
in the situation of Proposition~\ref{prop:hyponBi}(a) and
Lemma~\ref{lem:redtoisol}, but with $H_0\zent{G}/\zent{G}\lhd H/\zent{G}$
taking the place of $H_0\lhd H$, completes the proof. So, assume $B'$ is cyclic. Note that $\bH=\bH_0\bZ^\circ(\bH)$, and $H_0\bZ^\circ(\bH)^F/\zent{G}$ is normal in $H/\zent{G}$ with $3'$-index, so that a block $B'''$ of this group under $\bar{B}$ is non-cyclic. Further, this group can be identified with a $3'$-quotient of $H_0\times \bZ^\circ(\bH)^F/\zent{G}$, and hence we may identify $B'''$ with a block $B'\otimes B''$ of $H_0\times \bZ^\circ(\bH)^F/\zent{G}$, where $B''$ has non-trivial defect groups. Then taking two non-$\Aut(H)$-conjugate members of $B''$ (these exist since $B''$ is a tensor product of the unique block of the Sylow $3$-subgroup of $\bZ^\circ(\bH)^F/\zent{G}$ with some character of $3'$ order), we obtain $k_{\Aut(H)}(\bar{B})\geq 3$ and $k_{\Aut(G)}(\bar{b})\geq 3$, again arguing as in Propositions \ref{prop:hyponBi} and Lemma~\ref{lem:redtoisol}.
\end{proof}

%%%%%%%%%%%%%%%%%%%%%%%%%%%%%%%%%%%%%%%%%%%%%%%%%%%%%%%%%%%%%%%%%%%%%%%%%
\section{Invariant blocks and defect groups}\label{sec:ThmC}
This section is devoted to the proof of the following result on blocks of
quasi-simple groups (which will imply Theorem~C):

\begin{thm}  \label{thm:newg}
 Let $p$ be an odd prime and $S$ a quasi-simple group with $Z=\zent S$,
 $\bar S:= S/Z$. Let $b$ be a $p$-block of $S$ with an abelian, non-cyclic
 defect group $D$, and $b_D$ a block of $\cent SD$ with defect group $D$
 inducing $b$. For parts {\rm (a)} and {\rm (c)}, in the cases
 $\bar S = \tE_6(\eps q)$ with $p=3|(q-\eps)$ and $\eps \in \{\pm 1\}$, assume
 in addition that $BHZ$ for the prime~$p$ holds for all groups of order smaller
 than $|\bar S|$.
 \begin{enumerate}[\rm(a)]
  \item Suppose that $\bar S \le H \le \Aut(S)$, $A:=H/\bar S$ has a normal
   $p$-complement and a cyclic Sylow $p$-subgroup $Q$. Assume $b$ is
   $H$-invariant, and for every $\chi\in\irr b$ we have that $|A:A_\chi|$ is
   $p'$. If $x\in\norm HD$ is a $p$-element that fixes $b_D$, then $[x,D]=1$.
  \item Suppose that $\bar S\le H\le\Aut(S)$ and $\bO^{p'}(H/\bar S) =H/\bar S$.
   Also assume that every $\chi\in\irr b$ is $H$-invariant. Set
   $J:=\Inndiag(\bar S)$ if $S$ is of Lie type, and $J=\bar S$ otherwise.
   Then $HJ/J$ is a $p'$-group.
  \item Suppose $p \nmid |Z|$ and let $S/Z \leq K/Z \leq \Aut(S)$ with $K/S$
   an abelian $p$-group. If every irreducible character in $b$ extends
   to~$K$, then the defect groups of the $K$-block covering $b$ are abelian.
 \end{enumerate}
\end{thm}

%%%%%%%%%%%%%%%%%%%%%%%%%%%%%%%%%%%%%
\subsection{First reductions}   \label{sec:red}
We keep the notation of Theorem~\ref{thm:newg} throughout the section. In
particular, $p$ is always an odd prime. For the proof we discuss the various
possibilities for $S$ and $p$ according to the classification of finite simple
groups. By assumption $|D|>1$ and thus $p$ divides $|S|$.

\begin{lem}   \label{lem:red (a)}
 In the situation of Theorem~{\rm\ref{thm:newg}(a)} assume $Q$ is normal in
 $A$. Then we may assume $A=Q\ne 1$.
\end{lem}

\begin{proof}
If $Q$ is normal in $A$, all orbits of $Q$ on $\Irr(b)$ will have $p'$-length as
well, and all $p$-elements in $A$ lie in $Q$. So we are done if we can show the
claim when $A=Q$. Furthermore, if $Q=1$, any $p$-element $x\in\bN_H(D)$ that
fixes $b_D$ centralises~$D$ by Proposition \ref{useful}.
\end{proof}

\begin{lem}   \label{lem:red (c)}
 In the situation of Theorem~{\rm\ref{thm:newg}(c)} assume $K/S$ is cyclic.
 Then the claim in~{\rm \ref{thm:newg}(c)} is a consequence of
 {\rm \ref{thm:newg}(a)}.
\end{lem}

\begin{proof}
Since every irreducible character in $b$ extends to $K$, the block $b$ is
$K$-invariant. Let $B$ be the unique block of $K$ covering $b$. By
Lemma~\ref{bD}, let $\hat D$ be a defect group of $B$ such that
$\hat D \cap S=D$ and $\hat D \leq T$, where $T$ is the stabiliser of $b_D$ in
$\norm KD$. Let $x \in \hat D$ then $[x,D]=1$ by~(a), therefore
$D\leq\zent{\hat D}$. As $\hat D/D$ is cyclic, this shows $\hat D$ is abelian.
\end{proof}

\begin{prop}   \label{prop:red 1}
 For the proof of Theorem~{\rm\ref{thm:newg}} we may assume that $S/\bZ(S)$ is
 simple of Lie type in characteristic different from~$p$.
\end{prop}

\begin{proof}
If $\bar S$ is an alternating group, a sporadic group, or the Tits simple
group, then $|\Out(S)|$ is a $2$-power. Hence \ref{thm:newg}(a) follows from
Proposition~\ref{useful} (applied to $G=N=S$), while \ref{thm:newg}(b), (c)
hold trivially. The same arguments apply whenever $p\nmid |\Out(\bar S)|$.

Now assume $\bar S$ is simple of Lie type in characteristic~$p>2$. 
If $\bar S$ has an exceptional covering group (see \cite[Tab.~24.3]{MT}) then
$\Out(S)$ is a 2-group, and we can conclude as above. Hence we may assume $|Z|$
is prime to $p$ and $p$ divides $|\Out(S)|$. By \cite{Hum} any
$p$-block of $S$
has either full defect or defect zero. Thus our assumption on $D$ being abelian
forces $\bar S=\PSL_2(q)$ for some $q=p^f$, so we may take $S=\SL_2(q)$. This
group has two $p$-blocks of maximal defect, the principal block $B_0(S)$ and
a block $B$ containing all faithful characters (see \cite{Hum}).
Since $|\Out(S)|$ is divisible by $p$, $Q$ must induce field automorphisms of
order~$p$, so $p|f$. By order reasons, the image of $Q$ is central in
$\Out(S)$, so by Lemma~\ref{lem:red (a)} we may assume $A=Q$. 
Now by inspection of the generic character table given e.g.~in
\cite[Ex.~2.1.17 and Tab.~2.6]{GM20} there exist irreducible characters
in $B_0(S)$ as well as in $B$ not stabilised by $Q$. Hence
the theorem holds in this case.
\end{proof}

\begin{prop}   \label{prop:exc cov}
 Theorem~{\rm\ref{thm:newg}} holds for $S$ an exceptional covering group of a
 simple group of Lie type in characteristic different from $p$.
\end{prop}

\begin{proof}
The simple groups with exceptional covering groups are listed in
\cite[Tab.~24.3]{MT}. Arguing as in the proof of Proposition \ref{prop:red 1},
we need only consider the ones with $\Out(S)$ not a $p'$-group, which are
$\PSL_3(4)$, $\PSU_6(2)$, $\PSO_8^+(2)$, $\tw2B_2(8)$ and $\tw2\tE_6(2)$, and
the only relevant prime is $p=3$. Assume $\bar S=\PSU_6(2)$ or $\tw2\tE_6(2)$
and that $3$ divides $|H/\bar S|$, respectively $|K/S|$. Since any outer
automorphism of order $3$ permutes the three involutions in the Schur multiplier
of $\bar S$ cyclically and the block $b$ is invariant under $H$,
respectively $K$, the relevant covering groups are $S=\bar S$ and $S=3.\bar S$
only, which are not exceptional coverings.
In all other cases, using \cite{GAP}, the $3$-blocks of an exceptional covering
group $S$ turn out to have either cyclic or non-abelian defect, 
whence the claim follows.
\end{proof}

\begin{prop}   \label{prop:suzree}
 Theorem~{\rm\ref{thm:newg}} holds for $S$ a Suzuki or Ree group.
\end{prop}

\begin{proof}
By Proposition~\ref{prop:red 1}, we may assume $p$ is not the defining
characteristic of $S$ and $S\not\cong \tw2 F_4(2)'$.
Now the Sylow $p$-subgroups of the Suzuki and the small Ree groups are cyclic
for any such $p\ge3$. For $S=\tw2F_4(q^2)$ with $q^2>2$, the only $p$-block
with non-cyclic defect groups is the principal block $b=B_0(S)$, and its defect
groups are abelian when $p>3$ (see \cite{Ma90}). Now $\Out(S)$ is cyclic, so in
particular its Sylow $p$-subgroup is normal; also $J=\bar S$. By
Lemma~\ref{lem:red (a)} we may assume $A=Q$, and by Lemma~\ref{lem:red (c)} it
suffices to show~\ref{thm:newg}(a) and \ref{thm:newg}(b). But then the
only $p$-block of $H$ above $b$ is the principal block $B=B_0(H)$. By
\cite[Bem.~3]{Ma90} all characters in $b$ have height~0, and if all orbits of
$Q$ on $\Irr(b)$ have $p'$-size, the same is true for the characters in
$\Irr(B)$. But then by the main result of \cite{MN21}, the Sylow $p$-subgroups
of $H$ are abelian. Alternatively, by inspection of the character table
\cite{Ma90}, the assumption of \ref{thm:newg}(a) is in fact never satisfied for
$A \neq 1$; and this establishes~\ref{thm:newg}(a), as well
as~\ref{thm:newg}(b).
\end{proof}

%%%%%%%%%%%%%%%%%%%%%%%%%%%%%%%%%%%%%
\subsection{Some results on $p$-blocks} \label{sec:blocks}
To deal with groups of Lie type for non-defining primes we first observe some
general facts on blocks of finite reductive groups that may be of independent
interest.

Let $\bG$ be a Levi subgroup of a simple linear algebraic group of simply
connected type over an algebraically closed field of positive characteristic and
$F:\bG\to\bG$ a Frobenius endomorphism with respect to an $\FF_q$-rational
structure. Let $\bG^*$ be a group in duality with $\bG$ with corresponding
Frobenius endomorphism again denoted $F$. We let $p\ge3$ be a prime not dividing
$q$ and denote by $e$ the order of $q$ modulo $p$.

Recall that any $p$-block $b$ of $G:=\bG^F$ has the property that
$\Irr(b)\subseteq\cE_p(G,s)$ for some semisimple $p'$-element
$s\in G^*:=\bG^{*F}$ (see e.g. \cite[Thm~9.12]{CE04}). Furthermore, by
\cite[Thm~A]{KM15} there is a bijection between $p$-blocks $b$ of $\bG^F$ and
$\bG^F$-classes of $e$-Jordan-cuspidal pairs $(\bL,\la)$ of $\bG$ of
quasi-central $p$-defect, with $\la$ lying in a $p'$-Lusztig series of
$\Irr(\bL^F)$ such that all constituents of $\RLG(\la)$ are contained in
$\Irr(b)$. We write 
$$(\bL,\la)\mapsto b_G(\bL,\la)$$ 
for this map.

\begin{lem}   \label{lem:norm}
 Let $d\ge1$ and $\bL$ be a maximal proper $d$-split Levi subgroup of $\bG$.
 Then $|\bN_\bG(\bL)^F/\bL^F|$ is not divisible by a prime bigger than
 $\max\{2,d\}$.
\end{lem}

\begin{proof}
The maximal proper $d$-split Levi subgroups $\bL$ of $\bG$ above the
centraliser of a fixed Sylow $d$-torus $\bS$ of $\bG$ are in one-to-one
correspondence with the maximal parabolic subgroups of the relative Weyl group
$W=\bN_\bG(\bS)^F/\bC_\bG(\bS)^F$ of
$\bS$ \cite[Prop.~3.5.12]{GM20}. Moreover, if $\bL$ corresponds to
$W_1\le W$ then $\bN_\bG(\bL)^F/\bL^F\cong \bN_W(W_1)/W_1$ (see
\cite[Prop.~26.4]{MT}). The claim is thus reduced to a question in reflection
groups. For these, there is an immediate reduction to the irreducible case.
The latter can be checked case by case using the explicit description of the
various relative Weyl groups given in \cite[3.5.11--3.5.15]{GM20}. For
example, when $\bG$ is of classical type, then $W=G(m,1,n)$ or $G(m,2,n)$ for
suitable $m\in\{d,2d\}$ and $n\ge1$, for which the assertion is easily verified.
\end{proof}

The following somewhat surprising result may be of independent interest.
Here, by a \emph{field automorphism of $G$} we mean any conjugate of an
automorphism induced by a Frobenius endomorphism $F_0$ of $\bG$ commuting
with~$F$.
%(in the sense of \cite[2.5.13]{GLS3})

\begin{prop}   \label{prop:field}
 Let $\sigma$ be a field automorphism of $G=\bG^F$ of order~$p$, and let
 $\gamma=\sigma\tau$, where $\tau$ is an inner-diagonal automorphism of $G$. 
 Then any $p$-block of $G$ of non-quasi-central defect contains characters in
 a Lusztig series that is not $\gamma$-stable.
\end{prop}

\begin{proof}
Let $b$ be a $p$-block of $G$ of non-quasi-central defect. As above,
$b=b_G(\bL,\la)$ for some $e$-split Levi subgroup $\bL$ of $\bG$, proper
since $b$ has non-quasi-central defect. By assumption, up to conjugation,
$\sigma$ is induced by a Frobenius endomorphism $F_0$ of $\bG$ commuting
with~$F$, with
respect to an $\FF_{q_0}$-rational structure on $\bG$ where $q=q_0^p$. Then
$q_0$ also has order~$e$ modulo $p$. In particular any $e$-torus of $(\bG,F_0)$
is also an $e$-torus of $(\bG,F)$, and conversely, any $F_0$-stable $e$-torus
of $(\bG,F)$ is also an $e$-torus of $(\bG,F_0)$. The same relations hold
for $e$-split Levi subgroups since these are the centralisers of $e$-tori.

By \cite[5.6]{Tay18}, $\sigma$ induces a field automorphism $\sigma^*$ of
$G^*=\bG^{*F}$ of the same order. Let $\bL^*\le\bG^*$ be in duality with $\bL$,
an $e$-split Levi subgroup of $(\bG^*,F)$ that is $F_0$-stable, hence $e$-split
in $(\bG^*,F_0)$. Let $\bL_1^*\ge\bL^*$ be a maximal (proper) $e$-split Levi
subgroup of $(\bG^*,F_0)$. Thus, $\bT_1:=\bZ(\bL_1^*)_{\Phi_e}$ is an $e$-torus
of $(\bG^*,F)$ and $(\bG^*,F_0)$ of
rank~1, and so $\bT_1^F$ and $\bT_1^{F_0}$ have cyclic Sylow $p$-subgroups
\cite[Prop.~25.7]{MT}, with $|\bT_1^F|_p>|\bT_1^{F_0}|_p$ since
$|\Phi_e(q)|_p>|\Phi_e(q_0)|_p$. Thus a generator $t$ of $(\bT_1^F)_p$ is not
fixed by $\sigma^*$, that is, $t^{\sigma^*}= t^a\ne t$ for some integer $a$.
\par
We claim that $\bC_{G^*}((\bT_1^F)_p) = \bC_{G^*}((\bT_1)_{\Phi_e})$. Indeed, if
$p$ is a good prime for $\bG$ then this is \cite[Prop.~2.1(ii)+(iii)]{CE94}. 
Otherwise, $\bG$ has a factor of exceptional type and $p=3$ (so $e=1,2$) or
$\bG=\tE_8$, $p=5$ and $e=1,2,4$. In the latter cases, the explicit description
of maximal $e$-split Levi subgroups in \cite[\S3.5]{GM20} together with the
fact \cite[Tab.~1]{KM13} that there are no quasi-isolated elements of
order~$p^k$, $k\ge2$, in $\bG^*$ shows the claim.
\par
Now assume there exists $g\in G^*$ with $t^a=t^g$, so conjugation by $g$ makes
an orbit of length $p$ on $\langle t\rangle=(\bT_1^F)_p$. Then $g$ normalises
$(\bT_1^F)_p$ and hence also
$$\bC_{\bG^*}((\bT_1^F)_p)^F = \bC_{\bG^*}((\bT_1)_{\Phi_e})^F=\bL_1^{*F}.$$
But by Lemma~\ref{lem:norm}, $\bN_{\bG^*}(\bL_1^*)^F/\bL_1^{*F}$ has order prime
to $p$ (since $p>e$), which contradicts the assumption that $g$ makes an orbit
of length $p$ on $\langle t\rangle$. Hence $t,t^{\sigma^*}$ are not
$G^*$-conjugate.

Now let $s\in G^*$ be a semisimple $p'$-element such that
$\Irr(b)\subseteq\cE_p(G,s)$. Since $\Irr(b)$ contains the constituents of
$\RLG(\la)$ we may assume $s\in \bL^{*F}$. So
$$t\in \bZ(\bL_1^{*F})\le\bZ(\bL^{*F})\le \bC_{G^*}(s).$$
Note that $st$ and $(st)^{\sigma^*}$ are not $G^*$-conjugate, since neither are
their $p$-parts. Thus, by \cite[Prop.~7.2]{Tay18} the Lusztig series
$\cE(G,st)\subset\Irr(G)$ is not $\sigma$-invariant. Since
$\cE(G,st)\cap\Irr(b)\ne\emptyset$ by
Lemma~\ref{lem:p-elt}, the statement follows in the case $\gamma=\sigma$.

It remains to consider the case $\gamma=\sigma\tau$ where $\tau$ is not an
inner automorphism of $G$. Note that every semisimple conjugacy class of~$G^*$
is stable under inner-diagonal automorphisms of $G^*$ (indeed, $\Inndiag(G^*)$
is induced by the action of $(\tilde\bG^*)^F$ on $G^*$ for some regular
embedding $\bG^*\hookrightarrow\tbG^* = \bZ(\tilde\bG^*)\bG^*$). It follows 
that $\tau$ fixes the rational series $\cE(G,st)$ in the preceding paragraph,
and hence $\gamma$ again does not stabilise $\cE(G,st)$
which contains characters from $b$.
\end{proof}

By Propositions~\ref{prop:red 1}, \ref{prop:exc cov} and~\ref{prop:suzree}, to
complete the proof of Theorem~\ref{thm:newg} we may assume $S$ is a central
quotient of a group $G=\bG^F$ as above, for $\bG$ simple. We now distinguish
two situations:
\begin{equation}\label{eq:caseA}
  p\mbox{ does not divide }|\bZ(G)|,
   \mbox{ and }G{\not\cong}\tD_4(q)\mbox{ with }p=3;
\end{equation}
and
\begin{equation}   \label{eq:caseB}
  p\mbox{ divides }|\bZ(G)|,\mbox{ or }G\cong \tD_4(q)\mbox{ with }p=3.
\end{equation}

\begin{prop}   \label{prop:generic case}
 Theorem~{\rm\ref{thm:newg}} holds in Case~\eqref{eq:caseA}.
\end{prop}

\begin{proof}
Let $(\bG,F)$ be such that $S$ is a central quotient of $G=\bG^F$. Let $b$ be a
$p$-block of $S$ of non-trivial defect. Since in Case~\eqref{eq:caseA} the
order of $\bZ(G)$ is prime to $p$, we may consider $b$ as a block of $G$. For
the groups in~\eqref{eq:caseA},
the only outer automorphisms of $p$-power order are field automorphisms 
modulo $J=\Inndiag(\bar S)$. So by Lemma~\ref{lem:red (c)}, it suffices
to prove parts~(a) and~(b) of Theorem ~\ref{thm:newg}, with $Q \neq 1$ in the
case of \ref{thm:newg}(a), or $p$ divides $|HJ/J|$ in the case of
\ref{thm:newg}(b). Then $H$ contains an element $\gamma$ as in
Proposition~\ref{prop:field}. But then Proposition~\ref{prop:field} shows the
assumptions of \ref{thm:newg}(a), (b) are not satisfied.
\end{proof}

%%%%%%%%%%%%%%%%%%%%%%%%%%%%%%%%%%%%%
\subsection{The groups in Case~\eqref{eq:caseB}} \label{sec:tricky}

\begin{prop}   \label{prop:D4}
 Theorem~{\rm\ref{thm:newg}} holds for $\bar S=\PSO^+_8(q)$ and $p=3$.
\end{prop}

\begin{proof}
Let $b$ be a 3-block of $S$. Since $\bZ(\bG)$ has order prime to~3 we may
consider $b$ as a block of $G$ and thus assume $S=G$. To prove any part of
Theorem \ref{thm:newg}, it suffices to show that $|H/\bar S|$, respectively
$|K/S|$ is coprime to~$3$.

Assume the contrary that $H$, respectively $K$, contains an element $\sigma$
that induces an outer automorphism of order $3$ of $\bar S$.
By Proposition~\ref{prop:field} we may assume $\sigma$ induces a triality
graph automorphism modulo inner-diagonal and field automorphisms. 
We have $3|(q^2-1)$, so $e\in\{1,2\}$. Write
$b=b_G(\bL,\la)$. Let $\bL^*\le\bG^*$ be dual to $\bL$. Let $\bL_1^*\ge\bL^*$
be a maximal (proper) $e$-split Levi subgroup of $\bG^*$. Then $\bL_1^*$ is
either of type $\tA_3$ or $\tA_1^3$. If all maximal $e$-split Levi
subgroups above $\bL^*$ are of type $\tA_1^3$, then $\bL^*$ itself must be of
type~$\tA_1^3$. A computation inside the Weyl group shows that the Levi
subgroup of type $W(\tA_1^3)$ of $W(\tD_4)$ has index~2 in its normaliser. Thus,
by \cite[Lemma~4.16]{CE99} a defect group of $b$ is $\bZ(\bL)_3^F$, hence
cyclic, which was excluded.
\par
Thus we may choose $\bL_1^*$ of type $\tA_3$. Let $t$ be a generator of the
cyclic group $\bZ(\bL_1^*)_3^F$. Note that there are three $\bG^*$-classes of
Levi subgroups of type $\tA_3$ in $\bG^*$, each of which is fixed by
inner-diagonal and field automorphisms, and permuted transitively
by~$\sigma^*$. It follows
that the $G^*$-class of $t$ is not $\sigma^*$-invariant. We can now conclude as
in the proof of Proposition~\ref{prop:field} that $b$ contains characters that
are not $\sigma$-stable.
\end{proof}

We now show an extension of Proposition~\ref{prop:field} to groups of
type~$\tA$.

\begin{prop}   \label{prop:field type A}
 Let $\bG=\SL_n$ and assume that $p$ divides $|\bZ(\bG^F)|$.
 Let $b=b_G(\bL,\la)$ be a $p$-block of $G=\bG^F$ with abelian defect and assume
 $\bZ^\circ(\bL)_{\Phi_e}$ has rank at least~$2$. Suppose $F=F_0^p$ for a
 Frobenius endomorphism $F_0$ of $\bG$ and let $\sigma$ be the automorphism
 of~$G$ induced by $F_0$. Then for any inner-diagonal automorphism $\tau$
 of~$G$, $\Irr(b)$ contains a character that is not $\sigma\tau$-stable
 and is trivial on $\bZ(G)_p$.
\end{prop}

\begin{proof}
We have $G=\SL_n(\eps q)$ and $G^*=\PGL_n(\eps q)$, with $\eps\in\{\pm1\}$.
The assumption on $|\bZ(\bG^F)|$ forces $p|(q-\eps)$. First consider the case
$\eps=1$. Then $e=1$. Let $\bL^*\le\bG^*$ be dual to~$\bL$. Up to replacing
$\bL^*$ by a $\bG^{*F}$-conjugate we may assume that $F_0$ acts on
$\bZ(\bL^*)_{\Phi_1}=\bZ(\bL^*)$ by the $q_0$-power map, where $q=q_0^p$. Now
$\bZ(\bL^*)_{\Phi_1}$ has rank at least two, so there exist at
least two maximal proper 1-split Levi subgroups $\bL_1^*,\bL_2^*\ge\bL^*$, with
$\bZ(\bL_i^*)$ a 1-torus of $\bZ(\bL^*)$ of rank~1. Let $t_i$ be a generator of
the cyclic Sylow $p$-subgroup of $\bZ(\bL_i^*)^F$. Then preimages $\tilde t_i$
of~$t_i$ in $\GL_n(q)$ (under the natural quotient map) can be taken as diagonal
matrices where $\tilde t_i$ has $m_i$ eigenvalues $\zeta$, a primitive $k$th
root of unity with $k=|\bZ(\bL_i^*)^F|_p=|t_i|$ (a $p$-power), and all other
eigenvalues~1, with $1\le m_1<m_2<n$ say. We have
$t_i^\sigma=F_0(t_i)=t_i^{q_0}$, so $t_i$ is not $\sigma$-stable.
\par
Let $a_1,a_2$ be integers not both divisible by~$p$ such that
$m_1a_1+m_2a_2\equiv0\pmod{k}$. Then
$\tilde t:=\tilde t_1^{a_1}\tilde t_2^{a_2}$
has determinant~1, so lies in $\SL_n(q)$, whence
$t:=t_1^{a_1}t_2^{a_2}\in[G^*,G^*]$. By \cite[Prop.~4.5]{NT2} this means that
any character in $\cE(G,st)$, for $s\in L^*$ a semisimple $p'$-element has
$\bZ(G)_p$ in its kernel. Furthermore the class of $t$ is not $\sigma$-stable.
Now $\tilde t$ has eigenvalues $\zeta^{a_1+a_2},\zeta^{a_2}$ and~1, and at
least two of those are distinct by the choice of $a_1,a_2$. Thus, either
$\bC_{\GL_n(q)}(\tilde t)=\GL_{n_1}(q)\times\GL_{n-n_1}(q)$ for a suitable
$1\le n_1<n$, or
$$\bC_{\GL_n(q)}(\tilde t)
  =\GL_{m_1}(q)\times\GL_{m_2-m_1}(q)\times\GL_{n-m_2}(q).$$
All of these have automiser of order at most~2 unless $m_1=m_2-m_1=n/3$ when it
is $\fS_3$. Assume we are in the latter case
and $p=3$. If $\bL^*$ itself is of this form, then the defect groups of $b$ are
non-abelian, being a non-trivial extension of $\bZ(\bL^*)_3^F$ with a Sylow
3-subgroup of the automiser of $\bL^F$ \cite[Lemma~4.16]{CE99}. Otherwise,
$\bL^*$ has strictly smaller
rank and then we may choose the $\bL_i^*$ such that the centraliser of
$\tilde t$ is not of this special form. We may then complete the
argument as in the proof of Proposition~\ref{prop:field}.
\par
If $\eps=-1$ (so $G$ is unitary) we may argue in an entirely similar fashion.
\end{proof}

The proof actually shows that the assumption on abelian defect groups can be
dropped when either $p\ge5$, or $p=3$ and in addition $\bL$ is not of type
$\tA_{n/3-1}^3$.

\begin{prop}   \label{prop:type A}
 Theorem~{\rm\ref{thm:newg}} holds for $\bar S=\PSL_n(\eps q)$,
 $\eps\in\{\pm1\}$ and $p\mid(q-\eps)$.
\end{prop}

\begin{proof}
By our previous reductions and Proposition~\ref{prop:generic case} we may
assume $S$ is a central quotient of $G=\SL_n(\eps q)$. Let $b$ be a $p$-block
of $S$ with abelian non-cyclic defect and $\tilde b$ the $p$-block of~$G$
dominating it. Then by the proof of \cite[Thm~5]{BE99}, one of the following
holds for any block $B$ of $\GL_n(\eps q)$ lying above $\tilde b$:
\begin{equation}\label{case1} 
 B \mbox{ has abelian defect groups; or}
\end{equation} 
\begin{equation}\label{case2} 
\begin{array}{l} p=3,~n=3m, \mbox{ and }\Irr(B)\subseteq\cE_3(\GL_n(\eps q),s),\
\mbox{with }C:=\bC_{\GL_n(\eps q)}(s)\cong\GL_3((\eps q)^m)\\ 
\mbox{for some semisimple }3'\mbox{-element }
 s \in \GL_n(\eps q), \mbox{ and }((\eps q)^m-1)_3=3.
 \end{array}
\end{equation}

Now first assume that $H$ or $K$ induce an automorphism $\gamma=\sigma\tau$,
where $\sigma$ is a field automorphism of $S$ of order~$p$, and
$\tau\in J=\Inndiag(S)$. That is, $\sigma$ is induced by a field automorphism
$F_0$ of $\bG$, where after conjugation we may assume $F=F_0^p$. As
$p|(q-\eps)$ and $q$ is a $p$th power, in fact we have $p^2|(q-\eps)$, and so
we are in case~\eqref{case1}. Also note that if $\tilde b=b_G(\bL,\la)$ then
$\bZ^\circ(\bL)_{\Phi_e}$ has rank at least~2 since $\tilde b$
has non-cyclic defect. But then by Proposition~\ref{prop:field type A} the
block $\tilde b$ contains an irreducible character of $G$ not fixed by~$\gamma$
and trivial on $\bZ(G)_p$, hence a character of~$b$. Thus the assumptions of
any of the parts (a), (b), and~(c) in Theorem \ref{thm:newg} are not satisfied.
\par
So we may now assume that the $p$-elements in $H$ and $K$ only induce diagonal
automorphisms. In particular, this establishes \ref{thm:newg}(b) and shows
that $K/S$ is cyclic. Hence by Lemma~\ref{lem:red (c)} it suffices to
prove~\ref{thm:newg}(a); moreover, using Lemma~\ref{lem:red (a)} we may assume
$H/\bar S$ is a $p$-group contained in $J/\bar S$.

Suppose we are in the situation of~\eqref{case1}. Since $\GL_n(\eps q)$ induces
all diagonal automorphisms on~$G$ and thus on~$S$ and since $H/\bar S$ is a
cyclic $p$-group, we can find a $p$-element $g \in \GL_n(\eps q)$ such that
$G_1:= \langle G,g \rangle$ induces $H$ while acting on $S$. Furthermore,
$G \lhd G_1 \leq\GL_n(\eps q)$ and $G_1/G$ is a $p$-group. Since every
character in $b$ is $H$-invariant, $\tilde b$ is $G_1$-invariant, and since
the defect groups of any $G_1$-block lying above $\tilde b$ are abelian,
the statement follows from Proposition \ref{useful}. 
\par
Finally we consider the situation of~\eqref{case2}. Recall our hypothesis that
the defect group $D$ of $b$ is abelian but non-cyclic. As shown in the proof of
\cite[Thm 5]{BE99}, $|D| \leq 3^2$, whence $|D|=3^2$. On the other hand, $B$
has defect groups of order $3^4$, and $(q-\eps)_3 = 3$, so the defect groups of
$\tilde b$ have order (at least) $3^3$. It follows that $S$ is a quotient
of~$G$ by a central subgroup $Z$ of order $z$ divisible by $3$. By
Lemma~\ref{lem:red (a)} we may assume $|H/\bar S|=(q-\eps)_3=3$, and that $H$
is induced by the conjugation action of
$G_1:= \langle G,g \rangle \leq \GL_n(\eps q)$ for some $3$-element~$g$. As
$G_1/G$ is a $3$-group, $\tilde b$ lies under a unique $3$-block $B_1$ of $G_1$
and so $B_1$ lies under~$B$. Since BHZ holds for $b$ by \cite{BE99}, every
character $\theta\in\Irr(b)$ has height zero, so the $3$-part of its degree is 
$$d:=|\PSL_n(\eps q)|_3/3^2 = |\GL_n(\eps q)|_3/3^4.$$
By assumption, $\theta$ is $g$-invariant, so, viewed as $G$-character, it
extends to a character of~$G_1$, which is still trivial at~$Z$. Thus the
degrees of all characters in $\Irr(B_1)$ that are trivial at~$Z$ have
$3$-part~$d$. 

By \cite[Thm 1]{BE99}, $\Irr(B)$ consists of all characters in
$\cE(\GL_n(\eps q),st)$, for $t$ a $3$-element in $C = \GL_3((\eps q)^m)$. As
$\theta$ is trivial at $Z$ and lies under some such character, using
\cite[Prop.~4.5]{NT2} we see that the order of $\det(st)$ divides $(q-\eps)/z$.
But $s$ is $3'$ and $t$ is a $3$-element, so the order of $\det(s)$ divides 
$(q-\eps)/z$. Let $\omega\in\FF_{q^2}^\times$ be of order $3$, and consider 
$$t := \diag(1,\omega,\omega^2) \in \SL_3((\eps q)^m) < C.$$ 
Then $t$ centralises $s$, and, as $\det(t)=1$, $\det(st)=\det(s)$ has order
dividing $(q-\eps)/z$. Again by \cite[Prop.~4.5]{NT2}, the semisimple character
$\chi_{st}\in\cE(\GL_n(\eps q),st)$ is trivial at~$Z$. Any irreducible
constituent $\psi$ of the restriction of $\chi_{st}$ to $G_1$ is also
trivial at~$Z$. By uniqueness of $B_1$, we have $\psi \in \Irr(B_1)$.
Since $G_1$ has $3'$-index in $\GL_n(\eps q)$,
$$\psi(1)_3=\chi_{st}(1)_3= |\GL_n(\eps q)|_3/|\bC_C(t)|_3
  = |\GL_n(\eps q)|_3/((\eps q)^m-1)^3_3 = |\GL_n(\eps q)|_3/3^3 = 3d,$$
and this is a contradiction.
\end{proof}

\begin{prop}   \label{prop:E6}
 Let $\bar S=\tE_6(\eps q)$, $\eps\in\{\pm1\}$ with $3|(q-\eps)$. Assume that
 BHZ for $p=3$ holds for all groups of order smaller than $|\bar S|$. Then
 Theorem~{\rm\ref{thm:newg}} holds for $\bar S$ at $p=3$.
\end{prop}

\begin{proof}
Let $b$ be a 3-block of~$S$ with non-cyclic abelian defect groups and
$\hat b=b_G(\bL,\la)$ be the block of $G=\bG^F$ dominating it. So $b=\hat b$ if
and only if $S=G$. Again we first show that field automorphisms of order~3
modulo $J:=\Inndiag(\bar S)$ do not point-wise fix all irreducible characters
in $b$. Let first $\eps=1$, so $G=\tE_6(q)$ and $e=1$.
Set $Z:=(\bZ^\circ(\bL^*)^F)_3$, a 3-group of rank at least~2. Assume $F=F_0^3$
and accordingly write $q=q_0^3$.
Now $F_0$ acts by $x\mapsto x^{q_0}$ on $Z$. Let $\bL_i^*$, $i=1,2$, be two
distinct maximal 1-split Levi subgroups of $\bG^*$ containing the dual Levi
$\bL^*\le\bG^*$ of $\bL$ and $t_i\in\bZ(\bL_i^*)^F_3\le Z$ generators of the
Sylow 3-subgroups. Let $a_i\in\ZZ$ not both divisible by~3 be such that
$t:=t_1^{a_1}t_2^{a_2}\in[G^*,G^*]$. Then we have $F_0(t)=t^{q_0}\ne t$.
Note that $t$ has order divisible by~9.
\par
Now $\bL':=\bC_{\bG^*}(t)$ is either a maximal proper 1-split Levi subgroup of
$\bG^*$ or equal to the intersection $\bL_1^*\cap\bL_2^*$, a Levi subgroup of
semisimple rank~4. Assume for a moment that $\bL'$ does not have type $\tD_4$.
Now $\bN_{G^*}(\langle t\rangle)$ normalises $\bC_{\bG^*}(t)=\bL'$. But the
automiser of the latter does not contain elements of order~3 by
Lemma~\ref{lem:norm}, resp. by inspection. Thus, $t$ is not
$\bG^{*F}$-conjugate to $F_0(t)$, and we may complete the argument as in
Proposition~\ref{prop:field type A} to see that $\Irr(\hat b)$ contains
characters trivial on $\bZ(\bG)^F$ in Lusztig series not fixed by $F_0$. If
$\bL'$ has type $D_4$, its automiser is the symmetric group of degree~3. By a
computation in the reflection representation of the Weyl group using \cite{GAP},
there are generators $x,y$ of the Sylow 3-subgroup of $\bZ(\bL')^F$ such that
an element $w$ of order~3 in the automiser acts by $x\mapsto y\mapsto(xy)^{-1}$.
(Note that $\bZ(\bL')^F$ lies in a maximally split torus on which the Weyl
group naturally acts.)
Then $t:=x^wx^{-1}=x^{-1}y$ is a 3-element in $[G^*,G^*]$. Now $w$ acts by
$t^w=x^{-1}y^{-2}$, while $F_0$ sends every element in $\bZ(\bL')$ to its
$q$th power. So again, $t$ is not $G^*$-conjugate to $F_0(t)$ and we conclude
as before.
\par
The above result already establishes \ref{thm:newg}(b) and shows that $K/S$ is
cyclic in~\ref{thm:newg}(c). By Lemma~\ref{lem:red (c)}, it remains to prove
\ref{thm:newg}(a), and we may now assume $H$ only induces diagonal
automorphisms of $S$, and in fact $H=G^*$ using Lemma~\ref{lem:red (a)}.
First assume $S=G$. Let $\bG\hookrightarrow\tbG$ be a regular embedding and
$\tilde G:=\tbG^F$. Then any defect group of $b$ is contained in a maximally
split torus of $G$ and thus defect groups of any block $\tilde b$ of $\tilde G$
above $b$ are contained in a maximally split torus of $\tilde G$ (see
\cite[Thm~1.2(b)]{KM13} for quasi-isolated blocks and \cite[Lemma~4.16]{CE99}
for the others) and so are abelian. Since $\tilde G$ induces all diagonal
automorphisms of $G$, we are done by applying Proposition \ref{useful}
(and arguing as in case~\eqref{case1} of the proof of
Proposition~\ref{prop:type A}).

Finally, assume $S=G/\bZ(G)$ is the simple group (which is isomorphic to
$[G^*,G^*]\le G^*=H$ as duality keeps the root system of type $E_6$). Consider
the block $\tilde b$ of $G^*$
covering $b$. If $\tilde b$ has abelian defect groups, then the statement again
follows from Proposition \ref{useful}. Assume $\tilde b$ has a non-abelian
defect group $\tilde D$, and is not quasi-isolated. Then its Morita-equivalent
Jordan correspondent block $\tilde b_1$, of a group of order less than
$|G^*|/3 = |\bar S|$, also has non-abelian defect by \cite[Thm~1.1]{BDR}. By
assumption, $\tilde b_1$ satisfies BHZ, so contains a character of positive
height. But then so does $\tilde b$ since the Bonnaf\'e--Rouquier Morita
equivalence \cite{BR03} preserves heights. Also by assumption, $b$ has abelian
defect, and hence all characters in $\Irr(b)$ possess height zero by the main
result of~\cite{KM13}.
Now the existence of characters of positive height in $\tilde b$ implies
that $G^*$ does not fix all characters in $\Irr(b)$.  On the other hand, if
$\tilde b$ is quasi-isolated then defect groups of $b$ are non-abelian or
cyclic by \cite[Prop.~4.3 and Thm~1.2(b)]{KM13} and \cite[pp.~353--354]{En00}.
\par
For the twisted type groups, that is $\eps=-1$, entirely similar
considerations apply.
\end{proof}

\begin{cor}   \label{str2}
 Let $p$ be an odd prime and let $S$ be a quasi-simple group. Let
 $\bar S:= S/\bZ(S)$, $\bar S\le H\le \Aut(S)$, and assume 
 $\bO^{p'}(H/\bar S)=H/\bar S$. If $b$ is a $p$-block of $S$ with abelian,
 non-cyclic defect groups such that every character in $b$ is $H$-invariant,
 then $H/\bar S$ is a cyclic $p$-group.
 \end{cor}

\begin{proof}
The statement follows from Proposition~\ref{str}(b) unless $\bar S$ is of type
$\tA_n$, $\tw2\tA_n$, or $p = 3$ and $\bar S$ is of type $\tD_4$ or 
$\tE_6(\eps q)$ with $3|(q-\eps)$ and $\eps \in \{\pm 1\}$. 
In the latter cases, for $J:=\Inndiag(S)$ we have $p \nmid |HJ/J|$ by
Propositions~\ref{prop:generic case}, ~\ref{prop:D4},~\ref{prop:type A},
and~\ref{prop:E6}.
By assumption, $H/\bar S$, and hence $HJ/J$, has no non-trivial $p'$-quotient.  
Thus $H \leq J$, and the claim follows from Proposition~\ref{str}(c).
\end{proof}

\numberwithin{equation}{subsection}

%%%%%%%%%%%%%%%%%%%%%%%%%%%%%%%%%%%%%%%%%%%%%%%%%%%%%%%%%%%%%%%%%%%%%%%%%
\section{Proofs of Theorems A and C}\label{sec:finalpf}

\subsection{Proof of Theorem A}
Now we proceed to prove the ``only if'' implication of Brauer's Height Zero
Conjecture for primes $p>2$. Suppose $B$ is a $p$-block of $G$ with defect group
$D$, and assume all $\chi\in\irr B$ have height zero. We want to show that
$D$ is abelian. We will assume that $G$ is a counterexample to BHZ, first with
$|G/\zent G|$ smallest possible, and then with $|G|$ smallest possible.
%The case $p=2$ was shown in \cite{Ru}, so we may assume that $p$ is odd.
By the Gluck--Wolf theorem \cite{GW}, $G$ is not $p$-solvable. Recall that if
$H \le G$ and $N\nor G$, then $|H:\zent H|\le |G:\zent G|$ and
$|G/N:\zent{G/N}|\le |G: \zent G|$.
\medskip

\iitem{Step 1.}~~$B$ is a quasi-primitive block; that is, if $N \nor G$
 and $e$ is a block of $N$ covered by~$B$, then $e$ is $G$-invariant.
\medskip

This follows by Fong--Reynolds (\cite[Thm~9.14]{N}) and induction.
\medskip

\iitem{Step 2.}~~If $N$ is a proper normal subgroup of $G$, then $D\cap N$ is
 abelian. In particular, $\Oh{p'}G=G$, $D$ is not contained in any proper
 normal subgroup of $G$, and $Q:=\oh p G$ is abelian. 
\medskip

Suppose $N$ is a proper normal subgroup of $G$ and $e$ is a $p$-block of~$N$
covered by~$B$. By \cite[Thm~9.26]{N}, we have $D\cap N$ is a defect group
of~$e$, using that $e$ is $G$-invariant (by Step 1). Let $\xi \in \irr e$.
By \cite[Thm~9.4]{N}, there is some $\chi\in\irr B$ (of height zero) lying
over~$\xi$. By \cite[Lemma~2.2]{Mu1}, we have that $\xi$ has height zero.
Hence $D\cap N$ is abelian by minimality of $G$.

Now, if $\Oh {p'}G < G$, then $D=D \cap \Oh{p'}G$ is abelian by the previous
claim applied to $N=\bO^{p'}(G)$, a contradiction. Finally, since $Q=\oh{p}G<G$
and $Q\le D$, the claim applied to $N=Q$ shows that $Q=Q\cap D$ is abelian.
\medskip

Now let 
$$C:=\cent GQ,$$ 
so that $Q \leq C$. 
\medskip

\iitem{Step 3.}~~We have that $\zent G$ has $p'$-order. 
 In particular, we may assume that $C<G$ or that $Q=1$.
 Also, if $L \nor G$ is a non-trivial $p$-group, then $D/L$ is abelian.
\medskip

Suppose that $1<L$ is a normal $p$-subgroup of $G$. Then there is a block
$\bar B$ of $G/L$ contained in $B$ with defect group $D/L$ by
\cite[Thm~9.9(b)]{N}. By the definition of heights, all irreducible
characters in $\bar B$ have height zero.
Thus $D/L$ is abelian, by BHZ applied to~$G/L$.

Assume now that $1<L$ is a central $p$-subgroup. Let $\delta\in\irr D$, and let
$\nu \in \irr L$ be under~$\delta$. Let $\chi\in\irr B$ over~$\nu$. Since
$\chi$ has height zero, it follows that $\nu$ extends to some linear
$\tilde\nu\in\irr D$, by \cite[Thm 4.4]{Mu0}, using that $\nu$ is
$G$-invariant. By Gallagher's Corollary~6.17 of \cite{Is},
we have $\delta=\tilde\nu \beta$, for some $\beta\in\irr{D/L}$.
Since $D/L$ is abelian by the previous paragraph,
we conclude that $\delta(1)=1$. Hence $D$ is abelian, contrary to the choice of
$G$. Hence $p \nmid |\bZ(G)|$.
\medskip

%\iitem{Step 3.}~~Suppose that $N\nor G$ and that $e$ is a $p$-block of $N$
% covered by $B$, where $\irr e=\{\theta\}$ has $p$-defect zero. Then we may
% assume that $N\leq \zent G$ and $N$ is cyclic. In particular, 
% $$Z:=\oh{p'}G$$ 
%is cyclic and central.
%\medskip
%
%This follows from the theory of ordinary/modular character triples.
%(See Problems 8.12 and 8.13 of \cite{N}.) 

\iitem{Step 4.}~~If $Z:=\oh{p'}G$, then $Z=\zent G$ and $Z$ is cyclic.

\medskip
Let $\theta \in \irr Z$ such that the $p$-block $\{\theta \}$ of $Z$ is covered
by $B$.
We know that $\theta$ is $G$-invariant. We prove this step using the language
of $\theta$-blocks and character triples (see \cite{R}.) We have that
$(G,Z,\theta)$ is a character triple. By Problems 8.12 and 8.13 of \cite{N},
there exists an ordinary-modular character triple $(G^*,Z^*,\theta^*)$
isomorphic to
$(G,Z, \theta)$, which we can construct as in \cite[Thm~3.4]{R}.  Hence $Z^*$
has order not divisible by $p$ and is central in~$G^*$. Since $G/Z$ is
isomorphic to $G^*/Z^*$, we have that $Z^*=\oh{p'}{G^*}$. If 
$$\sigma: \irr{G|\theta} \rightarrow \irr{G^*|\theta^*}$$ 
denotes the associated standard bijection,
we have $\sigma(\irr B)=\irr{B^*}$ for a unique $p$-block $B^*$ of $G^*$.
In particular, $\irr B=\irr{B|\theta}$ is a $\theta$-block.
Since $\sigma(\chi)(1)=\chi(1)/\theta(1)$, then $\sigma(\chi)(1)_p=\chi(1)_p$
for $\chi \in \irr B$, and so all characters in $\irr{B^*}$ have the same
height (zero). 

Let $D^*$ be a defect group for $B^*$. We show next that $D$ and $D^*$ are
isomorphic, hence if $D^*$ is abelian, then so is $D$. Since 
$$|G:D|_p=\chi(1)_p=\sigma(\chi)(1)_p=|G^*:D^*|_p$$ 
and $|G|_p=|G^*|_p$, we have that $|D|=|D^*|$. If $D_\theta/Z$ is a
$\theta$-defect group of $B$, then, by \cite[Def.~4.1]{R}, we have $D_\theta/Z \cong D^*Z^*/Z^* \cong D^*$. By \cite[Thm~5.1]{R}, we have that $D_\theta /Z \leq DZ/Z$, replacing $D$ by a $G$-conjugate, if necessary. By comparing orders, we have that $D_\theta/Z=DZ/Z$ is isomorphic to $D$. Therefore $D$ and $D^*$ are isomorphic. 

Notice that $|G^*:\zent{G^*}| \le |G^*:Z^*|=|G:Z| \le |G:\zent G|$, using that
$\zent G \leq Z$ (by Step 3). Therefore, if $|G^*:\zent{G^*}|< |G:\zent G|$,
then we are done by applying BHZ to~$G^*$. In the case of equality, we have
$Z=\zent G$.
Finally, we show that $Z$ is cyclic.  Let $\{\lambda\}$ be the block
of $Z$ covered by $B$, where $\lambda \in \irr Z$. Let $K=\ker(\lambda)$.
Hence, $K$ is contained in $\ker(\chi)$ for all $\chi \in\irr B$. If $K>1$,
then we apply \cite[Thm~9.9(c)]{N} and BHZ to $G/K$. The choice of $G$ shows
that $K=1$ and therefore, that $Z$ is cyclic.
\medskip

From now on, let $\bE(X)$ denote the layer of a finite group $X$.

\medskip

\iitem{Step 5.}~~We have that $C$ is not $p$-solvable and $\bE(G) \neq 1$.
\medskip
 
Assume that $C$ is $p$-solvable. As $\oh{p'}C=Z$ is central, we have
$\oh p{C/Z}=QZ/Z$. Since $C$ is $p$-solvable, we have $\cent {C/Z}{QZ/Z}=QZ/Z$.
This implies $C=QZ$. By \cite[Lemma 3.4]{NT1}, then $\irr B=\irr{G|\lambda}$
for some $\lambda\in\irr Z$. Let $P\in\syl pG$, and let
$\hat\lambda=1_P\otimes \lambda \in \irr{P \times Z}$.
Then $\hat\lambda^G$ has $p'$-degree, and therefore
it contains some $p'$-degree irreducible constituent $\chi$.
Now, $\chi$ lies over $\lambda$, and therefore $\chi \in \irr B$. We conclude
that $P$ is a defect group of $B$. By hypothesis, $p$ does not divide $\tau(1)$
for every $\tau \in \irr B$, whence $P$ (and hence $D$) is abelian by the main
result of \cite{NT2}. Since $G$ is a minimal counterexample, $C$ is not
$p$-solvable. 

Next, assume that $\bE(G) = 1$. Then $\bE(C)=1$ and 
$\bF^*(C) = \bF(C)= \bO_{p'}(\bF(C)) \times \bO_p(C)$. Now, $\bO_p(G) = Q \leq \bO_p(C) \leq \bO_p(G)$, whence
$\bO_p(C)=Q$ is central in $C$, and $\bO_{p'}(\bF(C)) \leq \bO_{p'}(G) = Z$ is central in $G$. It follows 
that $C = \bC_C(\bF^*(C)) \leq \bF^*(C)$, and so $C$ is $p$-solvable, a contradiction to the previous conclusion.
\medskip

\iitem{Step 6.}~~Let $S$ be a component of $G$, and let $N$ be the normal
 subgroup of $G$ generated by the $G$-conjugates of $S$. Let $e$ be the block
 of $N$ covered by $B$, and let $b$ be the only block of $S$ that is covered
 by~$e$. Then $e$ is not nilpotent.
 In particular, there are $\alpha,\beta\in\irr b$ with different degrees.
\medskip

First notice that $b$ is the only block covered by $e$, because $N$ is a
central product of the different $G$-conjugates of $S$. In particular, every
irreducible character of $S$ is $N$-invariant.
Suppose that $e$ is nilpotent. Let $N_1=N\zent G$, and let $e_1$ be the unique
block of $N_1$ covered by $B$ and covering $b$. It is clear that $e_1$ is
nilpotent, using the definition of nilpotent blocks. (See
\cite[Lemma~7.5]{Sa}.) We have that
$D_1:=D\cap N_1$ is a defect group of $e_1$ (see \cite[Thm~9.26]{N}).
By Theorem~\ref{kp}, there is a finite group $L'$ with
$|L':\zent{L'}|<|G:\zent G|$, where $L'$ has a block $B'$ with defect group $D$
and such that all the irreducible characters of $B'$ have height zero.
Therefore $D$ is abelian by BHZ applied to $L'$. Thus $e$ is not nilpotent. 

If all the irreducible characters in $b$ have the same degree, then all the
irreducible characters in $e$, a central product of the $G$-conjugates of $b$,
also have the same degree. Then by \cite[Prop.~1 and Thm~3]{OT}, we have
$D \cap N$ is abelian with inertial index one. By Brou\'e--Puig
\cite[1.ex.3]{BP}, the block $e$ is nilpotent, and we are again done.
\medskip

\medskip
\iitem{Step 7.}~~Let $S$ be any component of $G$. Suppose $B'$ is the
$\bE(G)$-block covered by $B$ and $b$ is the $S$-block covered by $B'$. Then
$b$ has non-cyclic defect groups. In particular, no component of $G$ has cyclic
Sylow $p$-subgroups. 
\medskip

By Step 6, $b$ has non-central defect groups (since any block with central
defect groups is nilpotent, by \cite[1.ex.1]{BP}).
Suppose that $b$ has cyclic defect groups. 
Consider the central product $N$ of the different $G$-conjugates of $S$,
so that $N \nor G$. Let $D_1:=D\cap N$, so that $G=\norm G{D_1}N$. Let $b_0$ be
the block of $N$ covered by $B$.
Let $b_1$ be the Brauer first main correspondent of $b_0$, which is the block
of $\norm N{D_1}$ that induces $b_0$. Let $B_1$ be the
unique block of $\norm G{D_1}$ that covers $b_1$ and induces $B$ (by the
Harris--Kn\"orr Theorem 9.28 of \cite{N}), which is a block with defect
group~$D$. By Theorem \ref{cyciAM}, we conclude that all the irreducible
characters of $B_1$ have height zero. If $\norm G{D_1}<G$, by BHZ applied to
$\norm G {D_1}$ we will have that $D$ is abelian, a contradiction. Hence
$D_1 \nor G$. In this case, $D_1 \leq Q$. Since $[Q,N]=1$, we conclude that
$D_1\leq \zent N$, and $b_0$ is nilpotent, in contradiction with Step 6.
\medskip

\iitem{Step 8.}~~All components of $G$ are normal in $G$.
\medskip

Suppose that $S_1$ is a non-normal component, and write the normal subgroup $N$
in Step~6 as $N = S_1*S_2*\cdots*S_m$, a central product of $m >1$ components,
where $G/N$ permutes $S_1,\ldots,S_m$ transitively. In particular,
$S_1, \ldots,S_m$ are isomorphic to each other, and 
we fix an isomorphism between $S_1$ and any $S_i$. 

Again by Step 1, $B$ covers a unique block $e$ of $N$ which is $G$-invariant, 
and $e$ is then the central product of blocks $b_i$ of $S_i$, $1\le i\le m$.
By Step 6, each $\irr{b_i}$ contains characters of different degrees.
%, then all characters in $\Irr(e)$ also have the same degree, in contradiction with Step 6.
Furthermore, as before, the block $e$ has defect group $D \cap N$,
which is abelian by Step 2. It follows that $b_i$ has abelian defect groups, which are non-cyclic by Step 7.
%If $A_i$ denotes the stabiliser of $b_i$ in $\Aut(S_i)$, then 
By Theorem~\ref{newmandi}, $\Irr(b_i)$ contains at least three
$\Aut(S_i)$-orbits, say of $\al_i$, $\beta_i$, and~$\gamma_i$.
For each $i$ and given the fixed isomorphism between $S_1$ and $S_i$, we can view $\al_1$, $\beta_1$, and $\gamma_1$
as $S_i$-characters. Relabeling $\al_i$, $\beta_i$, and $\gamma_i$ if necessary, we may assume that
\begin{equation}\label{eq:3c}
  \begin{array}{l}
  \alpha_i \mbox{ is not }\Aut(S_i)\mbox{-conjugate to }\beta_1\mbox{ or }\gamma_1,\\
  \beta_i \mbox{ is not }\Aut(S_i)\mbox{-conjugate to }\alpha_1\mbox{ or }\gamma_1,\mbox{ and}\\
  \gamma_i \mbox{ is not }\Aut(S_i)\mbox{-conjugate to }\alpha_1\mbox{ or }\beta_1.
  \end{array}
\end{equation}
Since $m > 1$, the (transitive) permutation action of $G$ on the set
$\{S_1,\ldots,S_m\}$ is non-trivial, and so the kernel $K$ of this action
is a proper normal subgroup of $G$ containing~$N$,
and from Step 2 we have $\bO^{p'}(G/K)=G/K$.
Now applying Theorem~\ref{orbits} to $G/K$ we obtain a partition 
$$\{S_1, \ldots,S_m\} = \Delta_1 \sqcup \Delta_2 \sqcup \Delta_3$$
such that $\cap^3_{i=1}\Stab_{G/N}(\Delta_i)$ has index divisible by $p$ in $G/N$. Setting 
$$\theta:= \theta_1 \otimes \theta_2 \otimes \cdots \otimes \theta_m,$$
where $\theta_i = \alpha_i$ if $S_i \in \Delta_1$, $\theta_i = \beta_i$ if $S_i \in \Delta_2$, and 
$\theta_i=\gamma_i$ if $S_i \in \Delta_3$,
we then see that $B$ covers the $N$-block of $\theta$. 

Consider any $g \in G$ that fixes $\theta$, and suppose
that $S_1^g = S_j$. Note that $\theta|_{S_j}$ is a multiple of $\theta_j$, and $\theta^g|_{S_j}$ is a multiple of 
$\alpha_1^g$. Again using the fixed isomorphism $S_1 \cong S_j$ and writing every $x^g$ with $x \in S_1$
as $x^\sigma$ for a suitable $\sigma \in \Aut(S_j)$,
we can write $\alpha_1^g$ as $\alpha_1^\sigma$, and so $\theta_j=\alpha_1^\sigma$ and thus $\alpha_1$ and $\theta_j$ are
$\Aut(S_j)$-conjugate. Since neither $\beta_j$ nor $\gamma_j$ are $\Aut(S_j)$-conjugate to $\alpha_1$ by \eqref{eq:3c}, 
$\theta_j$ must be $\alpha_j$, which means $S_j \in \Delta_1$ by the choice of the $\theta_i$'s. 
This argument, applied to any $S_i$, shows that $g$ stabilises the partition $\Delta_1 \sqcup \Delta_2 \sqcup \Delta_3$. 
Thus 
$$G_\theta/N \leq \bigcap^3_{i=1}\Stab_{G/N}(\Delta_i),$$ 
and so $p$ divides $|G:G_\theta|$. But this contradicts Proposition \ref{amaz}.
\medskip

\iitem{Step 9.}~~If $S$ is a (normal by Step 8) component of $G$, then $\zent S$ is of $p'$-order.
Also, if $D$ is any defect group of $B$ and $R:=D\cap S$, then
 $G=S\cent GR$, and $[D,R]=1$.
Furthermore, every $\alpha \in \irr b$ (where $b$ is the $S$-block covered by $B$) is $G$-invariant, and extends to $DS$. Moreover, $G/S\cent GS$ is
a $p$-group.\goodbreak
\medskip

First we prove that there is a defect group $D$ of $B$ satisfying that
$[D, D\cap S]=1$ and $G=S\cent G{S\cap D}$. Let $b$ be the block of $S$ covered
by~$B$. Notice that $\cent GS \leq K:=G[b] \nor G$, the Dade group, by the
definition of $G[b]$. (See \cite{Mu3}, for instance.) Also, notice that $R=D\cap S$
is a defect group of $b$, and that $G=S\norm GR$, by the Frattini argument.
Let $b_R$ be a block of $\cent SR$ with defect group $R$ inducing $b$, and let
$T=\norm GR_{b_R}$. By Lemma~\ref{bD}, there is a defect group $D^*$ of $B$
such that $D^* \cap S=R$ and $D^* \leq T$. For the sake of notation, we assume
that $D^*=D$. Let
$H:=G/\cent GS$ and $\bar S=S\cent GS/\cent GS$. By Proposition~\ref{str},
we have that $G/S\cent GS$ has a normal $p$-complement $U/S\cent GS$.
Since $|G:U|$ is a power of $p$, and $B$ covers a unique block of $U$, we have
$G=UD$, by \cite[Thm 9.17]{N}.
Now, every $\tau \in \irr b$ has the property that $|G:G_\tau|$ is coprime
to~$p$, by Proposition \ref{amaz}. Let $x \in D$. Let $W=U\langle x\rangle$.
Since $|G:G_\tau|$ is $p'$, then $G=UG_\tau$, so $|W:W_\tau|$ is also $p'$ for
every $\tau \in \irr b$. So the hypotheses of Theorem~\ref{thm:newg}(a) are
satisfied with the group $W/\cent GS$ and every $p$-element $x \in D$.
Moreover, by the choice of $G$, BHZ holds for all finite groups $X$
with $|X/\bZ(X)| < |G/\bZ(G)|$, in particular for all groups of order less than
$|S/\bZ(S)|$. We conclude by Theorem \ref{thm:newg}(a) that $[x,R]=1$, for
$x \in D$. Hence $D \leq S\cent GR \nor G$.
If $ S\cent GR<G$, we are done by Step 2. 

Now, let $D_1:=D\cap K$. By \cite[Thm~3.5(ii)]{Mu3}, we have $D_1=R\cent DR=RD$
and we conclude that $D=D_1$. Hence $K=G$ by Step 2 (recall that $K$ is normal
in $G$). Therefore every $\alpha \in \irr b$ is $G$-invariant
(by \cite[Lemma 3.2(a)]{KoS1}, for instance). 
By Corollary \ref{str2},
we have that $H/\bar S$ is a $p$-group. Thus $G/S\cent GS$ is a $p$-group.

Let $\theta \in \irr b$. Then $\theta$ lies under some
$\chi \in \irr B$, which by hypothesis has height zero. By \cite[Thm~4.4]{Mu0},
we have $\theta$ extends to $SD^g$ for some $g \in G$.
Therefore $\theta=\theta^{g^{-1}}$ extends to $SD$, as claimed.

Now, we prove that if $g\in G$, then $[D^g, D^g \cap S]=1$ and $G=S\cent G{D^g \cap S}$. Indeed, since $S$ is normal in $G$, we have that
$[D^g, D^g \cap S]= [D, D\cap S]^{g}=1$.
The second part follows because $D^g \cap S$ is a defect group of $b$, and
therefore $D^g \cap S=R^s$ for some $s\in S$. Thus $G=S\cent GR=S\cent G{R^s}$.

Suppose $1<Z_p$ is a Sylow $p$-subgroup of $\zent S$. By Step 3, $D/Z_p$ is
abelian. Since $Z_p \leq \oh p S \leq D\cap S$, we have that $SD/S$ is abelian.
In this case, the block $\tilde b$ of $SD$ that covers $b$ has defect group $D$
(by \cite[Probl.~9.4]{N}, using that $b$ is $G$-invariant) and, as we saw, all
irreducible characters of $b$ extend to $SD$. Using that and the fact that
$SD/S$ is abelian, it follows that all irreducible characters in $\tilde b$
restrict irreducibly to $S$. In this case, we easily check that all the
irreducible characters in $\tilde b$ have height zero.
If $SD<G$, then we are done by minimality of $G$. So we may assume that
$SD=G$.  Since $[D,R]=1$, we have that $Z_p \leq \zent G$. By Step 3, this is
not possible. We conclude that $Z_p=1$ and thus $p \nmid |\bZ(S)|$.
\medskip

\iitem{Final Step.} From Step 9, we conclude that
\begin{equation}\label{eq:gk}
  G/K \mbox{  is a }p\mbox{-group}, 
\end{equation}
where  
$$K:= \bigcap_{S {\rm{~component~of~}}G} S\cent GS = \bE(G)\bC_G(\bE(G)) = E\bC_G(E),~E:=\bE(G)Z.$$ 
We also have $F=EQ$ for 
$$F:=\bF^*(G) = \bF(G)\bE(G)=(Q \times Z)*\bE(G).$$
If $F=G$, then $G$ is a central product of an abelian group with quasi-simple
groups. However in this case BHZ holds for $G$ by the quasi-simple case
\cite{KM17}, using Corollary~\ref{central}. Hence we may assume $F < G$.

Next we show that 
\begin{equation}\label{eq:kf}
  K/F \mbox{ is a }p'\mbox{-group}. 
\end{equation}
Indeed, let $T:=\bC_G(E)$ which contains $Q$.  Let $c_1$ be the block of $T$
covered by $B$, using Step 1. By minimality of $G$, $c_1$ has an abelian defect
group $D_1$. Note that 
\begin{equation}\label{eq:ctq}
  \bC_T(Q) = \bC_G(F) \leq  F \cap T=QZ.
\end{equation}  
Hence, by \cite[Lemma 3.4]{NT1} we have that $c_1=\Irr(T|\lambda)$ for the
character $\la \in \Irr(Z)$ that lies under $B$.
By inducing $1_{P_1}\times\la\in\Irr(P_1Z)$ to $T$, where $P_1\in \Syl_p(T)$
contains $Q$, we see that there is some $p^\prime$-degree irreducible character
$\nu$ of $T$ over $\la$. As $\nu$ lies in $c_1$, $c_1$ has maximal defect, and 
hence $D_1 \in \Syl_p(T)$. Hence the abelian $p$-group $D_1$ centralizes $Q$,
and so $D_1 \in \bC_T(Q)=QZ$, i.e. $D_1=Q$, proving \eqref{eq:kf}. 

By Step 1, there is a unique block $e$ of $E$ and a unique block $f$ of $F$
covered by $B$; in particular, $e$ is covered by $f$. Let $D$ be any defect
group of $B$. Then $D \cap F$ is abelian by Step~2 (since $F<G$). Let
$$R:=D\cap E,$$ 
so that $R$ is a defect group of the $E$-block $e$. Write
$$E=Z*S_1* \cdots *S_n,$$ 
where $S_i$, $1 \leq i \leq n$, are the components of $G$, which all are normal
in $G$ by Step 8. Since $D_i:= D \cap S_i$ is a defect group of the unique
$S_i$-block $b_i$ covered by $e$, by Corollary \ref{central} we have
$R = D_1 \times \cdots \times D_n$
(being a direct product since $\bZ(E)=Z$ is a $p'$-group).
By Step~9, we have that $[D,R]=1$. Also, since $[Q,E]=1$ and the $\zent{S_i}$
are $p'$-groups by Step~9, we have
\begin{equation}\label{eq:qe}
  Q\cap E=1 \mbox{ and }F = Q \times E.
\end{equation}

Let $c_2$ be the unique block of $C=\cent GQ$ covered by $B$, using Step 1.
By \cite[Cor. 9.21]{N}, we have that $B$ is the unique block of $G$ that
covers~$c_2$. By \eqref{eq:ctq}, $C \cap K=\cent KQ=F$.  Since $K/F$ is a
normal Hall $p$-complement of $G/F$ by \eqref{eq:gk} and \eqref{eq:kf}, we have
that $C/F$ is a $p$-group. Therefore $C/E$ is a $p$-group by \eqref{eq:qe}. By
\cite[Cor.~9.6]{N}, $c_2$ is the only block of $C$ covering $e$.

Now, fix some $\rho \in \Irr(e)$, and consider any $\chi \in \Irr(G|\rho)$. We
claim that $\chi$ lies in $B$. If $\gamma \in \irr{C}$ lies under $\chi$ and
over $\rho$, then we see that $\gamma$ lies in $c_2$. In particular $\chi$ lies
in a block that covers $c_2$, and therefore $\chi$ lies in $B$, as claimed. 
Recall that $G/K$ is a $p$-group by \eqref{eq:gk}, and by Step 1, any block
$c_3$ of $K$ covered by $B$ is $G$-invariant. By \cite[Cor.~9.6]{N}, 
$B$ is the unique block of $G$ that covers $c_3$. By \cite[Cor.~9.18]{N}, and
using the height zero hypothesis, for every $\tau\in\irr{c_3}$ that lies under
$\chi$ we have $p \nmid \chi(1)/\tau(1)$. But $K/F$ is a $p'$-group
by~\eqref{eq:kf} and $F= E \times Q$ by \eqref{eq:qe} with $Q$ abelian. It
follows that $p \nmid \chi(1)/\rho(1)$. Hence $G/E$ has abelian Sylow
$p$-subgroups by \cite[Thm A]{NT2}, and so
\begin{equation}\label{eq:de}
  D/(D \cap E) \mbox{ is abelian}.
\end{equation}

Now, if $Q>1$, then $D/Q$ is abelian by Step 3, and since $D/(D\cap E)$ is
abelian by \eqref{eq:de}, we obtain $[D,D] \leq Q \cap E=1$ using
\eqref{eq:qe}, i.e., $D$ is abelian, and we arrive at a contradiction.
Thus we have shown
\begin{equation}\label{eq:q}
  Q=1.
\end{equation}

For each $i$, recall that $b_i$ is the unique $S_i$-block covered by $e$. Then
any $\theta_i\in\irr{b_i}$ is $G$-invariant and extends to
$$H_i:=S_iD$$ 
by Step 9. We also have that $H_i/S_i$ is a $p$-group, and $[D,D]\leq D\cap E$
by \eqref{eq:de}. If $d \in D \cap E$, we can write $d = xs \in E = S_i\bC_E(S_i)$ with $s \in S_i$ and $x \in \bC_E(S_i)$. Thus
$x = ds^{-1} \in H_i$ centralises $S_i$, i.e., $x \in \bC_{H_i}(S_i)$ and $d \in S_i\bC_{H_i}(S_i)$. It follows that
$[D,D] \leq D\cap E \leq S_i\bC_{H_i}(S_i)$, and so
\begin{equation}\label{eq:hi}
  H_i/S_i\bC_{H_i}(S_i) \mbox{ is an abelian }p\mbox{-group}. 
\end{equation}
We also note that $\bC_{H_i}(S_i)$ has a normal $p$-complement ---
indeed, $\bC_{H_i}(S_i) \cap S_i = \bZ(S_i)$ is a $p'$-group (by Step 9), and
$\bC_{H_i}(S_i)/\bZ(S_i) \hookrightarrow H_i/S_i$ is a $p$-group.
So $\bC_{H_i}(S_i) = \bZ(S_i) \rtimes Q_i$ for a Sylow $p$-subgroup $Q_i$ of
$\bC_{H_i}(S_i)$. But then $Q_i \leq \bC_{H_i}(S_i)$ centralises
$\bZ(S_i)$, so in fact
$$\bC_{H_i}(S_i) = \bZ(S_i) \times Q_i$$
and $Q_i=\bO_p(\bC_{H_i}(S_i)) \lhd H_i$.
Note that $\bO_p(H_i) \cap S_i \leq \bO_p(S_i)=1$, implying $\bO_p(H_i)=Q_i$.
Recall that the unique block $\tilde{b_i}$ of $H_i$ that covers $b_i$ has
defect group $D$, by \cite[Lemma~2.2]{Mu0}; in particular, $Q_i \lhd D$.

Now $S_i$ naturally embeds in 
$$L_i:= H_i/Q_i$$ 
as a normal subgroup. Our goal now is to apply Theorem~\ref{thm:newg}(c) to
$L_i$ with respect to the quasi-simple group 
$$\bar{S_i}:=S_iQ_i/Q_i \cong S_i$$
and the block $\bar b_i$ of $\bar S_i$ which is naturally isomorphic to $b_i$.
If $yQ_i \in L_i$ centralises $S_iQ_i/Q_i$ (modulo $Q_i$), then $[y,S_i] \leq Q_i$.  As $S_i \nor H_i$, we must have $[y,S_i] \sbs S_i \cap Q_i=1$, so $y \in \bC_{H_i}(S_i) = \bZ(S_i) \times Q_i$. Thus 
$\bC_{L_i}(\bar S_i) = \bZ(\bar S_i)$, and 
\begin{equation}\label{eq:li}
  L_i/\bar S_i \cong H_i/S_iQ_i = H_i/S_i\bC_{H_i}(S_i)
    \mbox{ is an abelian }p\mbox{-group}
\end{equation}
by \eqref{eq:hi}.

Next we show that each $\theta_i \in \irr{\bar b_i}$ extends to $L_i$. By
Step~9, $\theta_i$, considered as a character of $S_i$, has an extension
$\hat\theta_i$ to $H_i$. Restricting to $S_iQ_i=S_i \times Q_i$, we have 
$$\hat\theta_i|_{S_i \times Q_i} = \theta_i \otimes \lambda$$
for a unique linear character $\lambda \in \Irr(Q_i)$. In particular, $\lambda$
is $H_i$-invariant. Recall that $D_i=D \cap S_i$, and we have
$D_i \cap Q_i \leq S_i \cap Q_i = 1$, so we can view $\lambda$ as a character
of $D_iQ_i$ which is trivial at $D_i$, and $D$-invariant. Since the characters
in $\irr{b_i}$ are $H_i$-invariant, they are also $H_i$-invariant when
considered as characters of $\bar b_i$.  By Corollary \ref{str2} applied to
$L_i$ with respect to the block $\bar{b_i}$, we have that 
$$H_i/S_i\bC_{H_i}(S_i) = S_iD/(S_i \times Q_i) \cong D/(D_i \times Q_i)$$
is cyclic. Hence $\lambda$ extends to a (linear) character $\nu$ of
$D/D_i \cong H_i/S_i$. Now, viewing $\nu$ as a linear character of~$H_i/S_i$,
we have that $\hat\theta_i\bar\nu$ restricts to $\theta_i$ on $S_i$ and trivial
on $Q_i$, and thus $\theta_i$ extends to $L_i$, as wanted.

Since $b_i$ and $\bar{b_i}$ are isomorphic, in particular $D_iQ_i/Q_i$ is a
defect group of~$\bar b_i$. 
Since $D$ is a defect group of $\tilde{b_i}$, by Lemma \ref{elem},
we have that $D/Q_i$ is a defect group of the block $\bar b_i$ of $L_i=H_i/Q_i$.
As $\theta_i$ extends to $L_i$ for every $\theta_i \in \irr{\bar b_i}$,
and $L_i/\bar{S_i}$ is an abelian $p$-group by \eqref{eq:li}, we see that
every character in $\bar b_i$ has height $0$. 
Applying Theorem~\ref{thm:newg}(c) to the block $\bar b_i$ of $\bar{S_i}$ (which
again holds in the case $p=3$ and $S_i$ is of type $\tE_6$ or $\tw2\tE_6$, by
minimality of $G$), we see that $D/Q_i$ is abelian, and thus
$$[D,D] \leq Q_i \leq \bC_G(S_i).$$
This is true for all components $S_i$, so $[D,D]$ centralises $E=\bE(G)Z$.
Since $Q=1$ by \eqref{eq:q}, we have $E=F$. It follows that 
$$[D,D] \leq \bC_G(E) = \bC_G(F) \leq F,$$ 
and hence $[D,D] \leq \bZ(F)$, and the latter group is a $p'$-group, again
because $\oh pG=Q=1$ by \eqref{eq:q}. Consequently, $[D,D]=1$, contrary to the
choice of $G$ as a minimal counterexample.
\qed

\subsection{Proof of Theorem C}
Having proved Theorem A, and hence BHZ for the prime~$p$, we now see that the
extra assumption for types $\tE_6(\eps q)$ in Theorem~\ref{thm:newg}(a) is
always satisfied. Hence Theorem~C follows from Theorem~\ref{thm:newg}(a), if
the defect group $D$ is not cyclic.
So suppose now that we have a quasi-simple group $S$ with $\zent S$ cyclic and
$p'$-prime, that $\sigma$ is a $p$-power order automorphism of $S$ fixing all
the irreducible characters of a block $b$ of $S$ with cyclic defect group~$D$,
and that $\sigma$ stabilizes a block $b_D$ of $\cent GD$ that induces $b$. We
want to show that $\sigma$ fixes the elements of $D$. Let $N=\norm GD$. Then
$b_N=(b_D)^N$ is the Brauer First Main correspondent of $b$.
By \cite{KoS}, we know that $b$ satisfies the Alperin--McKay inductive
condition. In particular, there is a bijection $\irr b \rightarrow \irr{b_N}$
that commutes with the action of $\sigma$. Hence, all irreducible characters in
$b_N$ are $\sigma$-invariant too. By \cite[Thm~9.12]{N}, there is a unique
irreducible character $\theta \in \irr{b_D}$ with $D \sbs \ker\theta$. 
In particular, $\theta$ is $\sigma$-invariant. Also, the stabiliser $T$ of
$b_D$ in $N$ is the stabiliser of $\theta$ in~$N$.
Also $T/\cent GD$ is a $p'$-group by \cite[Thm~9.22]{N}. If $b_T$ is the
Fong--Reynolds block of $T$ covering $b_D$ corresponding to $b_N$, then all
irreducible characters in $b_T$ are $\sigma$-invariant, by the uniqueness in
the Fong--Reynolds correspondence. 
Notice that $b_T$ is the only block of $T$ that covers $b_D$, by
\cite[Cor.~9.21]{N}. Finally, let $\lambda\in\irr D$ be with $o(\lambda)=|D|$,
and consider the irreducible character $\theta_\lambda \in \irr{b_D}$,
constructed in \cite[Thm~9.12]{N}. Let $\eta \in \irr{T|\theta_\lambda}$, which
necessary belongs to $b_T$ and is therefore $\sigma$-invariant. 
Since $T/\cent GD$ is a $p'$-group and $o(\sigma)$ is a power of~$p$, we have
that some $T$-conjugate of $\theta_\lambda$ is $\sigma$-invariant, by a
counting argument. Since $\theta$ is $T$-invariant, then
$(\theta_\lambda)^t=\theta_{\lambda^t}$ for $t\in T$.
Hence, we deduce that $\mu=\lambda^t$ is $\sigma$-invariant for some $t\in T$.
Since $o(\mu)=|D|$, we have that $\mu$ is faithful.
Therefore, $\mu(d^\sigma)=\mu(d)$ for all $d\in D$ implies that $d^\sigma=d$
for all $d\in D$. \qed
\medskip

We conclude the paper with two remarks. First, the assumption $p>2$ is crucial
for our approach: 
the conclusions of Proposition~\ref{str} and Corollary~\ref{str2}, which play a
key role at various steps of the proof of Theorem~A, do not hold when $p=2$.
\par
Secondly, we would like to point out that Brauer's Height Zero Conjecture
implies its so-called \emph{projective} version, as shown in \cite{Sa2}, as well
as the version for $\theta$-blocks in \cite{R}.

%%%%%%%%%%%%%%%%%%%%%%%%%%%%%%%%%%%%%%%%%%%%%%%%%%%%%%%%%%%%%%%%%%%%%%%%%

\end{document}